%% file: main.tex
\documentclass[10pt]{article}

\input{header}
\usepackage{macros}

\title{Estimation and Recovery of a Planted Dense Subgraph from a Single Network Cascade}

\author{
Maximilien Dreveton\textsuperscript{1}, 
Paula Mürmann\textsuperscript{2}, 
Patrick Thiran\textsuperscript{2} 
}
\date{}

\newcommand{\Ik}{\mathcal{I}_k}
\newcommand{\Ikc}{\mathcal{I}_k^c}
\newcommand{\Ck}{\mathcal{C}_k}
\newcommand{\tth}{^{\text{th}}}
\newcommand{\gnp}{G_{n,{p_n}}}
\newcommand{\gnpx}{G(n,p_n,\alpha,\xi)}

\newcommand{\stleq}{\preceq_{\text{st}}}

\newcommand{\prujF}{\pr\left(u_{K_1+j} \in V_1 \cond \cF_{K_1 +j-1}\right)}

\newcommand{\Cut}{\operatorname{Cut}}
\newcommand{\bigO}{\mathcal{O}}

\begin{document}

\maketitle

\begin{center}
\begin{tabular}{c}
\textsuperscript{1} LAMA, UMR-CNRS 8050, Université Gustave Eiﬀel \\ \texttt{maximilien.dreveton@univ-eiffel.fr} \\ 
\textsuperscript{2} École Polytechnique Fédérale de Lausanne (EPFL) 
\end{tabular}
\texttt{\{paula.murmann,patrick.thiran\}@epfl.ch}
\end{center}

\begin{center}
\today
\end{center}

\begin{abstract}
  \input{abstract}
\end{abstract}

\input{text}

\bibliographystyle{chicago}
\bibliography{biblio.bib}

\appendix
\input{appendix}

\end{document}

%% file: abstract.tex
We study the inference of a planted dense component in a sparse random graph from a single spreading process. The graph has an \Erdos--\Renyi background with edge probability $p_n$ and contains a planted dense component of size $\lfloor n^\alpha\rfloor$, with $\alpha>1/2$, whose internal edge density $\xi>0$ is constant. The edge set is unobserved; the data only consist of the successive infection times from a single realization of a continuous-time SI process with independent and exponentially distributed transmission times. We show that the planted dense component leaves a detectable signature in the spreading process: after the exploration enters the dense component, it undergoes a short phase of accelerated growth. By analyzing this phase, we localize its onset and endpoint. These localization results yield consistent estimators of the component-size exponent $\alpha$, the background edge density~$p_n$, and the internal edge density $\xi$ from the infection times alone. When the identities of the infected vertices are also observed, we further establish consistent recovery of the planted dense component. 

%% file: text.tex
\section{Introduction}

Graphs and stochastic processes on graphs are fundamental models for complex systems in which local interactions give rise to macroscopic dynamics. A canonical example is first-passage percolation, in which an exploration process propagates through a graph by traversing the shortest weighted paths. 
Such exploration processes arise in several settings: in epidemics, vertices represent individuals, edges represent contacts, and edge weights represent transmission times; in communication or transportation networks, weights may represent delays or travel times; and in information diffusion, weights may represent the time needed for a signal, a rumor, or an innovation to pass between connected agents. 
%Depending on the application, vertices may represent individuals, routers, locations, or agents, while edge weights encode transmission times, communication delays, travel times, or the time needed for information to spread across a connection.
%In an epidemic interpretation, vertices represent individuals, edges represent contacts, and edge weights represent transmission times. 

In many applications, however, the underlying graph is not observed, as knowing who was in contact with whom typically requires expensive and error-prone contact tracing~\citep{stargel_case_2022}. 
This motivates the problem of recovering the structure of the graph directly from the data generated by the exploration times of a spreading process. 
More precisely, we call \textit{network cascade} the observation model in which one observes the exploration timestamps, and possibly the identities of the explored vertices, but \textit{not} the edge set of the graph nor the edge through which each vertex was infected. Previous work has shown that, under appropriate conditions, the edge set of a bounded-degree graph can be reconstructed from $\Omega(\log n)$ independent network cascades on the same underlying graph with~$n$ vertices~\citep{gomez_rodriguez_inferring_2010, netrapalli_learning_2012,
abrahao_trace_2013, gray_bayesian_2020, wilinski_prediction-centric_2021}.

Recovering the full graph is, however, a challenging task, as it entails identifying every edge.
More recent work has considered the inference of coarser structural features from cascade data.
In particular,~\cite{mossel_high_deg_2024,brandenberger_detecting_change_points_2025} study the detection of a finite number of high-degree vertices, or ``super-spreaders'', with degrees of order $n^\alpha$. 
Their guarantees hold uniformly over a broad class of admissible underlying graphs with a finite number of super-spreaders; they prove that detecting the existence of a super-spreader is possible for $\alpha>1/2$ and impossible (in a worst-case sense) for $\alpha<1/2$, but recovering it requires at least $1/(2\alpha-1)$ independent cascades. 
%In particular, Assuming the observation of several independent cascades, detecting the existence of a super-spreader is possible for $\alpha>1/2$ and impossible for $\alpha<1/2$, but recovering it requires to observe a number of independent cascades that diverges as $\alpha$ approaches $1/2$.
%Assuming the observation of a finite number of independent cascades, they showed that such detection is possible for $\alpha>1/2$, and impossible if $\alpha<1/2$, and the number of cascades required increases as $\alpha$ approaches $1/2$. 

Observing several independent cascades on the same graph is a strong requirement: In an epidemic setting, for instance, this would amount to observing several independent outbreaks spreading through the same contact network. It is therefore natural to ask what structural information can be inferred from a \textit{single} cascade.

%Nevertheless, observing several independent cascades on the same graph is a strong requirement.  In an epidemic setting, for instance, this would amount to observing several independent outbreaks spreading through the same contact network. Moreover, recovering the full graph is also a very strong requirement. It is therefore natural to ask what graph properties can be inferred from a single realization. 

In this paper, we study this question for a sparse random graph containing a planted dense component. The graph has $n$ vertices, a sparse background edge probability $p_n = o(1)$, and a planted dense component of size $\lfloor n^\alpha\rfloor$ inside which edges appear with a fixed probability~$\xi$. Each edge is weighted with an independent exponential random variable of rate~$1$. We observe one network cascade, i.e. the shortest path distances from a random source to all other vertices, on this graph. The central question is whether the planted dense component leaves a detectable signature in the cascade, even though the graph itself is not observed. 

The key object in our analysis is the cut process seen by the exploration. If $\cI_k$ denotes the set of the $k$ first explored vertices, we define the cut $\Cut(\cI_k)$ as the set of edges from the explored set to its complement. Then the next inter-arrival is, conditionally on the graph and current exploration state, distributed as $\Exp( | \Cut(\cI_k) | )$. Thus, the cascade does not reveal the graph directly, but it samples the evolution of $C_k$ through exponential waiting times. In a homogeneous sparse graph, this cut size is approximately $p_n k(n-k)$. Once the exploration enters the planted dense component, an additional contribution appears, of strength $(\xi-p_n) I_1(k)\bigl(|V_1|-I_1(k)\bigr)$,
where $I_1(k) := |V_1 \cap \cI_k |$ is the number of explored vertices in~$V_1$ among the $k$ first explored vertices. 
This additional contribution creates a short high-rate phase in the exploration, which lasts until $V_1$ has been fully explored. The main probabilistic contribution of the paper is a detailed analysis of this phase: we show that the exploration rapidly enters and then almost exhausts the dense component, and we quantify the location, duration, and magnitude of the resulting acceleration in the cut process.

As statistical consequences of this probabilistic analysis, we construct estimators from a single cascade. From the ordered exploration times, we consistently estimate the size exponent $\alpha$, the background density $p_n$, and the dense-component density $\xi$. If vertex labels are also observed, then we can also consistently recover the planted dense component. 

This problem is related to, but different from, the classic planted clique and planted dense subgraph detection and recovery problems. In that literature, the graph is observed and the task is to detect or recover a denser-than-background vertex set. For planted cliques in $G(n,1/2)$, detection is impossible below the size of the
largest naturally occurring clique \citep{matula_largest_1976,grimmett_colouring_1975}, while polynomial-time algorithms are known above the $\sqrt n$ scale \citep{kucera_expected_1995, alon_finding_1998}, with a conjectured statistical-computational gap; see, for instance, \citep{feige_finding_2000, feige_finding_2010,dekel_finding_2014,deshpande_finding_2015,gamarnik_landscape_2024}. The planted dense subgraph setting, where the planted component need not be a clique, has also been
studied extensively~\citep{arias-castro_community_2014,verzelen_community_2015,hajek_information_2016, montanari_finding_2015,hajek_computational_2015,brennan_reducibility_2018,brennan_universality_2019}.  In contrast, our setting does not reveal the edge set. The only available information is the sample path of a stochastic process running on the graph. This changes the nature of the problem: the planted component must be inferred from the way it accelerates the exploration, rather than from direct inspection of the adjacency matrix of the graph.

%Our work is also connected to recent results on structural inference from a finite number of cascades.  In particular, \cite{mossel_high_deg_2024} and~\cite{brandenberger_detecting_change_points_2025} show that a finite number of high-degree vertices, or ``super-spreaders'', having degrees $n^{\alpha}$ with $\alpha > 1/2$ can be detected from cascade data. These works address arbitrary graphs in which the relevant signal is vertex-level degree heterogeneity, and the number of cascades required grows unbounded as the degree exponent~$\alpha$ approaches the critical value $1/2$. In this work, we instead study an average-case random graph model with a planted dense component and observe a single cascade. The dense set is not made of isolated high-degree vertices but is a small subgraph whose collective internal connectivity creates a short high-rate phase in the exploration process. 
Although the related works~\citep{mossel_high_deg_2024,brandenberger_detecting_change_points_2025} also infer coarse structural information from cascade data, our work differs significantly. Indeed, in \cite{mossel_high_deg_2024,brandenberger_detecting_change_points_2025}, a finite number of well-separated high-degree vertices creates vertex-level degree heterogeneity, whereas in our model the signal is collective and arises from the internal connectivity of a planted dense subgraph. We consider inference from a single cascade in an average-case random graph model, whereas these works adopt a worst-case perspective over a broad class of underlying graphs and require several cascades for recovery. 

Moreover, community detection from network cascades has also been studied by~\cite{prokhorenkova2022less}, which develop and compare methods based on likelihood maximization and on clustering surrogate graphs constructed from infection timestamps, with the aim of inferring a community partition from collections of cascades, but without providing theoretical guarantees. %Our focus is on consistency guarantees for a single planted dense component inferred from one cascade, with parameter estimation requiring only the unlabeled exploration times. 

Finally, our work contributes to the study of first-passage percolation on random graphs. Distances and exploration processes with exponential edge weights have been extensively studied on complete graphs and \Erdos-\Renyi graphs, often through branching-process approximations~\citep{janson_one_1999,bhamidi_first_2010,barbour_approximating_2013,bhamidi_front_2014}. 
Our regime and objective are different. We work with a sparse but sufficiently connected background graph and a planted inhomogeneity, and we analyze the exploration not only to understand its metric behavior, but also to infer latent structure from one realization of the process. 

The rest of the paper is organized as follows. Section~\ref{sec:setup} introduces the graph and cascade models, the observation models, and the main results. Section~\ref{sec:exploration_curve} analyzes the cut process and the growth of the exploration inside the dense component. This analysis is then used in Sections~\ref{sec:detection} and~\ref{sec:estimation} to establish the consistency of the dense-component phase detection and of the model-parameter estimators.
Section~\ref{sec:numerical_experiments} presents numerical experiments. 
Finally, Section~\ref{sec:exploration_curve_proofs} contains the main proofs, while more technical estimates are deferred to the Appendix.

\paragraph*{Notations}
 We use standard asymptotic notation. For two sequences $a_n,b_n$, we write $a_n\ll b_n$ if $a_n/b_n\to0$; $a_n\gg b_n$ if $b_n/a_n\to0$; and $a_n\sim b_n$ if $a_n/b_n\to1$. We additionally write $b_n = o(a_n)$ to denote any sequence $b_n\ll a_n$; $\omega(a_n)$ to denote any sequence $b_n\gg a_n$; $b_n=\bigO(a_n)$ if $\lim\sup_{n\to\infty}|b_n/a_n|<\infty$; $b_n=\Omega(a_n)$ if $\lim\inf_{n\to\infty}|b_n/a_n|>0$; and $b_n=\Theta(a_n)$ if $a_n=\bigO(b_n)$ and $b_n=\bigO(a_n)$.

We say that a sequence of events $(A_n)_{n \ge 1}$ occurs \emph{with high probability} (or \emph{w.h.p.} in short) if $\lim_n \pr(A_n) = 1$. For a sequence of random variables $(X_n)_{n \in \Z_+}$ and a deterministic sequence $(z_n)_{n \in \Z_+}$, we write $X_n = \bigO_{\pr}(z_n)$ if $X_n / z_n$ is a tight sequence of random variables, and $X_n = o_{\pr}(z_n)$ if $X_n / z_n$ converges to $0$ in probability. 
%We say that a sequence of random variables $X_n$ converges \emph{in probability} to $X$ (denoted $X_n\wprto X$) if for any $\varepsilon>0$, $|X_n-X|<\varepsilon$ holds with high probability. We say that a sequence of random variables $X_n$ converges \emph{in distribution} to $X$ (denoted $X_n\stackrel d{\to}X$) if $\pr(X_n\le x)\to\P(X\le x)$ for any $x$ where $\pr(X\le x)$ is continuous. 
We say that a random variable \(X\) is stochastically dominated by \(Y\), and write \(X\stleq Y\), if $\pr(X\ge t)\le \pr(Y\ge t)$ for every $t\in\R$. 

For a graph $G=(V,E)$ and a subset $V' \subseteq V$, we denote by $G[V']$ the subgraph of $G$ induced by $V'$, consisting of the vertices in $V'$ and all edges of $G$ whose endpoints both belong to $V'$. Finally, for a set \(V\), we write \(\binom{V}{2}\) for the set of unordered pairs of distinct elements of \(V\), that is $\binom{V}{2} := \bigl\{\{u,v\}: u,v\in V,\ u\neq v\bigr\}$.

\section{Definitions and Main Results}
\label{sec:setup}

%\paragraph{Notation} For a set $S$, we let $|S|$ denote its cardinality. For a graph $G = (V, E)$, we denote the number of vertices by $n = |V|$ and let $N(v)$ denote the neighborhood of a vertex $v$. For a subset of vertices $S \subseteq V$ , we denote $\Cut(S) := |\{(u, v) \in E \colon u \in S, v \not \in S\}|$. 

\subsection{Graph and cascade model} 
\label{subsection:model}

Let $G=(V,E)$ be an undirected graph on $n$ vertices generated as follows.
The vertex set is partitioned as $V = V_0 \cup V_1$, with $|V_1| = \lfloor n^{\alpha} \rfloor =: n_1$ and $|V_0|=n-\lfloor n^{\alpha} \rfloor$, where $\alpha \in (0,1)$ is unknown and independent of $n$. Edges are generated randomly and independently according to:
\begin{align} \label{eq:edge_prob}
 \pr \left( \{u,v\} \in E \cond V_0, V_1\right) \weq 
 \begin{cases}
      \xi & \text{ if } u \in V_1 \text{ and } v \in V_1 \\
      p_n & \text{ otherwise, }
 \end{cases}
\end{align}
where $\xi \in (0,1]$ is constant and $p_n \in (0,1)$ depends on $n$. 

Throughout, we always assume that $\xi \ge p_n$, and we call $V_1$ the dense component and~$V_0$ the sparse background. We will always assume that $\lim n p_n = \infty$, ensuring that almost all vertices belong to a (unique) giant connected component. We denote this random graph model by $G(n,p_n,\alpha,\xi)$. Note that this is equivalent to a stochastic block model with blocks $V_0$ and $V_1$, edge probability $\xi$ between vertices in $V_1$ and probability $p_n$ otherwise.

We consider a continuous-time first-passage  (FPP) process on $G$, which is equivalent to a susceptible-infected (SI) model with independent exponential transmission times along edges. The process starts at time $t_1=0$ with an explored vertex $u_1$ chosen uniformly at random in~$V$. Each edge joining an explored vertex to an unexplored vertex transmits independently at rate $\lambda$,\footnote{We assume $\lambda$ is known and rescale time so that $\lambda =1$.} and vertices remain explored once reached, and vertices remain explored once reached. 
For a vertex $v\in V$, let
\[
T(v) \, := \, \inf\{ t\ge 0 \colon v \text{ has been explored by time } t\}
\]
denote its exploration time, with the convention $T(v)=\infty$ if $v$ is never reached from the initial vertex $u_1$ (this happens if and only if $v$ does not belong to the same connected component as the initial vertex). We write
\[
 0 \weq t_1 \wle t_2 \wle \cdots \wle t_n
\]
for the order statistics of the collection $\{T(v) \colon v\in V\}$, and we let $u_k$ denote the vertex explored at time $t_k$, so that $T(u_k)=t_k$. When $G$ is connected, the exploration times are almost surely distinct, and hence
$
0 = t_1 < t_2 < \cdots<t_n .
$

We consider two observation models. The first one consists only of the exploration times, whereas the second one also records the identities of the explored vertices. The two models are denoted $\cD_{\mathrm{time}}$ and $\cD_{\mathrm{id}}$, and we have 
\[
\mathcal D_{\mathrm{time}} := (t_1,\ldots,t_n) 
\quad \text{ and } \quad 
\cD_{\mathrm{id}} := \left(( (u_1,t_1), \cdots, (u_n,t_n) \right). 
\]

\subsection{Main results} 
Define $\cI_k := \{ u_1, \cdots, u_k\}$, the set of the first $k$ explored vertices, and 
\begin{align*}
 K_1:=\min\{k \colon u_k\in V_1\} 
 \quad \text{ and } \quad 
 K_2 := 1+\min\{k \colon V_1 \subseteq \cI_k\}. 
 %K_2^{\star}:=K_1 + \lfloor n^{\alpha} \rfloor.
\end{align*}
The quantity $K_1$ (resp., $K_2$) is the first step at which the exploration visits a vertex in the dense component (resp., the first exploration index after all vertices of $V_1$ have been explored). Thus $[K_1,K_2)$ is the interval between the first visit to the dense component and its complete exploration. 
For convenience, we also define the reference endpoint
\[
    K_2^\star:=K_1+n_1.
\]
If all vertices of $V_1$ were explored consecutively after the first visit, then $K_2=K_2^\star$; although this is not the case, Proposition~\ref{prop:I1_in_V1_bd} ensures that $K_2-K_2^\star=O_{\mathbb P}(np_n\log n)$.
\begin{comment}
We also introduce the reference endpoint
\[
    K_2^\star:=K_1+n_1.
\]
If all vertices of $V_1$ were explored consecutively after the first visit, then $K_2=K_2^\star$.  In the actual process, occasional explorations in $V_0$ may occur before $V_1$ is exhausted, so that $K_2\ge K_2^\star$.  Proposition~\ref{prop:I1_in_V1_bd} shows that $K_2-K_2^\star=\cO_{\pr}(np\log n).$
The estimator $\widehat K_2$ is designed to localize the reference endpoint $K_2^\star$, rather than the random completion index $K_2$, because near the end of the dense phase the additional cut induced by the remaining unexplored vertices of $V_1$ becomes comparable to the background cut.

---- 

The index~$K_2^{\star}$ should be interpreted as follows: if, after first entering $V_1$, the exploration visited all vertices of $V_1$ consecutively, then the dense-component phase would occupy the interval $[K_1,K_2^{\star})$. In the actual process, occasional activations outside $V_1$ may occur during this phase, and a small number of vertices of $V_1$ may be explored after step $K_2^{\star}$. Thus, $K_2^{\star}$ is \textit{not} defined as the last activation step of a vertex in $V_1$, as estimating it would be noisy; rather, it is defined to identify the size of $V_1$ through the relation $K_2^{\star}-K_1 = |V_1|$. \marginnote{Do we have $K_2^\star$ or $K_2$? Later we use $K_2$ I think}
\end{comment}

Our first main result shows that the latent dense-phase indices and the model parameters can be consistently estimated from exploration times $\cD_{\mathrm{time}}$.

\begin{theorem}[Localization and estimation from the exploration times]
\label{thm:estimation}
Suppose that the graph is generated according to $G(n,p_n,\alpha,\xi)$, with $1 \ll np_n \ll n^{2\alpha-1} / \log n $, $\alpha\in(1/2,1)$, and $\xi\in(0,1]$. 
Then, there exist estimators $\hK_1$, $\hK_2$, $\halpha$, $\hp$ and $\hxi$, computed from the exploration times $\cD_{\mathrm{time}}$ of a single continuous-time FPP process with known rate $\lambda=1$, such that the following statements hold. 

\begin{enumerate}
 \item The dense-phase change-points are localized as %\marginnote{Updated $8.1$ to $4.1$.}
 \[
 \pr\left( \left| \hK_1 - K_1 \right| \wle 4.1 \log n \right) = 1-o(1)
 \quad \text{ and } \quad 
 \left|\hK_2 - K_2^{\star} \right| \weq \cO_{\pr}\left( n p_n \log n \right).
 \] 
 \item The model parameters are consistently estimated:
 \[
 \frac{|\widehat p-p_n|}{p_n} \weq \cO_{\pr} \left( \frac{1}{\sqrt{ n p_n } } \right),
  \quad 
 |\widehat\xi-\xi| = o_{\pr}(1), 
 \quad 
 |\widehat\alpha-\alpha|
 \weq \cO_{\pr} \left( n^{1-\alpha} p_n \right). %\cO\left( \frac{\log^{3/2}n+p_n n\log n}{n^\alpha\log n} \right). 
 \]
\end{enumerate}
\end{theorem}

The condition $np_n\gg 1$ ensures the concentration of the estimators, and the condition $np_n \ll \frac{n^{2\alpha-1}}{\log n}$ ensures that the sparse background does not absorb the acceleration caused by the dense component once the cascade enters $V_1$. Together, these conditions define a broad sparse regime. In particular, for any fixed $\alpha>1/2$, they are satisfied whenever $np_n=(\log n)^c$ for any constant $c>0$. 
Finally, for ease of exposition, the theorem and its proof are written for fixed \(\xi>0\), but we expect the argument to extend to sequences \(\xi_n= o(1)\), provided the excess density \(q_n=\xi_n-p_n\) satisfies suitable separability conditions. In that regime the natural conclusion would be relative consistency, \(|\widehat\xi-\xi_n|/\xi_n= o_{\pr}(1)\). 

 The condition $\alpha>1/2$ ensures that the dense component produces a sufficiently pronounced and sustained acceleration of the exploration process, so that it can be localized from a single cascade. 
 Although the $\sqrt n$ scale also appears in several related graph-inference problems, such as the planted clique problem and high-degree-vertex inference from several cascades, these settings involve different observation models or graph assumptions, as discussed in the introduction. Whether $\alpha>1/2$ is a fundamental threshold for inference in our model is an interesting open question for future work.

Our second main result concerns recovery of the vertex set $V_1$. Unlike the parameter estimators above, recovering $V_1$ requires observing the identities of the explored vertices. Theorem~\ref{thm:weak_recovery} controls the number of mistakes made by simply returning all the vertices explored between $\hat{K}_1$ and $\hat{K}_2$. 
%The simplest algorithm for recovering~$V_1$ is to return all the vertices explored between $\hat{K}_1$ and $\hat{K}_2$. Theorem~\ref{thm:weak_recovery} shows that this approach is consistent. % (as in \cite{hajek_information_2016}).

\begin{theorem}[Recovery of the dense component]
\label{thm:weak_recovery}
Assume the conditions and notations of Theorem~\ref{thm:estimation}. Suppose in addition that the identities of the explored vertices $\cD_{\mathrm{id}}$ are observed, and define $\widehat V_1 := \{u_k \colon \hK_1 \wle k \ < \ \hK_2\}$. Then,
\[
\left| \widehat V_1 \cap V_0  \right| \weq \bigO_{\pr} \left( n p_n \log n \right)
\quad \text{ and } \quad 
\left| \widehat V_1^c \cap V_1 \right| \weq \bigO_{\pr} \left( n p_n \log n \right).
\]
%\[\frac{ \left| V_1\Delta \widehat V_1 \right|}{|V_1|} \weq \cO\left( n^{1-\alpha} p_n \log n \right).\]
%where $\Delta$ denotes symmetric difference between two sets.
\end{theorem}

In the theorem, the quantity $\left| \widehat V_1 \cap V_0  \right|$ describes the number of false positives, i.e., vertices outside the planted component that are incorrectly included in $\widehat V_1$. Similarly, $\left| \widehat V_1^c \cap V_1 \right| $ is the number of false negatives, i.e., vertices in the planted component that are missed by the estimator. Both errors are controlled, as under our assumptions, $np_n \log n \ll n^{2\alpha-1} \ll |V_1|$. 

\subsection{Overview of the Proof of Theorem~\ref{thm:estimation}} \label{sec:proof_overview}

For $k\geq 1$, let $\Ik:=\{u_1,\ldots,u_k\}$ and $\Ikc:=V\setminus\Ik$ denote respectively the explored and unexplored set of vertices after $k$ explorations, and let $C_k:=|\Cut(\Ik)|$ be the number of edges between $\Ik$ and $\Ikc$. 
The natural filtration of the exploration process is $\cF_k^{\mathrm{nat}} :=\sigma\bigl((u_i,t_i)_{1\leq i\leq k}\bigr).$ For the probabilistic analysis, we use instead the filtration $\cF_k:=\sigma\bigl(G,V_1,(u_i,t_i)_{1\leq i\leq k}\bigr),$ where $G$ denotes the unweighted underlying graph. Thus, $\Ik$ and $C_k$ are $\cF_k$-measurable, whereas the next increment $a_k:=t_{k+1}-t_k$ is not. We view $G$ as the random environment of the exploration process; unless otherwise specified, all probability statements are annealed, i.e., they average over both the graph and the exploration randomness.

The proof of Theorem~\ref{thm:estimation} proceeds in two steps.

\subsubsection{Characterization of the exploration rate in a pure \texorpdfstring{$\gnp$}{gnp} graph}

%The increments \(a_k:=t_{k+1}-t_k\) are well-defined only if \(t_{k+1}<\infty\). Under our assumptions \(np_n\to\infty\), the background graph has a unique giant component of size \(n-o(n)\), and the randomly chosen source vertex lies in it with probability \(1-o(1)\). When the source belongs to the giant component, all the increments for \(k\le n/2\) are finite with high probability. In the following, we will restrict $k \le n/2$ and implicitly work conditioning on the event that the source lies in the giant component. 
The inter-exploration times $a_k:=t_{k+1}-t_k$ are well-defined if $t_{k+1}<\infty$. As $np_n\to\infty$ by assumption, the background graph has a unique giant component of size $n-o(n)$, and the randomly chosen source vertex lies in it with probability $1-o(1)$. On this event, all increments $a_k$ with $k\leq n/2$ are finite with high probability, and throughout the analysis we restrict attention to $k\leq n/2$. By convention, we set $C_k = 0$ after the exploration stops.

By the memorylessness property of the exponential edge weights, conditionally on $\cF_k$ and on $C_k>0$, we have  
\[ 
a_k\mid\cF_k\sim\Exp(C_k). 
\] 

Describing the exploration times therefore reduces to analyzing the sequence of cut-sizes $(C_1,\ldots,C_{n-1})$. 
A natural, but incorrect, guess would be $C_k \sim \Bin(k(n-k),p)$, as there are $k(n-k)$ potential edges across the cut from $\cI_k$ to $\cI_k^c$. This holds for the complete graph~\citet{janson_one_1999}, where the cut size is deterministically $k(n-k)$, but not for a $\gnp$ graph with $p<1$, because the next infected vertex is not chosen uniformly from $\Ik^c$ but by sampling an edge uniformly at random from the cut and following it to its endpoint. Unexplored vertices with more neighbors in the explored set $\cI_k$ are thus more likely to be explored next, inducing a size bias that distorts the cut-size distribution across steps. 
However, we establish in Proposition~\ref{prop:homogeneous-cut-concentration} a uniform concentration of the cut-sizes $C_k$ around $k(n-k)p_n$. 

%To handle this dependency, we show that, with high probability, $C_k$ is sandwiched between two sums of independent binomials. For the upper bound, we compare the true exploration process to an auxiliary one in which the next vertex is chosen uniformly at random from $\Ikc$. This removes the size bias and yields, via stochastic domination, $C_k \stleq \Bin(k(n-k),p_n)$, which has expectation $p_nk(n-k)$. For the lower bound, we track the number of edges from each newly infected vertex into $\Ik$. A monotonicity argument, implemented through an FKG-type inequality, gives a matching stochastic lower bound with mean $pk(n-k)-(1-p)(k-1)$. Combined with Bernstein-type  inequalities, this shows that $C_k$ is tightly bounded with high probability. 

 \subsubsection{Exploration-rate deviation and change-point detection}
\label{subsubsec:change_point_detection}

We now explain how the perturbation created by the dense component can be detected from the exploration times. 

Before the exploration reaches \(V_1\), the process evolves as in an \Erdos-\Renyi graph $\gnp$. Thus, the rate of the \(k\)-th exploration is approximately \(m_k := p k(n-k).\) 
Once the first vertex of \(V_1\) is activated, however, the internal edges of the dense component create a much larger contribution to the active boundary, and the waiting times \(a_k\) between two subsequent explorations become much shorter than those predicted by the pure $\gnp$ baseline.

This produces a localized upward deviation in the exploration rate. The phenomenon is sharp at the beginning of the dense-component phase: once the cascade enters \(V_1\), the cut into the remaining vertices of \(V_1\) dominates the cut into the sparse background, and subsequent activations are overwhelmingly likely to remain inside \(V_1\). Thus, the exploration spends a short but macroscopically visible interval rapidly infecting the dense component. 

 To make this deviation visible, we transform each inter-exploration time $a_k$ using the pre-change background rate. More precisely, given $a_k$, we define 
\begin{align}
 \label{def:probability_integral_transform}
 Z_k := \sqrt{12}\left(e^{-m_k a_k}-\frac12\right),
\end{align}
where $m_k = pk (n-k)$.  
To motivate this choice, observe that, if $X\sim \Exp(m)$, then \(e^{-mX} \sim \Uni(0,1)\) and therefore \( \sqrt{12}\left(e^{-m X}-\frac12\right) \sim \Uni(-\sqrt 3,\sqrt 3)\) has zero mean and unit variance.
Hence, under an idealized baseline where the rate of $a_k$ is $m_k$, the random variables $Z_k$ would be centered, bounded, and have unit variance. After the process enters~$V_1$, however, the times between subsequent explorations become smaller, and the random variables $Z_k$ acquire a positive drift. 

 However, because $m_k$ involves the unknown quantity $p_n$, the quantity $Z_k$ in Equation~\eqref{def:probability_integral_transform} cannot be computed; we instead use $\hm_k = \hp_n k(n-k)$ to define a feasible transform $\widehat Z_k$, where $\hp_n$ is the estimate of $p_n$ defined as follows. 
Observe that the dense component is fully explored very early in the cascade: indeed, a.a.s., all vertices of $V_1$ are explored within the first $o(n)$ steps, \textit{i.e.,} $K_2 = o_{\pr}(n)$. After this, the dense-component contribution to the exploration disappears, and the cut size returns to the background scale. 
We therefore restrict attention to a window of explorations lying well after $K_2$, for example between $n/3$ and $n/2$. 
On this window, our analysis of the cut-size process implies that $C_k = p_n k(n-k)\bigl(1+o_{\pr}(1)\bigr)$ uniformly in $k$. 
As $a_k\mid\mathcal F_k\sim\Exp(C_k)$, the rescaled waiting times $k(n-k)a_k$ therefore have rates asymptotically close to~$p_n$, leading to the following feasible maximum likelihood estimator: 
 \begin{align}
  \label{eq:estimate_p}
   \hp_n \ := \ \frac{ \lfloor \frac{n}{2}\rfloor - \lfloor \frac{n}{3}\rfloor +1 }{\sum_{k = \lfloor n/3 \rfloor}^{ \lfloor n/2 \rfloor } k(n-k)\,a_k},
   %\qquad \cI_p \ := \ \bigl\{k \colon \tfrac{n}{3} \le k \le \tfrac{n}{2}\bigr\}.
 \end{align}

Using $\hm_k$ instead of $m_k$ in the formula~\eqref{def:probability_integral_transform} yields the feasible variables $\widehat Z_k$. We average these variables over bins of logarithmic size, to obtain the bin averages $\overline Z_j$. These bin averages concentrate near zero before the first exploration of $V_1$ and rise as the exploration accelerates inside the dense component. 
Hence, we define $\widehat K_1$ as the left endpoint of the first bin whose average exceeds a fixed threshold $H_1 \in (0,\sqrt{3})$, localizing $K_1$ to within one bin. Because bins have logarithmic sizes, we obtain $|\widehat K_1 - K_1 | = \cO_{\pr}(\log n)$, and the constant $4.1$ in Theorem~\ref{thm:estimation} comes from a finer choice of the constant $H_1$ and bin size. Moreover, as $K_1$ occurs early, we can in fact restrict this search to the first $\kappa_1 = \Theta(\sqrt{n})$ samples to control false positives. 
When $V_1$ is fully explored, the dense contribution to the cut fades and the bin averages return toward the background level. After a minimum separation from $\widehat K_1$, we define $\widehat K_2$ as the left endpoint of the first bin whose average falls below a second threshold. Although the endpoint $K_2$ is localized less precisely than the onset, these estimates consistently recover the dense-phase length and hence the exponent $\alpha$.

The estimation of $\xi$ is a bit more involved and is detailed in Section~\ref{subsec:estimating_xi}. 

\begin{algorithm}[!ht]
\caption{ Estimation of \(K_1\), \(K_2\), $p_n$, and \(\alpha\)}
\label{algo:k1k2_estimation}

\KwInput{
activation differences \(\{a_k\}_{k=1}^{\lceil n/2\rceil}\); 
%background rates \(\{m_k\}_{k=1}^{\kappa}\);
$K_1$-cutoff $\kappa_1$;
bin size \(w\);
thresholds \(H_1,H_2>0\);
minimum separation \(s_n\);
graph size \(n\)
}

\KwOutput{
background density $\widehat p_n$;
estimated change points \(\widehat K_1,\widehat K_2\);
exponent \(\widehat\alpha\)
}

\BlankLine
Compute $\widehat p_n$ as defined in~\eqref{eq:estimate_p}.

\For{$k=1,\ldots,\lceil n/2\rceil$}{
Compute the transformed increments:
$ Z_k \gets \sqrt{12}\left(e^{-\widehat m_k a_k}-\frac12\right),$ where $\widehat m_k = \widehat p_n k (n-k)$. 
}
Set
\(
 L \gets \left\lfloor \frac{n}{2w}\right\rfloor  \quad \text{and} \quad L_1 \gets \left\lfloor \frac{\kappa_1}{w}\right\rfloor
\)

\For{$j=1,\dots,L$}{
Compute the binned averages:
$\overline Z_j \gets\frac1w \sum\limits_{i=(j-1)w+1}^{jw} Z_i.$\label{eq:algo_binned_averages}
}
Find the first upward crossing:
$j_1^* \gets \min\left\{ j\in\{1,\dots,L_1\}:\overline Z_j>H_1 \right\}.$

\If{\(j_1^*\) does not exist}{
    \KwReturn{\(\widehat p_n, \texttt{none},\texttt{none},\texttt{none}\)}
}
Set $\widehat K_1 \gets (j_1^*-1)w$. 

Set the first bin from which the endpoint search begins:
$j_{\mathrm{start}} \gets j_1^*+\left\lceil \frac{s_n}{w}\right\rceil$.

Find the first downward crossing after $j_{\mathrm{start}}$: % the minimum separation:
\( j_2^* \gets \min\left\{ j\in\{j_{\mathrm{start}},\dots,L\} \colon \overline Z_j < H_2 \right\}. \)

\If{\(j_2^*\) does not exist}{
    \KwReturn{\(\widehat p_n, \widehat K_1,\texttt{none},\texttt{none}\)}
}

Set
$\widehat K_2 \gets (j_2^*-1)w$.

Estimate the dense-component size by
$\widehat n_1 \gets \widehat K_2-\widehat K_1$.

Set
$\widehat\alpha \gets \frac{\log \widehat n_1}{\log n}$.

\KwReturn{\( \widehat p_n, \widehat K_1,\widehat K_2,\widehat\alpha\)}

\end{algorithm}

\subsection{Use of Large Language Models}
In the preparation of this manuscript we used OpenAI's ChatGPT-5.5, 5.6 Sol, 6 Astra and Anthropic's Claude Opus 5.5 to edit and clarify language, develop proof ideas and generate code for the numerical experiments. Any results proposed by the large language model (LLM) were rigorously checked and rewritten, and the authors take full responsibility for the results presented hereafter. 

The ideation of the project and the approach to bound the exploration cut sizes were developed without LLM use. The lower bound on the cut size in Section~\ref{sec:exploration_curve} was developed with LLM support through many iterations until a valid approach was found. The upper bound was first developed without LLM usage, but LLM was later used to simplify the arguments. The estimation procedure was guided by exploratory numerical experiments, which helped the development of the theoretical arguments; LLMs were then used to refine these arguments and improve their clarity. 
The code for the experimental illustrations was first written with minimal LLM contribution, and later refactored and optimized by Claude Code Opus 5.5.

\section{Exploration Curve with and without a Dense Component}
\label{sec:exploration_curve}

 \subsection{Exploration Curve without a Dense Component}
\label{subsec:homogeneous-exploration}

Consider first the exploration on $G\sim \gnp$. 
For a deterministic set $S$ of size $k$, we have $|\Cut(S)| \sim \Bin(k(n-k),p_n)$. 
This does not hold for the stochastic exploration set, because the exploration favors vertices with more neighbors in the explored set. 
The following proposition shows that $C_k$ is concentrated around $k(n-k)p_n$, and we show that $C_k \stleq \Bin(k(n-k),p_n)$ in Appendix~\ref{subsec:stoc-upper-bound-cut}. 
 
%Nevertheless, the following stochastic upper bound holds.

\begin{proposition}%[Uniform cut concentration]
\label{prop:homogeneous-cut-concentration}
Consider the sequence $(C_1,\cdots,C_n)$ of cut-sizes of a continuous FPP process with rate $\lambda=1$ in a $\gnp$ graph. 
Assume that $n p_n\to\infty$. With probability $1-o(1)$, we have 
\[
 \max_{1\le k\le\lfloor n/2\rfloor} \left|\frac{C_k}{k(n-k)p_n}-1\right|
 \wle \sqrt{\frac{24\log(e \, n p_n)}{n p_n}}.
\]
\end{proposition}

\subsection{Exploration Curve with a Dense Component} 
\label{sec:I_1}

We now quantify how the planted dense component of the $\gnpx$ model perturbs the exploration process. The intuition is threefold. First, before the exploration reaches $V_1$, the process is indistinguishable from a pure $\gnp$ exploration. Then, once the first vertex of $V_1$ is visited, the number of edges into $V_1$ dominates the number of edges into $V_0$ such that most subsequent explorations fall within~$V_1$. Finally, once $V_1$ has been fully explored, the process resumes its $\gnp$-like behavior. 
%In effect, the dense component behaves like an independent $G_{n^\alpha,\xi-p}$ superimposed on the sparse background where the core quantity to control is the number of visited vertices inside~$V_1$.

The following lemma shows that the index $K_1$ of the first explored vertex in $V_1$ is stochastically dominated, up to a negligible term, by a geometric random variable with success probability $n_1 / n$; in particular, $K_1 \le n^{1-\alpha}\log n$ with high probability.

 \begin{lemma}
\label{lem:first-entry}
Suppose $\alpha\in(0,1)$ and $np_n\to\infty$. Set $K_1=\infty$ if the exploration never reaches $V_1$. Then
$\varepsilon_n:=\pr(K_1=\infty)=o(1)$, and, for every
$k\in\mathbb N$,
\[
 \pr(K_1>k)
 \wle \left(1-\frac{n_1}{n}\right)^k+\varepsilon_n.
\]
In particular, for every sequence $f_n\to\infty$, we have 
$\pr\bigl(K_1\le f_n n^{1-\alpha}\bigr)=1-o(1).$ 
\end{lemma}

\begin{proof}
First, we show that $\varepsilon_n=o(1)$. 
Because $\xi\ge p_n$, we may couple the planted graph with a graph $H\sim G(n,p_n)$ such that every edge of $H$ is also present in the planted graph. 
Fix a vertex $v\in V_1$.
As $np_n\to\infty$, the graph $H$ has a unique giant component of size $n-o(n)$ with high probability. 
By exchangeability, both $v$ and the uniformly chosen source belong to this component with probability $1-o(1)$.  
They are therefore connected in the planted graph with probability $1-o(1)$, which proves $\varepsilon_n=o(1)$.

Next, extend the exploration by restarting from a uniformly
chosen unexplored vertex whenever its cut becomes empty.
Let $\widetilde K_1$ be the first index at which this extended
exploration visits $V_1$.

Before the first visit to $V_1$, no edge internal to $V_1$
can affect the exploration. Hence, the extended
exploration up to this visit has the same law as the
corresponding exploration on a \Erdos-\Renyi $G(n,p_n)$ graph.
In that graph, exchangeability of vertex labels, together
with the uniform source and restart rules, implies that the
full exploration order is a uniformly random permutation. 
%This statement averages over the graph and the exploration randomness. 
It follows that, for $0\le k\le n-n_1$,
\[
    \pr(\widetilde K_1>k)
    =\prod_{i=0}^{k-1}
        \left(1-\frac{n_1}{n-i}\right)
    \le \left(1-\frac{n_1}{n}\right)^k.
\]
For $k>n-n_1$, the left-hand side is zero, so the inequality
remains valid.

On the event $\{K_1<\infty\}$, the original exploration reaches $V_1$ before any restart occurs. Hence
$\widetilde K_1=K_1$ on this event, and therefore
\[
 \pr(K_1>k) \wle \pr(\widetilde K_1>k)+\pr(K_1=\infty),
\]
which proves the claimed bound.

Finally, take $k=\lfloor f_n n^{1-\alpha}\rfloor$.
Because $n_1/n^\alpha\to1$, we have $kn_1/n\to\infty$.
Thus
\[
 \pr(K_1>k)
 \wle \exp\left(-\frac{kn_1}{n}\right)+\varepsilon_n
 \weq o(1).
\]
\end{proof}

Next, for $k\ge K_1$, let $u_k$ be the $k\tth$ explored vertex and let
\[
 I_1(k) \ := \ |V_1\cap \Ik| \weq 1 + \sum_{j=1}^{k-K_1} \1 \{u_{K_1+j}\in V_1\}
\]
be the number of explored vertices belonging to the dense component among the $k$ first explorations in the whole graph. For any $j \in [n-K_1]$, we have the deterministic bound 
\[
 I_1(K_1 - 1 + j) \wle \min\left\{ j, n^{\alpha} \right\}. 
\]
The following result, proved in Section~\ref{sec:proof_I1}, provides a matching lower bound on $I_1(k)$ and describes the exploration of the dense component on two scales: 
(i) During the first $\cO(\log n)$ explorations after $K_1$, almost all newly explored vertices belong to $V_1$; (ii) Over the full interval until $V_1$ is fully explored, explorations in $V_0$ may occur, but their total number is at most $\cO(np_n\log n)$. In particular, this implies that %$I_1(K_1+n_1+\lceil Cnp_n\log n\rceil)=n_1$ a.a.s., and thus
\[
 0 \wle K_2-K_2^\star \weq \cO_{\pr} \left( np_n\log n \right). 
\]

 \begin{proposition}[Dense-component exploration]
\label{prop:I1_in_V1_bd}
Suppose that $\alpha\in(1/2,1)$ and $np_n\gg 1$. 
The following statements hold conditionally on the event $\{K_1\le n^{1-\alpha}\log n\}$.
\begin{enumerate}
 \item[(i)] Assume additionally that $n p_n \log n=o(n^{2\alpha-1})$. For every deterministic sequence of positive integers $m_n=o(n_1/n p_n)$, with probability $1-o(1)$,
 \[
   I_1(K_1-1+j)=j  \qquad\text{simultaneously for }1\le j\le m_n.
 \]
 In particular, this holds for $m_n=\lfloor D\log n\rfloor$ for every fixed $D>0$.

 \item[(ii)] Assume additionally that $n p_n \log n=o(n^{\alpha})$. 
 There exist constants $c,C>0$ such that, with probability at least $1-n^{-c}$,
 \[
   I_1(K_1-1+j) \wge \min\{j-C n p_n \log n,n_1\}
 \]
 simultaneously for all integers $j\ge1$ such that $t_{K_1-1+j}<\infty$. 
 On the same event, $V_1$ is fully explored and we have $0\le K_2-K_2^\star\le C n p_n \log n$
\end{enumerate}
\end{proposition}

\begin{comment}
2026-10-02: OLD VERSION OF THE PROPOSITION
\begin{proposition}[Dense-component exploration]
\label{prop:I1_in_V1_bd} 
Suppose that $\alpha>1/2$ and $np_n\gg 1$, and condition on the event \(\{K_1\le n^{1-\alpha}\log n\}\).
Then the following hold. 
\begin{enumerate}[(i)]
\item %\emph{Early regime: Rapid exploration after entry.} 
Suppose $n p_n \ll n^{2\alpha-1} / \log n$. 
For every fixed $D>0$, there exists a constant $C>0$ such that, with probability $1-o(1)$, simultaneously for all $1\le j\le D\log n$, we have 
\[
 I_1(K_1-1+j)
 \wge j\bigl(1-Cp_n n^{2-2\alpha}\log n\bigr) - \sqrt{2D\log n\cdot\log\log n}.
\]
\item %\emph{Bulk and ending regime: Few background exploration before exhaustion of $V_1$.} 
Suppose that $n p_n \ll n^{\alpha} / \log n$. 
There exist constants $c,C>0$ such that, with probability at least $1-n^{-c}$, we have 
\[
 I_1(K_1-1+j)
 \wge \min\left\{ j-Cnp_n\log n,\, n_1 \right\}.
\]
simultaneously for all \(j\ge1\) such that $t_j < \infty$. 
%In particular, $I_1(K_1+n_1+\lceil Cnp_n\log n\rceil)=n_1$. 
\end{enumerate}
\end{proposition}
\end{comment}

The next proposition controls the cut during and after the exploration of the dense component. 

\begin{proposition}[Cut bounds during the dense phase]
\label{prop:G_npbx_cut_bound}
Suppose $\alpha \in (0,1)$ and $n p_n \gg 1$. 
Let $\eta>0$, $D_k:=I_1(k)(n_1-I_1(k))$, $M_k:=\min\{I_1(k),n_1-I_1(k)\}$, $U_k := kn p_n + \sqrt{6knp_n\log n}+\log n$, and
\[
    \Delta_k^{\mathrm{in}}
    :=
    \begin{cases}
    0, & \xi=1,\\[1mm]
    (1+\eta)
    \sqrt{2\xi (1-\xi) D_k M_k\log(en_1)},
    & \xi<1.
    \end{cases}
\]
Then there exists a constant $A_\eta<\infty$ such that, for $n$ large enough, we have simultaneously for all $1\le k\le n/2$, 
\[
 \xi D_k-\Delta_k^{\mathrm{in}}
 \wle C_k 
 \wle \xi D_k+\Delta_k^{\mathrm{in}}+U_k,
\]
with probability at least $1-n^{-2}-A_\eta n_1^{-\eta^2}$. 
\end{proposition}

Once $V_1$ has been completely explored, no edge internal to $V_1$ crosses the exploration cut, so all potential cut edges have probability $p_n$.  The following lemma quantifies the resulting concentration of $C_k = |\Cut(\cI_k)|$ both shortly after $K_2$ and at macroscopic exploration indices. 

\begin{lemma}%[Cut concentration after dense-component exploration]
\label{lem:post_dense_cut}
Suppose that \(\alpha \in (1/2,1)\) and $1 \ll n p_n \ll n^\alpha / \log n$.
\begin{enumerate}[(i)]
\item \emph{(Shortly after full exploration of the dense component.)}
For every fixed $B>0$,
\[
 \max_{ K_2\le k\le K_2+\lceil Bnp_n\log n\rceil } 
 \left| \frac{C_k}{p_nk(n-k)}-1 \right|
 \weq o_{\pr}(1).
\]

\item \emph{(Longer after the full exploration of the dense component.)}
For every $0<c_1<c_2<1$,
\[
 \max_{ \lceil c_1n\rceil\le k\le \lfloor c_2n\rfloor }
 \left| \frac{C_k}{p_nk(n-k)}-1 \right|
 \weq \cO_{\pr}\left(\frac{1}{\sqrt{np_n}}\right). 
 %= o_{\pr}(1).
\]
\end{enumerate}
\end{lemma}

\section{Detection of the Dense Component Phase}
\label{sec:detection}

\subsection{Oracle and plug-in transforms}
\label{subsec:plugin_transform_stability}

The change-point procedure in Algorithm~\ref{algo:k1k2_estimation} uses the plug-in background rate $\widehat m_k:=\widehat p_n\,k(n-k)$, where \(\widehat p_n\) is estimated from the late window \(\{k:n/3\le k\le n/2\}\). For the analysis, it is convenient to instead work with the oracle background rate $m_k:=p\,k(n-k)$.
Hence, we define the oracle transform and the feasible transform by 
%and to define the corresponding oracle transform
\[
 Z_k:=\sqrt{12}\left(e^{-m_k a_k}-\frac12\right)
 \quad \text{ and } \quad 
 \widehat Z_k := \sqrt{12}\left(e^{-\widehat m_k a_k}-\frac12\right),
\]
respectively. The algorithm uses the feasible transforms $\widehat Z_k$, and the following lemma shows that the difference between the oracle and the feasible transforms is asymptotically negligible. 
Importantly, this proof does not introduce any circular argument: the estimator \(\widehat p_n\) is computed on the fixed late window \([n/3,n/2]\), and its consistency follows from Corollary~\ref{cor:p_hat_rate}, and does not rely on the consistency of the change-point estimators \(\widehat K_1\) and \(\widehat K_2\). 

 \begin{lemma}[Stability of the probability integral transform]
\label{lem:PIT_plugin_stability}
Let $\varepsilon_n:=\left|\frac{\widehat p_n}{p_n}-1\right|$. 
Conditionally on \(\{\varepsilon_n<1\}\), we have 
\[
 \max_{1 \le k \le \lfloor n / 2 \rfloor } \left| \widehat Z_k-Z_k \right| \wle \frac{\sqrt{12}}{e} \, \frac{\varepsilon_n}{1-\varepsilon_n}.
\]
In particular, for every collection of bins \(B_j \subseteq [1, \lfloor n / 2 \rfloor ] \),
\[
 \sup_j \left| \frac1{|B_j|}\sum_{k\in B_j}\widehat Z_k - \frac1{|B_j|}\sum_{k\in B_j}Z_k \right|
 \wle \frac{\sqrt{12}}{e} \, \frac{\varepsilon_n}{1-\varepsilon_n}.
\]
\end{lemma}
\begin{proof}
Write $r_n:=\frac{\widehat p_n}{p_n}$ and $y_k:=m_k a_k\ge0$. 
Then $\widehat m_k a_k=r_n y_k$, and hence
\[
 |\widehat Z_k-Z_k|
 \weq \sqrt{12}\left|e^{-r_n y_k}-e^{-y_k}\right|.
\]
By the mean-value theorem, we have $\left|e^{-r_n y_k}-e^{-y_k}\right| = |r_n-1|\,y_k e^{-s y_k}$ for some $s \in [1, r_n]$. %for some \(s\) between \(1\) and \(r_n\).
On the event \(\{|r_n-1|<1\}\), we have \(s\ge 1-\varepsilon_n\). Therefore
\[
 y_k e^{-s y_k}
 \wle \sup_{y\ge0} \, y e^{-(1-\varepsilon_n)y}
 \weq \frac{1}{e(1-\varepsilon_n)}.
\]
Thus $|\widehat Z_k-Z_k| \le \frac{\sqrt{12}}{e} \frac{\varepsilon_n}{1-\varepsilon_n}$. Averaging over bins gives the second claim.
\end{proof}

Lemma~\ref{lem:PIT_plugin_stability} ensures that, if \(\widehat p_n / p_n\to1\) in probability, then $\sup_k|\widehat Z_k-Z_k|=o_{\mathbb P}(1)$, and the same holds uniformly for all binned averages. 
In the remainder of this section, we prove the change-point guarantees for the oracle variables \(Z_k\). By Lemma~\ref{lem:PIT_plugin_stability} and the consistency of $\widehat p_n$ (Corollary~\ref{cor:p_hat_rate}), the same conclusions hold for the feasible variables \(\widehat Z_k\) used by Algorithm~\ref{algo:k1k2_estimation}, after changing the fixed thresholds by a negligible \(o_{\mathbb P}(1)\) amount.

\subsection{Detecting the Start of the Dense-Component Phase}
\label{subsubsec:detecting_K1}

A key technical point is that the probability integral transform $Z_k$ defined in Expression (\ref{def:probability_integral_transform}) is computed using the idealized baseline rate $m_k = p_n k(n-k)$, whereas the true conditional rate of the process is the random cut size $C_k$. The following lemma shows that this mismatch is negligible before the cascade reaches $V_1$, but becomes detectable as soon as the exploration enters the dense component. Thus, the transformed variables have near-zero drift before $K_1$ and a positive drift of asymptotic magnitude $\sqrt{3}$ shortly after $K_1$. 

 \begin{lemma}%[Drift of the transformed increments before and after the start of the dense-component phase]
\label{lm:jump_size}
%The following holds 
\begin{enumerate}[(i)]
  \item \emph{(Before reaching the dense component)} Suppose that $\alpha \in (0,1)$ and \(n p_n \gg 1\). We have $\max\limits_{1\le k<K_1} \E[Z_k\mid\cF_k]=o_{\pr}(1)$. 
  \item \emph{(After reaching the dense component.)} Suppose that $\alpha \in (1/2, 1)$ and $1 \ll n p_n \ll n^{2\alpha-1} / \log n$. For any fixed $D > 0$, we have  $\max\limits_{K_1 \le k\le K_1+D\log n} \left| \E[Z_k\mid\cF_k] - \sqrt{3} \right| \weq o_{\pr}(1)$. 
  %Used to be:  $\max\limits_{K_1+c_1\log n\le k\le K_1+c_2\log n} \left| \E[Z_k\mid\cF_k] - \sqrt{3} \right| \weq o_{\pr}(1)$
\end{enumerate}
\end{lemma}

The proof is given in Section~\ref{sec:prf_jump_size}. Lemma~\ref{lm:jump_size} shows that the transformed variables $Z_k$ have vanishing mean for every $k < K_1$ and a positive mean drift of order $ \sqrt{3} $ shortly after~$K_1$ vertices have been explored. This suggests a simple detection rule: average the $Z_k$ over consecutive bins of size $w$ and declare a change as soon as a bin average crosses a threshold~$H$ strictly between $0$ and $\sqrt{3}$. The next proposition shows that, with bin size $w = \Theta(\log n)$, this rule localizes $K_1$ up to a single bin whp. 

 \begin{proposition}[Consistency of the first-entry estimator]
\label{prop:K1_detection}
Suppose that $\alpha\in(1/2,1)$ and $1 \ll np_n \ll n^{2\alpha-1} / \log n$.
Let $\kappa_1$ be an integer with $\kappa_1=\Theta(\sqrt n)$, let $z\in(0,\sqrt3)$ and $c>12/z^2$, and set $H_1:=z/2$ and $w:=\lceil c\log n\rceil$.
Then the estimator $\widehat K_1$ returned by Algorithm~\ref{algo:k1k2_estimation} satisfies
\[
    \pr\bigl(|\widehat K_1-K_1|\le w\bigr)=1-o(1).
\]
\end{proposition}

The proof is given in Section~\ref{sec:prf_K1_detection}. 
To obtain the explicit bound $|\widehat K_1-K_1|\le 4.1\log n$ in Theorem~\ref{thm:estimation}, choose, for instance, $z=1.72<\sqrt 3$, for which $\frac{12}{z^2}\simeq 4.056<4.09.$ As the post-change drift in Lemma~\ref{lm:jump_size} is $\sqrt 3(1-o(1))$, this choice of $z$ is valid for all sufficiently large~$n$. Thus, the bin width $w=\lceil 4.09\log n\rceil$ satisfies the condition required in Proposition~\ref{prop:K1_detection}. Moreover, $w\le 4.1\log n$ for all sufficiently large $n$, which yields the claimed bound.

\subsection{Detecting the End of the Dense-Component Phase}

The following lemma is an analogue of Lemma~\ref{lm:jump_size} and states that the transformed variables $Z_k$ have non-vanishing mean for $k< K_2$ and vanishing mean for $k > K_2$. 

\begin{lemma}%[Drift before and after the end of the dense-component phase]
\label{lm:K_2_Z_exp}
Suppose that \(\alpha \in (1/2,1)\) and $1 \ll n p_n \ll n^\alpha / \log n$. 
Let $s_n$ and $r_n$ be deterministic sequences such that $n^{1-\alpha}\log n+np_n\log n \ll s_n\ll n_1$ and $np_n\ll r_n\ll n_1$. 
%Set $s_n:=\sqrt n+np_n\log^2 n,$ and let $(r_n)$ be any deterministic sequence such that $np_n\ll r_n\ll n_1$. 
Fix two strictly positive constants $b$ and $B$. 
Then, 
\begin{enumerate}[(i)]
\item Before the dense component is fully explored, $\max\limits_{ K_1+s_n-b\log n \le k\le K_2^\star-r_n } \left| \E[Z_k\mid\mathcal F_k] - \sqrt{3} \right| = o_{\pr}(1);$
\item Right after, we have $ \max\limits_{ K_2 \le k\le K_2 + B n p_n\log n } \E[Z_k\mid\mathcal F_k] = o_{\pr}(1)$. 
\end{enumerate}
\end{lemma}

The proof of this lemma is given in Section~\ref{sec:K_2_Z_exp}.

\begin{proposition}[Consistency of the endpoint estimator]
\label{prop:K2_detection}
Suppose that $\alpha\in(1/2,1)$ and $1 \ll n p_n \ll n^{2\alpha-1} / \log n$. 
Fix $z\in(0,\sqrt3)$ and $c>12/z^2$, and set $H_2:=\frac z4$, $w:=\lceil c\log n\rceil$, $\widehat s_n :=\left\lceil\sqrt n+n\widehat p_n\log^2 n\right\rceil$. 
Then the estimator $\widehat K_2$ returned by
Algorithm~\ref{algo:k1k2_estimation} satisfies
\[
 |\widehat K_2-K_2^\star| \weq \cO_{\pr}(np_n\log n).
\]
\end{proposition}

\subsection{Estimation of Dense Component Size Parameter \texorpdfstring{$\alpha$}{alpha}}
Having estimated $K_1$ and $K^*_2 = K_1 + \lfloor n^\alpha \rfloor$, the estimation of $\alpha$ simply becomes
\begin{equation*}
    \widehat \alpha = \frac{\log \left(\widehat K_2 - \widehat K_1\right)}{\log n}.
\end{equation*}
%The following corollary shows the accuracy of the $\alpha$ estimation derived from the bounds on the $K_1$ and $K_2^*$ estimation.
\begin{corollary} \label{cor:alpha_consistency}
 Under the assumptions of Theorem~\ref{thm:estimation}, %the $\alpha$ estimate provided by Algorithm~\ref{algo:k1k2_estimation} satisfies  
 $|\widehat\alpha-\alpha| \weq \cO_{\pr} \left( n^{1-\alpha} p_n \right).$
\end{corollary}
The proof is given in Section~\ref{sec:pf_alpha_consistency}.

\section{Estimating \texorpdfstring{$p$}{p} and \texorpdfstring{$\xi$}{xi}}
\label{sec:estimation}

\subsection{Estimation of the background density \texorpdfstring{$p$}{p}}
\label{subsec:estimating_p}

We now turn to the estimation of the background edge probability $p$. Recall  that, conditionally on the history $\cF_k$, the inter-activation time $a_k$ is exponentially distributed with rate $C_k$. When $k$ lies outside the dense phase, the cut size $C_k$ is dominated by its background contribution $pk(n-k)$, and the rescaled variables $k(n-k)\,a_k$ are approximately i.i.d.\ $\Exp(p)$. 
This suggests estimating  $ p $  by maximum likelihood on a window of explorations occurring after the dense phase has been absorbed. We choose the window $ k \in [n/3,\, n/2] $, which lies entirely above the dense regime as $\alpha < 1 $.

We proceed in two steps. First, Lemma~\ref{lem:p_hat_concentration} gives a concentration bound for an MLE built from independent exponential whose rates are only \emph{approximately} of the form $p\,k(n-k)$. Second, Lemma~\ref{lem:post_dense_cut}~(ii) shows that, on the window $[n/3, n/2]$, the true rates $C_k$ indeed satisfy such an approximate relation with high probability. Combining the two yields Corollary~\ref{cor:p_hat_rate}. 
%Proofs of both lemmas are deferred to Appendix~\ref{app:proofs_estimation_p}.

\begin{lemma}[MLE for nearly-homogeneous exponentials]
\label{lem:p_hat_concentration}
Let $\mathcal W\subseteq\{1,\ldots,n-1\}$ be a nonempty
deterministic index set, and write $M:=|\mathcal W|$.
Let $p>0$, and let $(X_k)_{k\in\mathcal W}$ be independent
exponential random variables with rates $(\mu_k)_{k\in\mathcal W}$
satisfying
\[
    p\,k(n-k)(1-\varepsilon_-)
    \wle \mu_k
    \wle p\,k(n-k)(1+\varepsilon_+) 
    \qquad \forall k\in\mathcal W,
\]
for some $\varepsilon_-,\varepsilon_+\in[0,1)$.
Set $\varepsilon :=\max\left\{ \frac{\varepsilon_-}{1-\varepsilon_-}, \frac{\varepsilon_+}{1+\varepsilon_+} \right\}$ and define 
$\hp_{\MLE} := \frac{M}{\sum_{k\in\mathcal W}k(n-k)X_k}.$ 
There exists an absolute constant $C>0$ such that, for every $x\ge0$ satisfying $\varepsilon+u(x)<1$, 
\[
 \frac{|\hp_{\MLE}-p|}{p}
 \wle
 \frac{\varepsilon+u(x)}{1-\varepsilon-u(x)}.
\]
with probability at least $1-2e^{-x}$, where $u(x):=\frac{C}{1-\varepsilon_-} \left(\sqrt{\frac{x}{M}}+\frac{x}{M}\right)$. 
\end{lemma}

 \begin{corollary}%[Consistency of $\widehat p$]
\label{cor:p_hat_rate}
Assume \(\alpha \in (1/2,1)\) and $1 \ll n p_n \ll n^\alpha / \log n$.
%Assume the conditions of Theorem~\ref{thm:estimation}. 
Then
$\frac{|\widehat p_n-p_n|}{p_n}
 = \cO_{\pr}\left(\frac{1}{\sqrt{np_n}}\right)
    %= o_{\pr}(1).
$
\end{corollary}

\begin{proof}
Set $\cS_n := \left\{ \lceil n/3\rceil,\ldots,\lfloor n/2\rfloor \right\}$ and $m:=|\cS_n|=\Theta(n).$ 
By Lemma~\ref{lem:post_dense_cut}~(ii), applied with $c_1=1/3$ and $c_2=1/2$, we have 
\[
 \varepsilon_n
 := \sup_{k\in\cS_n} \left| \frac{C_k}{p_nk(n-k)}-1 \right|
 \weq \cO_{\pr}\left(\frac{1}{\sqrt{np_n}}\right).
\]
In particular, $\varepsilon_n=o_{\pr}(1)$.

Let $\mathcal J := \sigma\bigl(G,V_1,(u_k)_{k\ge1}\bigr).$ 
Conditionally on $\cJ$, the inter-arrival times $\{a_k:k\in\cS_n\}$ are independent and satisfy $a_k \cond \cJ \sim\Exp(C_k).$ 
Moreover, on the event $\{\varepsilon_n<1\}$,
\[
 p_nk(n-k)(1-\varepsilon_n)
 \wle C_k 
 \wle p_nk(n-k)(1+\varepsilon_n)
\]
uniformly for $k\in\cS_n$. 
By applying Lemma~\ref{lem:p_hat_concentration} conditionally on $\cJ$, we obtain 
\[
    \frac{|\widehat p-p_n|}{p_n}
    =
    \cO_{\pr}\left(
        \varepsilon_n+\frac1{\sqrt n}
    \right).
\]
Because $p_n\le1$, we have $\frac1{\sqrt n} \le \frac1{\sqrt{np_n}}$. 
We conclude using $\varepsilon_n = \cO_{\pr}\left(\frac1{\sqrt{np_n}}\right)$. 
\end{proof}

\subsection{Estimation of the dense-component density \texorpdfstring{$\xi$}{xi}}
\label{subsec:estimating_xi}

We finally estimate the internal edge density $\xi$ of the dense component.
It is convenient to write $q_0:=\xi-p_n,$ so that $q_0$ measures the excess edge density inside $V_1$ relative to the background graph.

Observe that the internal edge density $\xi$ affects the inter-arrival rates during the dense phase. Writing $D_k:=I_1(k)\bigl(n_1-I_1(k)\bigr),$ a natural reference rate is 
\begin{equation}
\label{eq:xi_cut_heuristic}
 \mu_k(q_0)
 := p_n k(n-k)+q_0D_k 
 %\weq \xi D_k + p_n\bigl(k(n-k)-D_k\bigr).
\end{equation}
Using the bound on $C_k$ established in Proposition~\ref{prop:G_npbx_cut_bound}, we prove in Lemma~\ref{lem:xi_bulk_approx} in the Appendix that $C_k=\mu_k(q_0) \left(1+o_{\pr}(1)\right)$ uniformly over the range of indices $k$ used for the estimation of $\xi$.  

To turn~\eqref{eq:xi_cut_heuristic} into an estimator of $q_0$, we also need to approximate the number $I_1(k)$ of already infected vertices in $V_1$. 
The results of Section~\ref{sec:detection} show that, away from the beginning and end of the dense phase, almost every newly infected vertex belongs to $V_1$. Hence, throughout the bulk of this phase, $I_1(k) \wapprox k-K_1+1.$ 
%Similarly, the duration of the dense phase estimates its size: $n_1\wapprox K_2^\star-K_1.$
This suggests replacing these unknown quantities by their estimated counterparts $\widehat n_1:=\widehat K_2-\widehat K_1$ and $\widehat s_k:=k-\widehat K_1+1$, and defining
\[
    \widehat D_k
    :=
    \widehat s_k(\widehat n_1-\widehat s_k).
\]
The quantity $\widehat D_k$ is therefore an approximation of the number
$D_k:=I_1(k)(n_1-I_1(k))$ 
of potential dense-component edges crossing the cut.

We do not use the entire estimated dense phase for this approximation.
Near its two endpoints, one of the factors in $D_k$ is small, so that
small localization errors in $\widehat K_1$ or $\widehat K_2$ can lead to
a comparatively large relative error in $\widehat D_k$. We therefore
trim a vanishing fraction of the estimated dense phase from each side.
Specifically, define
\begin{equation}
\label{eq:xi_feasible_trimming}
%\label{eq:xi_window}
    \widehat r_n
    :=
    \left\lceil
        \frac{\widehat n_1}{\log n}
    \right\rceil
\quad \text{ and } \quad 
    \mathcal W_\xi
    :=
    \left\{
        k:
        \widehat K_1+\widehat r_n
        \le k
        \le
        \widehat K_2-\widehat r_n
    \right\}.
\end{equation}
Under the assumptions of Theorem~\ref{thm:estimation}, $\widehat n_1=(1+o_{\mathbb P}(1))n_1$, and the proof of Proposition~\ref{prop:xi_consistency} shows that $\widehat r_n$ is asymptotically much larger than the localization error
$\cO_{\pr}( n p_n \log n )$, while still satisfying $\widehat r_n=o_{\mathbb P}(n_1)$. 
Thus, restricting to $\mathcal W_\xi$ removes the region in which endpoint uncertainty is relevant while retaining a $(1-o_{\mathbb P}(1))$ fraction of the dense phase.

We now construct an estimator of $q_0$. For $k\in\mathcal W_\xi$, define $\widehat B_k:=\widehat p_n \,k(n-k)$ and, for $q\ge0$,
\begin{equation}
\label{eq:xi_plugin_rate}
    \widehat\mu_k(q)
    :=
    \widehat B_k+q\widehat D_k.
\end{equation}
In view of~\eqref{eq:xi_cut_heuristic}, $\widehat\mu_k(q)$ is an approximation of the cut rate that would arise if the excess density were equal to $q$.

If an inter-arrival time $a_k$ were exactly exponential with rate
$\widehat\mu_k(q)$, its log-density, up to an additive constant, would be
$\log\widehat\mu_k(q)-a_k\widehat\mu_k(q).$ 
This motivates fitting $q$ by summing these exponential log-likelihood
contributions over the bulk of the dense phase. We therefore define
\begin{equation}
\label{eq:q_hat_def}
 \widehat q \ \in \
 \argmax_{0\le q\le 1-\widehat p} \sum_{k\in\mathcal W_\xi} \left[ \log\widehat\mu_k(q) - a_k\widehat\mu_k(q) \right].
\end{equation}
Finally, our estimator of the dense-component edge density is $\widehat\xi:=\widehat p+\widehat q.$ 

The likelihood in~\eqref{eq:q_hat_def} is a pseudo-likelihood: the true conditional rate is~$C_k$, rather than $\widehat\mu_k(q_0)$. The consistency proof shows that, uniformly over $k \in \cW_{\xi}$, the estimate $\widehat\mu_k(q_0)$ approximates the true cut rate $C_k$, while the number of retained observations $|\cW_{\xi}|$ is of order $n^\alpha$. This yields consistency of the resulting estimator. 

\begin{proposition}[Consistency of \texorpdfstring{\(\widehat\xi\)}{xi-hat}]
\label{prop:xi_consistency}
Assume the conditions of Theorem~\ref{thm:estimation}, and define
$\widehat r_n$, $\widehat q$, and $\widehat\xi$ by~\eqref{eq:xi_feasible_trimming}--\eqref{eq:q_hat_def}. Then $\widehat\xi-\xi=o_{\pr}(1).$
\end{proposition}

The proof is given in Section~\ref{sec:pf_xi_consistency}.

\subsection{Estimating the vertices in \texorpdfstring{$V_1$}{V1}}
%Given the estimators for the beginning and the end of the bulk of the $V_1$ exploration $\hat{K}_1$ and $\hat{K}_2$, an easy algorithm for the recovery of the vertices in the dense component is to return all vertices infected between $\hat{K}_1$ and $\hat{K}_2$. Theorem~\ref{thm:weak_recovery} shows that this approach satisfies asymptotic weak recovery (as in \cite{hajek_information_2016}).

 \begin{proof}[Proof of Theorem~\ref{thm:weak_recovery}]
Define the reference set $V_1^\star:=\{u_k:K_1\le k<K_2^\star\}.$ 
Because $K_2^\star-K_1=n_1=|V_1|$, we have $|V_1^\star|=|V_1|$.
Moreover, the number of vertices of $V_0$ contained in $V_1^\star$ is at most the total number of $V_0$-vertices explored between $K_1$ and the full exploration of $V_1$, namely $K_2-K_2^\star$. Hence
\[
 |V_1\triangle V_1^\star|
 \weq 2|V_1^\star\cap V_0|
 \wle 2(K_2-K_2^\star).
\]
Furthermore, because we defined $\widehat V_1=\{u_k:\widehat K_1\le k<\widehat K_2\}$, 
the symmetric difference between the estimated and reference intervals satisfies
\[
 |\widehat V_1\triangle V_1^\star|
 \wle |\widehat K_1-K_1| + |\widehat K_2-K_2^\star|.
\]
Therefore,
\[
 |V_1\triangle\widehat V_1|
 \wle 2(K_2-K_2^\star) + |\widehat K_1-K_1| + |\widehat K_2-K_2^\star|.
\]
By Propositions~\ref{prop:I1_in_V1_bd},
\ref{prop:K1_detection}, and~\ref{prop:K2_detection}, we obtain 
$|V_1\triangle\widehat V_1| = O_{\pr}(np_n\log n).$
Because both $|\widehat V_1\cap V_0|$ and $|\widehat V_1^c\cap V_1|$ are bounded by $|V_1\triangle\widehat V_1|$, the claimed bounds follow.
\end{proof}

\section{Numerical Experiments}
\label{sec:numerical_experiments}

We supplement our analysis with illustrations from experimental results on the cut size estimates and the estimation of the parameters. This gives an insight into the behavior of the processes beyond the asymptotic bounds.

\subsection{Cut Size with a Dense Component}

We simulate the cut-size process in the presence of a dense component; see Figure~\ref{fig:gnpx_planted_K1_w_closeup}. The simulations use a process whose cut-size distribution matches that of the exploration process on a $\gnpx$ graph with $n=10^6$, $p=2\log(n)/n$, $\alpha=0.7$, $\xi=0.8$ and ten repetitions. In the process, the start of the $V_1$ activation, $K_1$, is a random variable which makes it hard to plot the mean and variance of the process. Instead, we condition the process such that the first $V_1$ activation occurs at $K_1=\lfloor n^{1-\alpha}\rfloor$ to compare averages and standard deviations across runs. This conditioning is only used for illustrative purposes in Figures \ref{fig:gnpx_planted_K1_w_closeup} and \ref{fig:I1_frac}. All estimation simulations are performed on the unconditioned process.

The dense-component contribution dominates the background contribution throughout most of the interval $[K_1,K_1+n^\alpha]$. Writing $j=k-K_1+1$, the dense-component cut size is well approximated by $\xi j(n^\alpha-j)$.
Near the end of the dense-component activation, the empirical curve deviates from this approximation because the remaining cut within $V_1$ becomes comparable to the background cut, making activations outside $V_1$ more likely.

\begin{figure}[!ht]
    \centering
    \includegraphics[width=0.7\linewidth]{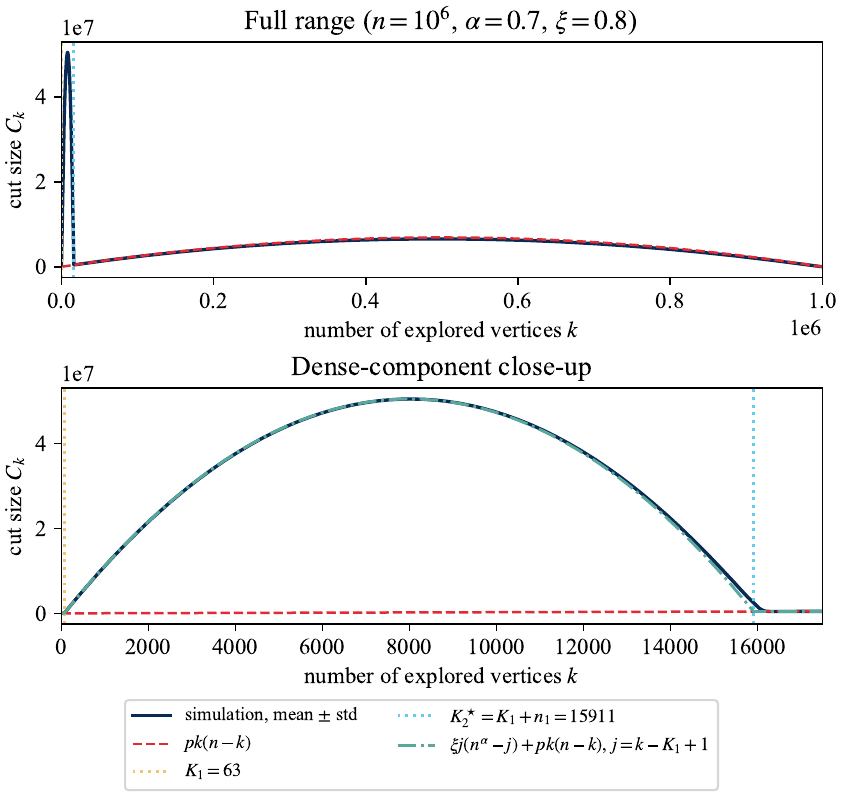}
    \caption{Cut-size simulations averaged over ten repetitions on a $\gnpx$ model with $n=10^6$, $p=2\log(n)/n$, $\alpha=0.7$, and $\xi=0.8$. The shaded region shows one standard deviation. The dense-component onset is planted at $K_1=\lfloor n^{1-\alpha}\rfloor$.}
    \label{fig:gnpx_planted_K1_w_closeup}
\end{figure}

A key quantity in the analysis is the fraction of activations that occur inside $V_1$ after $K_1$, which we bound in Proposition~\ref{prop:I1_in_V1_bd}. In Figure~\ref{fig:I1_frac} we plot the fraction of activations occurring in $V_1$, computed over windows of $100$ activations and averaged over ten repetitions, again conditioned on $K_1= \lfloor n^{1-\alpha}\rfloor$. We observe that, after $K_1$, almost all subsequent activations occur within $V_1$. However, near the end of the dense-component activation, this fraction decreases as the cut within $V_1$ becomes comparable to the background cut.

\begin{figure}[!ht]
    \centering
    \includegraphics[width=0.7\linewidth]{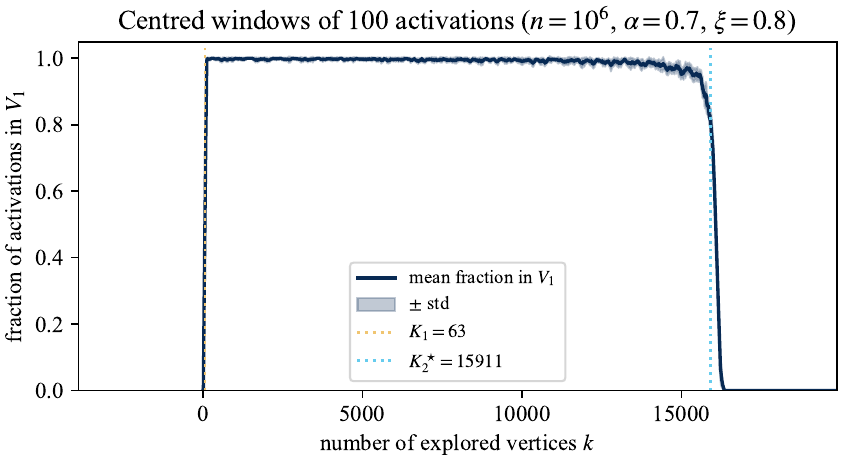}
    \caption{Fraction of activations occurring in $V_1$, computed over windows of $100$ activations. Results are averaged over ten repetitions on a $\gnpx$ model with $n=10^6$, $p=2\log(n)/n$, $\alpha=0.7$, and $\xi=0.8$. The shaded region shows one standard deviation. The dense-component onset is planted at $K_1=\lfloor n^{1-\alpha}\rfloor$.}
    \label{fig:I1_frac}
\end{figure}

\subsection{Estimation of \texorpdfstring{$K_1$}{K1} and of \texorpdfstring{$\alpha$}{alpha}}

We next evaluate the accuracy of the estimators $\hat{K}_1$ and~$\hat{\alpha}$ via Algorithm \ref{algo:k1k2_estimation}. The parameters $\kappa_1, w, H_1, H_2$ and $s_n$ were tuned for the finite-size experiments by simulating a grid of different settings $n, \alpha$ and $\xi$ and choosing a parameter configuration that minimizes $(|\hK_1- K_1| + |\hK_2 - K^\star_2|)/n_1$ defaulting to $1$ in case of failure averaged over five repetitions.
We report the estimation errors for ten cascades on ten unseen graph instances. Figure~\ref{fig:K1_beta_vs_n} shows the errors as $n$ increases with fixed $\alpha=0.6$ and $\xi=0.5$. The estimation accuracy improves as $n$ increases, consistent with concentration of the quantities used by the estimators.

\begin{figure}[ht!]
    \centering
    \includegraphics[width=0.7\linewidth]{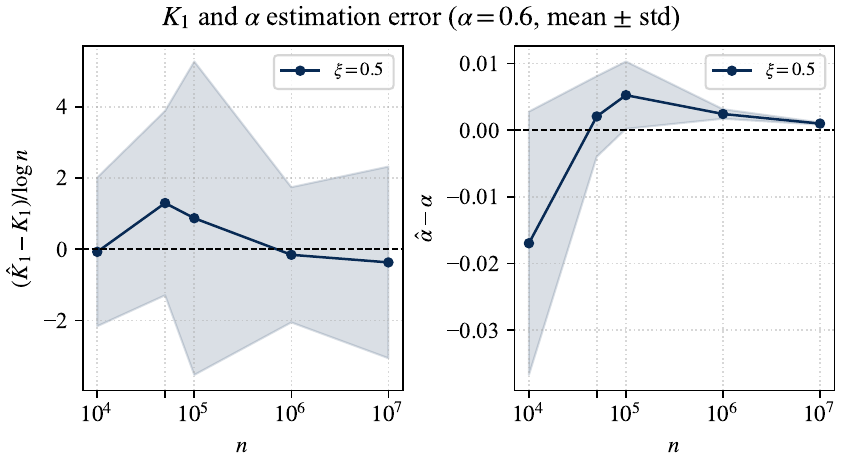}
    \caption{Estimation errors of $\hat{K}_1$ and $\hat{\alpha}$ as $n$ increases, with $\alpha=0.6$ and $\xi=0.5$, averaged over ten repetitions. The shaded region shows one standard deviation.}
    \label{fig:K1_beta_vs_n}
\end{figure}

\subsection{Estimation of \texorpdfstring{$\xi$}{xi}}
Next, we evaluate the estimator $\hat{\xi}$. We report the signed error $\hat{\xi}-\xi$. For each parameter setting, we compare two versions of the estimator: one using the previously estimated values $\hat p, \hat K_1$ and $\hat\alpha$ (using the previously tuned $\kappa_1, w, H_1, H_2$ and $s_n$), and an oracle version using the ground-truth values $p^*, K_1^*$ and $\alpha^*$. Figure~\ref{fig:xi_esti_vs_n} shows that the estimation error decreases as $n$ increases.
\begin{figure}[!ht]
    \centering
    \includegraphics[width=0.7\textwidth]{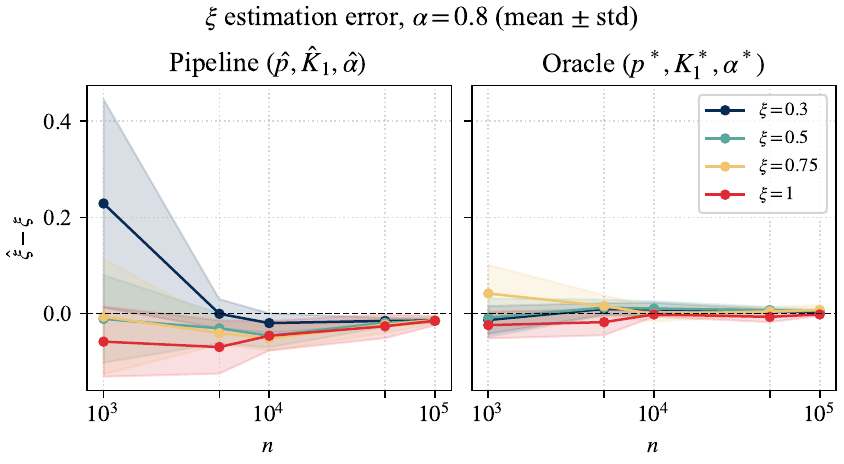}
    \caption{Signed estimation error $\hat{\xi}-\xi$ as $n$ increases, with $\alpha=0.8$, averaged over ten repetitions. The shaded region shows one standard deviation, and different colors correspond to different ground-truth values of $\xi$. The left panel uses the estimated values $\hat p, \hat K_1$ and $\hat\alpha$, while the right panel uses the ground-truth values $p^*, K_1^*$ and $\alpha^*$.}
    \label{fig:xi_esti_vs_n}
\end{figure}

\subsection{Recovery of the dense component \texorpdfstring{$V_1$}{V1}}
Finally, we show in Figure~\ref{fig:v1_recovery} the recovery error of the dense component $V_1$. We show the mean number of false positive (FP) and the false negative (FN) divided by $n_1$ for a graph with $10^7$ vertices and $\alpha = 0.6$ again with the tuned parameters $\kappa_1, w, H_1, H_2$ and $s_n$. We also plot one standard deviation in the shaded region over ten repetitions. 
The simulations show both error rates decreasing as $\xi$ becomes larger as the signal of the dense component becomes more pronounced with higher density.

\begin{figure}[!ht]
    \centering
    \includegraphics[width=0.7\textwidth]{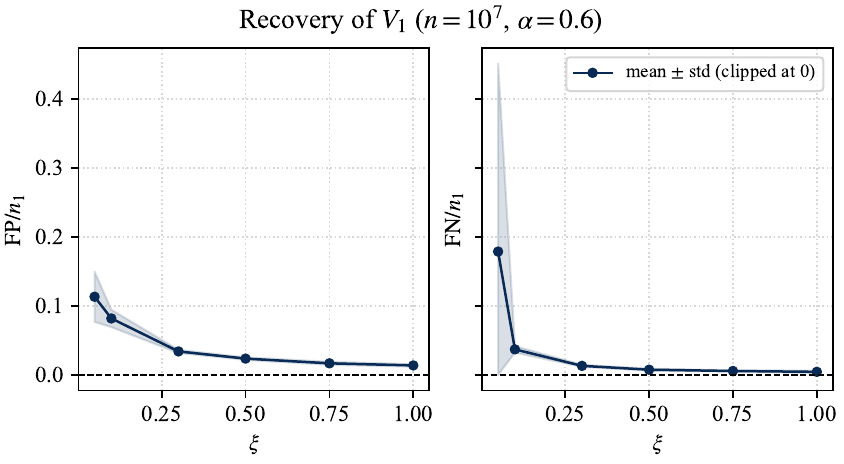}
    \caption{False positives (FP) and false negatives (FN) divided by $n_1$ in the recovery of the vertices in $V_1$. Experiments use $n=10^7$ and $\alpha=0.6$, averaged over ten repetitions. The shaded region shows one standard deviation.}
    \label{fig:v1_recovery}
\end{figure}

\section{Proofs of Section~\ref{sec:exploration_curve}}
\label{sec:exploration_curve_proofs}

\subsection{Proof of Proposition~\ref{prop:homogeneous-cut-concentration}}

\begin{proof}[Proof of Proposition~\ref{prop:homogeneous-cut-concentration}]
Condition on the source $u_1$. As $d_n := np_n \to\infty$, we have $\delta_n:= \sqrt{\frac{24\log(e d_n)}{d_n}}<1$ for all sufficiently large $n$. 

Fix a deterministic set $S$ containing $u_1$ such that $k := |S|\le n/2$, and let $\mu_k=p_nk(n-k)$. 
Define the two events
\[
    A_S:=\{G[S]\text{ is connected}\}
    \quad \text{ and } \quad 
    B_S:=\bigl\{
        \bigl||\operatorname{Cut}(S)|-\mu_k\bigr|
        >\delta_n\mu_k
    \bigr\}.
\]
The event $A_S$ depends only on edges internal to $S$, whereas
$B_S$ depends only on edges between $S$ and its complement.
These edge families are disjoint and $S$ was chosen independent of $G$, so $A_S$ and $B_S$ are independent.

There are $k(n-k)$ possible crossing edges, each present independently with probability~$p_n$. Thus $|\operatorname{Cut}(S)|\sim\Bin(k(n-k),p_n)$.
The multiplicative Chernoff bound, together with $\mu_k\ge d_n k/2$, gives
\[
 \pr(B_S)
 \wle 2\exp\left(-\frac{\delta_n^2\mu_k}{3}\right)
 \wle 2\exp\left(-\frac{\delta_n^2 d_n k}{6}\right).
\]

For $k\ge2$, connectivity requires the presence of a spanning tree.
Rooting such a tree at $u_1$ assigns a parent to each of the other
$k-1$ vertices, so there are at most $k^{k-1}$ possible trees.
Each tree has $k-1$ edges and is present with probability $p_n^{k-1}$.
A union bound therefore yields $\pr(A_S)\le k^{k-1}p_n^{k-1}.$ 

By independence of $A_S$ and $B_S$, we obtain
\[
 \pr(A_S\cap B_S)
 \wle 2k^{k-1}p_n^{k-1} \exp\left(-\frac{\delta_n^2 d_n k}{6}\right).
\]

There are $\binom{n-1}{k-1}$ sets of size $k$ containing $u_1$, and
\[
 \binom{n-1}{k-1}k^{k-1}p_n^{k-1}
 \wle (ed_n)^{k-1} \left(\frac{k}{k-1}\right)^{k-1}
 \wle e(ed_n)^{k-1}.
\]
Hence, a union bound over all these sets and their sizes gives
\begin{align*}
 \pr\left( \bigcup_{k=2}^{\lfloor n/2\rfloor} \ \bigcup_{\substack{S \colon u_1 \in S, \\ |S|=k}} (A_S\cap B_S) \right) 
 & \wle 2e\sum_{k=2}^{\lfloor n/2\rfloor} (ed_n)^{k-1} \exp\left(-\frac{\delta_n^2d_nk}{6}\right)\\
 & \wle 2e\sum_{k=2}^{\infty}(ed_n)^{-3k-1}
 \weq o(1),
\end{align*}
where the second inequality uses $\delta_n^2 d_n/6=4\log(ed_n)$.
For $k=1$, the only set is $\{u_1\}$, and the same Chernoff bound gives $\pr(B_{\{u_1\}})\le2(ed_n)^{-4}=o(1)$.

It follows that, with probability $1-o(1)$, every connected set $S$ containing $u_1$ and with $|S|\le n/2$ satisfies
\[
 \left| \frac{|\operatorname{Cut}(S)|}{p_n|S|(n-|S|)}-1 \right|
 \wle \delta_n.
\]
We conclude as $\cI_k$ contains $u_1$ and has cardinality $k \le n/2$.  
\end{proof}

\subsection{Proof of Proposition~\ref{prop:I1_in_V1_bd}}
\label{sec:proof_I1}

The key ingredient for proving Proposition~\ref{prop:I1_in_V1_bd} is to bound, at each step $k\ge K_1$, the conditional probability that the next visited vertex lies in $V_1$. 
Because the edge weights are independent and exponentially distributed, 
%By the memorylessness of the edge weights,
\begin{equation}\label{eq:pruj_defn}
    \pr(u_{K_1+j}\in V_1\mid \cF_{K_1+j-1})
    \weq \frac{ |\{(v,s)\in \mathrm{Cut}(\cI_{K_1+j-1}) : v \in \cI_{K_1+j-1}, s\in V_1\}| }{C_{K_1+j-1}},
\end{equation}
i.e., the fraction of cut edges leading into $V_1$.  

\begin{lemma}[Lower bound on the conditional probability of activating a $V_1$-vertex]
\label{lm:lower-bound-prujF}
Suppose $\alpha \in (0,1)$ and $n p_n \gg 1$. 
 There exist constants \(c,C_0,C_1>0\) and an event \(\mathcal E\), with
\(\mathbb P(\mathcal E^c)\le C_0n^{-c}\), such that, on \(\mathcal E\), for every \(k\le n/2\) with \(t_k<\infty\) and \(1\le i:=I_1(k)\le n_1-1\),
\[
\mathbb P(u_{k+1}\in V_1\mid\mathcal F_k)
\ge
\left(1+C_1\frac{p_nk(n-k)+\log n}{i(n_1-i)}\right)^{-1}
\ge
1-C_1\frac{p_nk(n-k)+\log n}{i(n_1-i)}.
\]
\begin{comment}OLD VERSION (2026/10/01)
Suppose $\alpha>1/2$ and $pn\log n = o(n^\alpha)$. 
Let $K_f:=\inf\{k\ge 1: I_1(k) = n_1\}$ be the index of the last explored vertex in $V_1$. 
For every fixed constant $A>0$, there exist constants $c,C,C''>0$ and an event $\mathcal E$ with $\pr(\mathcal E^c)\le C n^{-c}$
such that, on $\mathcal E$, for every
\TODO{add details on $\cE$}
\[
 1\le j < \min \left\{ n^\alpha+A pn\log n, K_f -K_1 +1 \right\},
\]
writing $k:=K_1+j-1$ and $i:=I_1(k)$, we have $1\le i\le n^\alpha-1$ and
\[
\prujF
\wge \frac{1}{1+C''\frac{pk(n-k)+\log n}{i(n^\alpha-i)}}
\wge 1-C''\frac{pk(n-k)+\log n}{i(n^\alpha-i)}.
\]
\end{comment}
\end{lemma}

The proof of Lemma~\ref{lm:lower-bound-prujF} is technical and is provided in Appendix~\ref{app:cond_prob_V1}.

\begin{proof}[Proof of Proposition~\ref{prop:I1_in_V1_bd}]
Let $H = \{K_1\le n^{1-\alpha}\log n\}$ and write $\pr_H(\cdot):=\pr(\cdot\mid H)$ and $\E_H[\cdot]:=\E[\cdot\mid H]$. 
As~$H$ is $\cF_{K_1}$-measurable, conditioning on $H$ preserves the conditional transition probabilities after $K_1$, that is, $\pr_H(u_{k+1}\in V_1\mid\cF_k) = \pr(u_{k+1}\in V_1\mid\cF_k)$ for all $k\ge K_1$.

Let $\mathcal E$ be the event on which the uniform cut bounds used in Lemma~\ref{lm:lower-bound-prujF} hold. 
Because $\pr(H)=1-o(1)$, these bounds give $\pr_H(\mathcal E^c)\le C_0n^{-c_0}$ for some constants $C_0,c_0>0$. 
Let $k\le n/2$ and $i:=I_1(k)\in\{1,\ldots,n_1-1\}$. On $\cE$, we have 
\begin{equation}
\label{eq:dense-transition-bound}
 \pr(u_{k+1}\in V_1\mid\cF_k)
 \wge \left( 1+C_1\frac{d_nk+\log n}{i(n_1-i)} \right)^{-1}.
\end{equation}
This also ensures that the cut is positive while $V_1$ remains partially explored.

\medskip
\noindent
\textit{(i) Early regime.}
Set $m:=m_n$. Before the first outside exploration, when exactly
$i$ vertices of $V_1$ have been explored, we have $k=K_1+i-1$.
For $1\le i<m$, using $m=o(n_1)$, on $\cE$ Equation~\eqref{eq:dense-transition-bound} therefore gives
\[
 \pr(u_{K_1+i}\in V_0\mid\cF_{K_1+i-1})
 \wle C\left( \frac{d_nK_1+\log n}{in_1} + \frac{d_n}{n_1} \right). 
\]

For the auxiliary exploration that restarts uniformly whenever its cut vanishes, the first-entry index $\widetilde K_1$ satisfies $\E[\widetilde K_1]=\frac{n+1}{n_1+1}.$\footnote{Indeed, in a $\gnp$, label symmetry and uniform restarts make the exploration order a uniform random permutation. Therefore \(\widetilde K_1\) has the distribution of the position of the first vertex of \(V_1\) in such a permutation. The \(n-n_1\) other vertices occupy \(n_1+1\) exchangeable gaps: before, between, and after the vertices of \(V_1\). Each gap has expected size \((n-n_1)/(n_1+1)\). Because \(\widetilde K_1\) is one plus the first gap’s size, we obtain
$\E[\widetilde K_1] = 1+\frac{n-n_1}{n_1+1} = \frac{n+1}{n_1+1}.$}
Moreover, $\widetilde K_1=K_1$ on $H$. Hence,
\[
 \E_H[K_1]
 \wle \frac{\E[\widetilde K_1]}{\pr(H)}
 \weq \cO(n/n_1).
\]

For $1\le i<m$, let $A_i$ be the event that the first $i$ explorations starting at $K_1$ are finite and belong to $V_1$, and that the cut bounds hold at step $K_1+i-1$.
Then $A_i$ is $\cF_{K_1+i-1}$-measurable. Using the preceding bound on the probability of exploring $V_0$ next, we obtain
\begin{align*}
 \pr_H\bigl(A_i\cap\{u_{K_1+i}\in V_0\}\bigr) 
 & \weq \E_H\left[ \mathbf1_{A_i} \pr(u_{K_1+i}\in V_0\mid\cF_{K_1+i-1}) \right] \\
 &\wle C\,\E_H\left[ \mathbf1_{A_i} \left( \frac{d_nK_1+\log n}{in_1} + \frac{d_n}{n_1} \right) \right] \\
 &\wle C\left( \frac{d_n\E_H[K_1]+\log n}{in_1} + \frac{d_n}{n_1} \right).
\end{align*}
For sufficiently large $n$, we have $m<n_1$.
On $\mathcal E$, extinction is impossible while the explored vertices of $V_1$ number between $1$ and $m-1$.
Thus, if the first $m$ explorations starting at $K_1$ are not all finite and in $V_1$, then either $\mathcal E$ fails or a first outside exploration occurs after $i$ consecutive vertices of $V_1$, for some $1\le i<m$.
A union bound therefore gives
\begin{align}
 &\pr_H\bigl( \text{the first $m$ explorations starting at $K_1$ are not all finite and in $V_1$ }
    \bigr) \nonumber
    \\
    &\quad\le
    \pr_H(\mathcal E^c)
    +\sum_{i=1}^{m-1}
      \pr_H\bigl(A_i\cap\{u_{K_1+i}\in V_0\}\bigr) \nonumber
    \\
    &\quad\le
    \pr_H(\mathcal E^c)
    +C\left[
        \frac{d_n\E_H[K_1]+\log n}{n_1}\log(m+1)
        +\frac{d_nm}{n_1}
    \right] \nonumber \\
    & \wle C_0n^{-c_0}
    +C\left[
        \left(\frac{d_nn}{n_1^2}+\frac{\log n}{n_1}\right)
        \log(m+1)
        +\frac{d_nm}{n_1}
    \right], \label{eq:in_proof}
\end{align}
where we used $\sum_{i=1}^{m-1}i^{-1}\le C\log(m+1)$, and $\E_H[K_1]=O(n/n_1)$ and
$\pr_H(\mathcal E^c)\le C_0n^{-c_0}$.

The right-hand side of~\eqref{eq:in_proof} tends to zero, because $\frac{d_nn}{n_1^2}\log(m+1) =O\left(\frac{d_n\log n}{n^{2\alpha-1}}\right)=o(1)$, $\frac{\log n\log(m+1)}{n_1}=o(1)$ and $\frac{d_nm}{n_1}=o(1).$ 
This proves (i). 

\medskip
\noindent
\textit{(ii) Bulk regime.}
Let $S$ be the total number of outside explorations after $K_1$
and before completing $V_1$, with $S=\infty$ if completion never
occurs. Fix a constant $A>0$, to be chosen sufficiently large,
and set
$M:=\lceil A d_n\log n\rceil.$ 
Stop the comparison at completion of $V_1$, the first failure of
a current cut bound, or the $(M+1)$-st outside exploration.
Before stopping,
\[
 k \wle K_1+n_1+M \wle 3n_1 \ < \ n/2
\]
for all sufficiently large $n$.
Thus \eqref{eq:dense-transition-bound} gives the lower success
probabilities
\[
    q_i:=
    \left(
        1+L\frac{d_nn_1}{i(n_1-i)}
    \right)^{-1},
    \qquad 1\le i\le n_1-1,
\]
where $L>0$ can be chosen independently of the fixed constant $A$.

Using fresh independent uniform sequences for the successive trials in each stage $i$, couple the exploration with independent variables
$\widetilde G_i\sim\operatorname{Geom}(q_i)-1.$ 
Up to stopping, the number of outside explorations in each stage, including any partially completed stage, is at most~$\widetilde G_i$.
On $\mathcal E$, extinction before completion is impossible.
Therefore, reaching $M+1$ outside explorations forces $\sum_i\widetilde G_i>M$, and hence
\[
    \pr_H(S>M)
    \wle
    C_0n^{-c_0}
    +
    \pr\left(
        \sum_{i=1}^{n_1-1}\widetilde G_i>M
    \right).
\]
Moreover, 
\[
 \E \left[ \sum_i \tG_i \right]
\weq \sum_{i=1}^{n_1-1}\frac{1-q_i}{q_i}
 \weq Ld_nn_1\sum_{i=1}^{n_1-1}\frac1{i(n_1-i)}
 \weq 2Ld_n\sum_{i=1}^{n_1-1}\frac1i
 \wasymp d_n\log n,
\]
where we used $\frac1{i(n_1-i)} = \frac1{n_1}\left(\frac1i+\frac1{n_1-i}\right).$
Therefore, Lemma~\ref{lm:geom_sum_concentration} applied with $q_\ast:=\min_{1\le i\le n_1-1}q_i\ge\frac{c_1}{d_n}$ and with $A$ sufficiently large, gives $\pr\left( \sum_{i=1}^{n_1-1}\widetilde G_i>M \right)
 \le n^{-c_2}$ for some $c_2>0$.
After adjusting constants, we obtain
\[
    \pr_H(S\le Cd_n\log n) \wge 1-n^{-c}
\]
for some $c,C>0$. 
On this event, before completion of $V_1$,
\[
    I_1(K_1-1+j)\wge j-S\wge j-Cd_n\log n.
\]
After completion, $I_1=n_1$. 
This proves the uniform bound in (ii).
Finally, completion occurs at index $K_1+n_1-1+S$, so $K_2-K_2^\star = S \le Cd_n\log n$.
\end{proof}

\subsection{Proof of Proposition~\ref{prop:G_npbx_cut_bound}}

\begin{proof}[Proof of Proposition~\ref{prop:G_npbx_cut_bound}]
Recall that $G[V_1]$ is the subgraph induced by $V_1$.  For $S\subseteq V_1$, define the set of edges of the induced subgraph $G[V_1]$ crossing from $S$ to $V_1\setminus S$ by 
\[
 \Cut_{G[V_1]}(S)
 := \bigl\{ \{u,v\}\in E(G[V_1]) \colon u\in S,\ v\in V_1\setminus S \bigr\}. 
\]

We decompose the cut $C_k = | \Cut(\cI_k) |$ after $k$ explored vertices as 
$C_k \weq C_k^{\mathrm{in}}+C_k^{\mathrm{out}},$ 
where $C_k^{\mathrm{in}} := \left| \Cut_{G[V_1]}(I_k\cap V_1) \right|$  is the contribution of cut edges internal to $V_1$, and $C_k^{\mathrm{out}}:=C_k-C_k^{\mathrm{in}}$ contains all remaining cut edges.

\medskip\noindent
\textit{(i) Internal dense contribution.}
The induced graph $G[V_1]$ is distributed as $G(n_1,\xi)$.

If $\xi<1$, then, because $\xi$ is fixed, we have $\xi(1-\xi)n_1\gg\log n_1,$ and the assumptions of Lemma~\ref{lm:dense_cut_conc} hold for sufficiently large $n$. 
Applying that lemma to $G[V_1]$, there exists a constant $A_\eta<\infty$ such that, with probability at least $1-A_\eta n_1^{-\eta^2}$, simultaneously for every subset $S\subseteq V_1$, $1\leq |S|\leq n_1-1$,
\begin{align*}
 &\left| \Cut_{G[V_1]}(S) - \xi |S|(n_1-|S|) \right| \\
 & \wle (1+\eta)  \sqrt{ 2\xi(1-\xi) |S|(n_1-|S|) \min\{|S|, n_1-|S|\}\log\left( \frac{e n_1}{\min\{|S|, n_1-|S|\}} \right) },
\end{align*}

As this bound holds simultaneously for all subsets of $V_1$, we can apply it to the random set $S=\Ik\cap V_1$. 
Moreover, observe that, for $M_k\ge1$, we have $\log\left(\frac{e n_1}{M_k}\right) \wle \log(en_1).$ 
Thus, we obtain, simultaneously for all $k$,
\begin{align}
\label{eq:in_proof_cut_all_phase}
 \left| C_k^{\mathrm{in}}-\xi D_k \right| \wle \Delta_k^{\mathrm{in}}.
\end{align} 

If $\xi=1$, the graph induced by $V_1$ is complete, so $C_k^{\mathrm{in}}=D_k$ and the bound~\eqref{eq:in_proof_cut_all_phase} holds with $\Delta_k^{\mathrm{in}}=0$.

Denote by $\mathcal E_{\mathrm{in}}$ the event $ \left| C_k^{\mathrm{in}}-\xi D_k \right| \wle \Delta_k^{\mathrm{in}}$. We just established that 
$\pr(\mathcal E_{\mathrm{in}}^c)
 \le A_\eta n_1^{-\eta^2}.$

\smallskip \noindent
\textit{(ii) Sparse contribution.}
We condition on $V_1$ and let $X_i$ be the random variable counting the number of edges from $u_i$ to $\cI_i^c$ excluding pairs contained in $V_1$. 
Conditioned on the exploration history just before $u_i$ is explored, the edge-indicators from $u_i$ to $\cI_i^c$ but excluding pairs contained in $V_1$ are independent $\Ber(p_n)$ variables, and there are at most $n$ of them. 
Hence, $X_i \le Z_i$ where $Z_i \sim \Bin(n,p_n)$.
Moreover, because $Z_i$ involves only edges incident to~$u_i$, and because the vertices $u_1,\ldots,u_k$ are distinct, the random variables $Z_1,\ldots,Z_k$ can be constructed to be mutually independent, and thus $\sum_{i=1}^k Z_i \sim \Bin(kn,p_n)$.
Because every edge in $C_k^{\mathrm{out}}$ was first counted in some $X_i$ with $i \le k$, we obtain 
\[
    C_k^{\mathrm{out}} \ \stleq \ \sum_{i=1}^k X_i \ \stleq \ \sum_{i=1}^k Z_i \wsim \Bin( kn, p_n).
\]
Bernstein's inequality gives $\pr(C_k^{\mathrm{out}} > p_nkn + \sqrt{6p_nkn\log n} + \log n) \le n^{-3}$. A union bound over $1\le k\le n/2$ therefore yields an event $\mathcal E_{\mathrm{out}}$ for which $C_k^{\mathrm{out}} \le U_k$ for all $1\le k\le n/2$, and such that $\pr(\mathcal E_{\mathrm{out}}^c) \le \frac{1}{2}n^{-2}$.  

\medskip\noindent 
\textit{Conclusion.} We conclude by considering the event $\mathcal E_{\mathrm{in}}\cap\mathcal E_{\mathrm{out}}$.  
\end{proof}

%% file: appendix.tex
\section{Deferred proofs from Section~\ref{sec:exploration_curve}}

\subsection{Stochastic upper-bound for \texorpdfstring{$C_k$}{Ck}}
\label{subsec:stoc-upper-bound-cut}
\begin{lemma}[Stochastic upper bound]
\label{lemma:homogeneous-cut-upper}
Consider the sequence $(C_1,\cdots,C_n)$ of cut-sizes of a continuous FPP process with rate $\lambda=1$ in a $\gnp$ graph. 
For every $1\le k\le n-1$,  
\[
    C_k \ \stleq \ \Bin(k(n-k),p_n).
\]
\end{lemma}

\begin{proof} 
Compare the FPP exploration with an auxiliary exploration on $G(n,p)$ that selects each next vertex uniformly from the unexplored vertices, independently of the graph. 
Its explored set $\mathcal J_k$ has size $k$ and is independent of the graph, so its cut size satisfies
\[
 C_k^+:=|\operatorname{Cut}(\mathcal J_k)|
 \sim \Bin(k(n-k),p).
\]
We couple the two explorations by revealing edges only when one of
their endpoints is explored.

After $k$ steps, let $x_1\le\cdots\le x_m$ and $y_1\le\cdots\le y_m$, where $m=n-k$, be the ordered numbers of explored neighbors of the remaining vertices in the FPP and auxiliary explorations, respectively. 
Observe that $C_k =\sum_{i=1}^m x_i$ and $C^+_k =\sum_{i=1}^m y_i.$

It remains to be shown that there exists a coupling such that $x_i\le y_i$ for every $i$.
Initially, couple these sequences to be equal. Suppose inductively
that $x_i\le y_i$ for every $i$ and that $\sum_i x_i>0$.
The FPP exploration selects a rank $R$ with probabilities proportional
to the $x_i$, whereas the auxiliary exploration selects a uniform
rank $U$. Since the $x_i$ are nondecreasing,
\[
    \pr(R\le r\mid x_1,\ldots,x_m)
    =\frac{\sum_{i=1}^r x_i}{\sum_{i=1}^m x_i}
    \le \frac{r}{m},
    \qquad 1\le r\le m.
\]
Thus $R$ and $U$ can be coupled so that $R\ge U$.

Deleting $x_R$ and $y_U$ preserves the coordinatewise ordering
of the remaining ordered sequences. Indeed, between the two deleted
ranks, the new comparison is $x_i\le y_{i+1}$; elsewhere the original
comparisons remain valid. The edges from each newly selected vertex
to its remaining unexplored vertices have not yet been revealed.
Their indicators are fresh independent Bernoulli$(p)$ variables.
Couple these indicators coordinatewise in the two explorations.
Adding the same indicators and then reordering preserves the
coordinatewise inequality.

Induction gives $C_k\le C_k^+$ until the FPP exploration stops.
Afterward, $C_k=0\le C_k^+$ by convention. The claimed stochastic
domination follows.
\end{proof}

\subsection{Cut Size in a Dense \texorpdfstring{$G_{n,p}$}{Gnp} Graph}
%Corollary~\ref{cor:cut_conc_local} studies the cut size in a sparse $\gnp$. 
In the dense component of the graph, higher edge density simplifies the description of the cut size, because a concentration bound can be derived uniformly over all sets of vertices. 
We use the following relatively standard result. 

\begin{lemma}[Uniform Concentration of Cuts in Dense $\gnp$]
\label{lm:dense_cut_conc}
Let $G \sim \gnp$, let $\eta>0$, and assume that either
$p(1-p)n \ge C_\eta \log(en)$, with $C_\eta:=\frac{(1+\eta)^2(1+\eta^2)^2}{\eta^2}$, or $p=1$. 
For $S \subseteq V(G)$, let $m(S):=\min\{|S|,n-|S|\}$. Then, with probability at least $1 - 2\sum_{k=1}^{\lfloor n/2\rfloor}\binom{n}{k}^{-\eta^2}$, we have simultaneously for all $S\subseteq V$ that
\[
\left| \Cut(S) - p|S|(n-|S|) \right|
\wle (1+\eta)\sqrt{ 2p(1-p)|S|(n-|S|) \log\binom{n}{m(S)}}.
\]
Moreover, for every fixed $\eta>0$, there is a constant $A_\eta<\infty$ such that $2\sum_{k=1}^{\lfloor n/2\rfloor}\binom{n}{k}^{-\eta^2} < A_{\eta} n^{-\eta^2}$. 
\end{lemma}

\begin{remark}
Suppose $p=\Theta(1)$. Then, using $|S|(n-|S|)\ge \frac{mn}{2}$ and $\log\binom{n}{m}\le m\log\left(\frac{en}{m}\right)$ where $n = m(S)$, the lemma implies that, for $n$ large enough, we have 
\[
\forall S\subseteq V(G) \colon \qquad
\frac{\left| \Cut(S)-p|S|(n-|S|)\right|}{p|S|(n-|S|)}
\wle
2(1+\eta)\sqrt{ \frac{1-p}{p} \frac{\log(en/m)}{n}} 
\weq \cO\left( \sqrt{\frac{\log n}{n}} \right)
\]
with probability at least $1-A_\eta n^{-\eta^2}$. 
\end{remark}

\begin{proof}
Since $\Cut(S)=\Cut(V\setminus S)$, it is enough to consider $S\subseteq V$ with $|S|\le n/2$. 
Fix such a set and let $k:=|S|$. Then for a graph generated independently of $S$, 
\[
 \Cut(S) \sim \mathrm{Bin}(k(n-k),p).
\]
Set $\mu_k:=pk(n-k)$, $\sigma_k^2:=p(1-p)k(n-k)$, and $ L_k:=\log\binom{n}{k}$. 
Bernstein's inequality gives, for every $t>0$,
\[
\pr\left( |\Cut(S)-\mu_k|\ge t \right)
\wle 2\exp\left(-\frac{t^2}{2\sigma_k^2+2t/3} \right).
\]
We will use this inequality with $t_k := (1+\eta)\sqrt{2\sigma_k^2 L_k} 
 %\weq(1+\eta)\sqrt{2p(1-p)k(n-k)\log\binom{n}{k}}.
$. 
Then
\begin{equation}\label{eq:bernstein_prob}
    \frac{t_k^2}{2\sigma_k^2+2t_k/3}
\weq
\frac{(1+\eta)^2L_k}{1+t_k/(3\sigma_k^2)}.
\end{equation}

We now bound the extra Bernstein denominator uniformly in $k$. Using $k\le n/2$, we have
$ n-k\ge \frac n2$, and therefore
$\sigma_k^2
= p_n(1-p_n)k(n-k)
\ge \frac12 p_n(1-p_n)kn.$
Also,
$L_k = \log\binom{n}{k}
 \le k\log\left(\frac{en}{k}\right)
 \le k\log(en).
$ 
Hence
\begin{align*}
%\label{eq:in_proof_bernstein_term_1}
\frac{t_k}{3\sigma_k^2}
%&= \frac{1+\eta}{3} \sqrt{\frac{2L_k}{\sigma_k^2}} \\
%& \wle \frac{1+\eta}{3}\sqrt{ \frac{2k\log(en)} {p(1-p)k(n/2)}} \\
\wle \frac{2(1+\eta)}{3} \sqrt{ \frac{\log(en)}{p_n(1-p_n)n} }
\wle \frac{2\eta}{1+\eta^2}, 
\end{align*}
where the last inequality uses $p_n(1-p_n)n \ge C_\eta \log(en)$ and the definition of~$C_\eta$. 
Hence, 
\begin{align}
\label{eq:in_proof_bernstein_term_2}
1+\frac{t_k}{3\sigma_k^2}
\wle 1+\frac{2\eta}{1+\eta^2}
\weq \frac{(1+\eta)^2}{1+\eta^2}.
\end{align}
Therefore, combining~\eqref{eq:bernstein_prob} and~\eqref{eq:in_proof_bernstein_term_2}, we obtain 
$
\frac{t_k^2}{2\sigma_k^2+2t_k/3}
%\weq \frac{(1+\eta)^2L_k}{1+t_k/(3\sigma_k^2)}
\ge (1+\eta^2)L_k.
$
Thus, for each fixed $S$ of size $k\le n/2$,
\[
\pr\left( \left| \Cut(S)-\mu_k \right| \ge t_k \right) 
\wle 2e^{-(1+\eta^2)L_k} 
\weq 2\binom{n}{k}^{-(1+\eta^2)}.
\]
Taking a union bound over all subsets of size $k$ and then over $1\le k\le n/2$ gives
\begin{align*}
\pr\left( \exists S\subseteq V,\ |S|\le n/2 \colon |\Cut(S)-\mu_{|S|}|\ge t_{|S|} \right)
 \wle \sum_{k=1}^{\lfloor n/2\rfloor} \binom{n}{k} 2\binom{n}{k}^{-(1+\eta^2)} 
 \weq  2\sum_{k=1}^{\lfloor n/2\rfloor} \binom{n}{k}^{-\eta^2}.
\end{align*}
This proves the first claimed probability estimate.

It remains only to prove that the last sum involving $\binom{n}{k}^{-\eta^2}$ is polynomially small. Indeed, for every fixed $\eta>0$ there exists a constant $A_\eta<\infty$ such that, for all sufficiently large $n$,
\begin{align}
\label{eq:in_proof_binomial_sum_bound}
 2\sum_{k=1}^{\lfloor n/2\rfloor} \binom{n}{k}^{-\eta^2} \wle A_\eta n^{-\eta^2}
\end{align}
We show this by splitting the sum into $k\le \sqrt n$ and $k>\sqrt n$. For $1\le k\le \sqrt n$, using $\binom{n}{k}\ge (n/k)^k$, we get $\binom{n}{k}^{-\eta^2} \le \left(\frac{k}{n}\right)^{\eta^2 k}$. 
The $k=1$ term is $n^{-\eta^2}$, and the remaining terms are of smaller order than a constant multiple of $n^{-\eta^2}$ for fixed $\eta$. For $k>\sqrt n$, again using $\binom{n}{k}\ge (n/k)^k$ and $k\le n/2$, we have
\[
        \binom{n}{k}^{-\eta^2}
        \le 2^{-\eta^2 k}
        \le 2^{-\eta^2\sqrt n},
\]
so the contribution of these terms is exponentially small in $\sqrt n$, and this establishes~\eqref{eq:in_proof_binomial_sum_bound}. 
\end{proof}

\subsection{Proof of Lemma~\ref{lm:lower-bound-prujF}}
\label{app:cond_prob_V1}

 \begin{proof}[Proof of Lemma~\ref{lm:lower-bound-prujF}]
Denote by $C_k^{\mathrm{in}}$ and $C_k^{\mathrm{out}}$ the numbers of cut edges whose unexplored endpoints belong to $V_1$ and $V_0$, respectively. 
Let $\cE_{\mathrm{dense}}$ be the event that
\[
 \bigl|\Cut_{G[V_1]}(S)\bigr|
 \wge \frac{\xi}{2}|S|(n_1-|S|)
 \qquad\text{for every }S\subseteq V_1.
\]
As $\xi>0$ is fixed and $n_1=\lfloor n^\alpha\rfloor$, Lemma~\ref{lm:dense_cut_conc} implies, for all sufficiently large $n$,
\[
    \pr(\mathcal E_{\mathrm{dense}}^c)\le Cn^{-c}
\]
for some constants $c,C>0$.
When $\xi=1$, this event holds deterministically.

The sparse-cut bound established in the proof of
Proposition~\ref{prop:G_npbx_cut_bound} gives an event
$\mathcal E_{\mathrm{out}}$, with
$\pr(\mathcal E_{\mathrm{out}}^c)\le \frac12 n^{-2}$,
on which, simultaneously for every $k\le n/2$ with $t_k<\infty$,
\[
 C_k^{\mathrm{out}}
 \wle U_k :=p_nkn+\sqrt{6p_nkn\log n}+\log n.
\]
Indeed, $C_k^{\mathrm{out}}$ is bounded by the full sparse contribution considered there.

Set $\cE := \cE_{\mathrm{dense}}\cap\cE_{\mathrm{out}}.$ 
After adjusting constants, we have $\pr(\cE^c)\le C_0n^{-c}$.

We now work on $\cE$, and fix $k\le n/2$ with $t_k<\infty$ and $1\le i:=I_1(k)\le n_1-1$. Every internal dense edge crossing from $I_k\cap V_1$ to $V_1\setminus I_k$ contributes to $C_k^{\mathrm{in}}$.
Hence
\begin{equation}
\label{eq:in_proof_lowerbound_Ckin}
 C_k^{\mathrm{in}}
 \wge \bigl|\Cut_{G[V_1]}(I_k\cap V_1)\bigr|
 \wge \frac{\xi}{2}i(n_1-i) \ > \ 0.
\end{equation}
Moreover, using $k\le n/2$ and $\sqrt{6x\log n}\le x+\frac32\log n$, we obtain
\begin{equation}
\label{eq:in_proof_lowerbound_Ckout}
 C_k^{\mathrm{out}}
 \wle 2p_nkn+\frac52\log n
 \wle 4\bigl(p_nk(n-k)+\log n\bigr).
\end{equation}

By the memorylessness of the exponential exploration times, the next exploration edge is uniform among the current edges belonging to the cut.
Combining \eqref{eq:in_proof_lowerbound_Ckin} and \eqref{eq:in_proof_lowerbound_Ckout}, we therefore obtain
\begin{align*}
 \pr(u_{k+1}\in V_1\mid\cF_k)
 \weq \frac{C_k^{\mathrm{in}}}{C_k^{\mathrm{in}}+C_k^{\mathrm{out}}}
 & \wge \left( 1+\frac{8}{\xi} \frac{p_nk(n-k)+\log n}{i(n_1-i)} \right)^{-1}\\
 & \wge 1-\frac{8}{\xi} \frac{p_nk(n-k)+\log n}{i(n_1-i)},
\end{align*}
where the last inequality uses $(1+x)^{-1}\ge1-x$ for $x\ge0$.
All bounds hold simultaneously over the stated range of $k$,
which proves the lemma.
\end{proof}

\subsection{Proof of Lemma~\ref{lem:post_dense_cut}}
\begin{proof}[Proof of Lemma~\ref{lem:post_dense_cut}] 
The two statements (i) and (ii) rely on the same observation. 
Condition on the planted set $V_1$, and fix a deterministic set $S\subseteq V$ such that $V_1\subseteq S$. 
Because all vertices of $V_1$ belong to $S$, the edges belonging to $\Cut(S)$ each have at most one endpoint in $V_1$. 
Therefore, all potential edges between vertices in $S$ and vertices in $V\setminus S$ are independent $\Ber(p_n)$ variables, and 
\begin{equation}
\label{eq:post_dense_fixed_set_cut}
 |\Cut(S)| \sim \Bin\bigl(|S|(n-|S|),p_n\bigr).
\end{equation}
To prove (i) and (ii), we use this fact over two different families of sets.

\medskip \noindent \emph{Proof of (i).}
By Lemma~\ref{lem:first-entry} and Proposition~\ref{prop:I1_in_V1_bd}, there exists a constant $C>0$
such that the event
$\mathcal E_n := \left\{ K_1\le n^{1-\alpha}\log n \text{ and } \ 0\le K_2-K_2^\star\le Cnp_n\log n \right\}$ satisfies $\pr(\mathcal E_n)=1-o(1)$.

We work on this event $\mathcal E_n$. 
As $K_2$ is the first exploration index after all vertices of $V_1$ have been explored, we have $V_1\subseteq\mathcal I_k$ for every $k \ge K_2$. 
Moreover, for $K_2 \le k \le K_2+\lceil Bnp_n\log n\rceil,$ using $K_2^\star=K_1+n_1$ and the definition of $\cE_n$, we obtain 
\begin{align*}
 |\mathcal I_k\setminus V_1|
 \weq k-n_1
 \wle K_2-n_1+\lceil Bnp_n\log n\rceil 
 \wle m_n, 
\end{align*}
where we defined $m_n := \left\lceil n^{1-\alpha}\log n + (B+C+1)np_n\log n \right\rceil.$
Because $\alpha>1/2$ and $np_n\log n=o(n_1)$, we have $m_n=o(n_1)$. 

We introduce 
\[
 \cS_n := \left\{ S\subseteq V\colon V_1\subseteq S \text{ and } |S\setminus V_1|\le m_n \right\}.
\]
For any $S\in\cS_n$, because $m_n=o(n_1)$ and $n_1=o(n)$, we have $|S|=(1+o(1))n_1$ and $n-|S|=(1-o(1))n$, uniformly in $S$. Hence, writing $\mu_S:=p_n|S|(n-|S|),$ we have
\begin{equation}
\label{eq:post_dense_local_mean}
 \inf_{S\in\cS_n}\mu_S \wge c(np_n)n_1
\end{equation}
for some constant $c>0$ and all sufficiently large $n$.

Next, we upper-bound the size of $\cS_n$. 
Because only vertices of $V_0$ may be added to $V_1$, we have  $|\cS_n| \le \sum_{j=0}^{m_n}\binom{n}{j},$ 
and hence, for some constant $C'>0$,
$\log|\cS_n| \le C'm_n\log\left(\frac{en}{m_n}\right).$ 
Using the definition of $m_n$,
\[
    m_n\log\left(\frac{en}{m_n}\right)
    =
    O\left(
        n^{1-\alpha}\log^2 n
        +
        np_n\log^2 n
    \right).
\]
Dividing by $(np_n)n_1\asymp(np_n)n^\alpha$, and using $\alpha>1/2$ and $np_n\gg 1$ gives
\[
 \frac{ n^{1-\alpha}\log^2 n }{ (np_n)n^\alpha }
 = \frac{ n^{1-2\alpha}\log^2 n }{ np_n }
 =o(1)
 \quad \text{and} \quad 
 \frac{ np_n\log^2 n }{ (np_n)n^\alpha }
 = \frac{\log^2 n}{n^\alpha}
 =o(1).
\]
Therefore, 
\begin{equation}
\label{eq:local_post_dense_entropy}
 \log|\cS_n|
 \weq o\bigl((np_n)n_1\bigr).
\end{equation}

Choose a deterministic strictly positive and vanishing sequence $\varepsilon_n = o(1)$ such that $\varepsilon_n^2(np_n)n_1 \gg \log|\cS_n|.$ 
For every fixed $S\in\cS_n$, Chernoff bounds give
\[
 \pr\left( \bigg| |\Cut(S)|-\mu_S \bigg| > \varepsilon_n\mu_S \right)
 \wle 2\exp\left\{ -c\varepsilon_n^2\mu_S \right\}
\]
for some constant $c>0$. 
Using~\eqref{eq:post_dense_local_mean}, \eqref{eq:local_post_dense_entropy}, and a union bound over $S\in\cS_n$, yields
\[
 \pr\left( \sup_{S\in\cS_n} \left| \frac{|\Cut(S)|}{p_n|S|(n-|S|)} - 1 \right| > \varepsilon_n \right)
 \weq o(1).
\]

Finally, on the event~$\cE_n$, every exploration set $\mathcal I_k$ with $K_2\le k\le K_2+\lceil Bnp_n\log n\rceil$ belongs to $\cS_n$. Hence, taking $S=\mathcal I_k$ in the preceding uniform estimate therefore proves 
\[
 \sup_{ K_2\le k\le K_2+\lceil Bnp_n\log n\rceil } \left| \frac{C_k}{p_nk(n-k)}-1 \right|
 \weq o_{\pr}(1). 
\]

\medskip \noindent 
\emph{Proof of (ii).}
Because $K_1=o_{\pr}(n)$, $n_1=o(n)$, and $np_n\log n=o(n_1)$, we have 
\[
    K_2
    =
    K_2^\star+\cO_{\pr}(np_n\log n)
    =
    K_1+n_1+\cO_{\pr}(np_n\log n)
    =
    o_{\pr}(n),
\]
Hence $\pr(K_2<c_1n) = 1-o(1).$
On this event $\{ K_2 < c_1 n \}$, we have 
\[
    V_1\subseteq\mathcal I_k
    \qquad
    \text{for all }
    \lceil c_1n\rceil\le k\le\lfloor c_2n\rfloor.
\]

Consider
\[
 \cS_n := \left\{ S\subseteq V \colon V_1\subseteq S \text{ and } c_1n\le |S|\le c_2n \right\}.
\]
For every $S\in\cS_n$, \eqref{eq:post_dense_fixed_set_cut} holds and
\[
 \mu_S := p_n|S|(n-|S|)
 \wge c_1(1-c_2)p_nn^2
 \weq c_1(1-c_2)n(np_n).
\]
Moreover, $|\cS_n|\le2^n.$ 
Let $\varepsilon_n:=\frac{M}{\sqrt{np_n}},$ where $M>0$ is a fixed constant. 
Because $np_n\to\infty$, we have $\varepsilon_n<1$ for all sufficiently large $n$. 
Again by Chernoff bounds,
\begin{align*}
\pr\left( \sup_{S\in\cS_n} \left| \frac{|\Cut(S)|}{p_n|S|(n-|S|)}-1 \right| > \frac{M}{\sqrt{np_n}} \right)
& \wle 2^n\, 2\exp\left\{ - c' \frac{M^2}{np_n} c_1(1-c_2)n(np_n) \right\} \\
& \weq 2\exp\left\{ n\left( \log2-c'c_1(1-c_2)M^2 \right) \right\}.
\end{align*}
Choosing $M$ sufficiently large makes the last expression tend to
zero exponentially fast.  Therefore
\[
 \sup_{S\in\cS_n} \left| \frac{|\Cut(S)|}{p_n|S|(n-|S|)}-1 \right|
 \weq \cO_{\pr}\left(\frac{1}{\sqrt{np_n}}\right).
\]
We conclude by noticing that $\cI_k\in\cS_n$ for every $c_1n\le k\le c_2n$, a.a.s. 
\end{proof}

\subsection{Additional Lemma: Concentration of sums of geometric r.v.}
\begin{lemma}[Concentration for sums of geometric random variables]
\label{lm:geom_sum_concentration}
Let \(G_1,\ldots,G_m\) be independent random variables with $G_i\sim \operatorname{Geom}(q_i)-1$, where \(q_i\in(0,1]\). Let $q_*:=\min_{i\in[m]}q_i$ and $\mu:=\sum_{i=1}^m\frac{1-q_i}{q_i}$. 
Then, for every \(M>3/2\), 
\[ 
 \pr\left( \sum_{i=1}^m G_i\ge M\mu \right)
 \wle \exp\left( -\frac{(2M-3)}{12}q_*\mu \right).
\]
\end{lemma}

\begin{proof}
If \(\mu=0\), then \(q_i=1\) for every \(i\), and hence \(G_i=0\) almost surely for every \(i\). The claim is then trivial.
We may therefore assume that \(\mu>0\).  
Recall that, for \(G_i\sim \operatorname{Geom}(q_i)-1\), 
\[
 \E\left[e^{\lambda G_i}\right]
 \weq \sum_{k=0}^\infty q_i(1-q_i)^k e^{\lambda k}
 \weq \frac{q_i}{1-(1-q_i)e^\lambda},
\]
provided $(1-q_i)e^\lambda<1$ or, equivalently, whenever $\lambda<-\log(1-q_i)$.

Let $0 < \lambda < -\log(1-q_*)$. 
By Chernoff's bound and independence,
\[
 \pr\left(\sum_{i=1}^m G_i\ge M\mu\right)
 \wle e^{-\lambda M\mu} \prod_{i=1}^m \E e^{\lambda G_i} 
 \weq e^{-\lambda M\mu} \prod_{i=1}^m \frac{q_i}{1-(1-q_i)e^\lambda}.
\]
Rewrite each factor as
\[
 \frac{q_i}{1-(1-q_i)e^\lambda} \weq \frac{1}{1-x_i}
 \qquad \text{ where }     x_i:=\frac{(1-q_i)(e^\lambda-1)}{q_i}.
\]
Let $\lambda = \log(1+\theta q_*)$ where $\theta\in(0,1)$. Observe that this choice of $\lambda$ is admissible, since 
$
 e^{\lambda} \weq 1+\theta q_* \wle 1+q_* \wle \frac{1}{1-q_*}. 
$ 
With this choice of $\lambda$, we get
\begin{align*}
 x_i \weq \frac{1-q_i}{q_i} \theta q_* \wle \theta(1-q_i) \wle \theta. 
\end{align*}
Thus,
\begin{align*}
 \log \frac{1}{1-x_i} 
 \wle \frac{x_i}{1-x_i}
 \wle \frac{x_i}{1-\theta}
 \weq \frac{\theta q_*}{1-\theta} \frac{1-q_i}{q_i}. 
\end{align*}
Because $\mu = \sum_{i=1}^m \frac{1-q_i}{q_i}$, we obtain 
\begin{align*}
 \pr\left(\sum_{i=1}^m G_i\ge M\mu\right)
 & \wle \exp\left( - \lambda M \mu + \frac{\theta}{1-\theta} q_* \mu \right).
\end{align*}
Using $\log(1+x) \ge x/(1+x)$ valid for $x \ge 0$, we finally obtain 
\begin{align*}
 \pr\left(\sum_{i=1}^m G_i\ge M\mu\right)
 & \wle \exp\left( - \left( \frac{M\theta}{1+\theta} - \frac{\theta}{1-\theta} \right) q_* \mu  \right).
\end{align*}
We conclude by taking $\theta=1/5$. 
\end{proof}

\section{Additional Lemmas for Section~\ref{sec:detection}}

\subsection{Auxiliary Lemmas for Section~\ref{subsubsec:detecting_K1}}

 \begin{lemma}[Concentration of bin averages]
\label{lemma:FP_bound}
Let $(\cF_k)_{k\ge1}$ be a filtration, and let $Z_1,\ldots,Z_\kappa$ be random variables such that $Z_k$ is $\cF_{k+1}$-measurable and takes values in $[-\sqrt3,\sqrt3]$.
Write $b_k:=\E[Z_k\mid\cF_k]$.
Fix an integer $1\le w\le\kappa$, and let $S$ be a stopping time taking values in $\{1,\ldots,\kappa-w+1\}$. Define $\overline Z_S:=\frac1w\sum_{r=0}^{w-1}Z_{S+r}.$ 
For any event $\mathcal E$, the following hold.
\begin{enumerate}
    \item[(i)]
    If, on $\mathcal E$, we have $b_{S+r}\le\delta$ for every
    $0\le r<w$, then, for every $H>\delta$,
    \[
        \pr\bigl(\mathcal E\cap\{\overline Z_S>H\}\bigr)
        \le
        \exp\left\{-\frac{w(H-\delta)^2}{6}\right\}.
    \]

    \item[(ii)]
    If, on $\mathcal E$, we have $b_{S+r}\ge z$ for every
    $0\le r<w$, then, for every $H<z$,
    \[
        \pr\bigl(\mathcal E\cap\{\overline Z_S\le H\}\bigr)
        \le
        \exp\left\{-\frac{w(z-H)^2}{6}\right\}.
    \]
\end{enumerate}
In particular, if $b_k\le\delta$ almost surely for every
$1\le k\le\kappa$, then, for the averages $\overline Z_j :=\frac1w\sum_{k=(j-1)w+1}^{jw}Z_k$ on deterministic bins, we have, for every $H>\delta$,
\[
    \pr\left(
        \max_{1\le j\le\lfloor\kappa/w\rfloor}
        \overline Z_j>H
    \right)
    \le
    \left\lfloor\frac{\kappa}{w}\right\rfloor
    \exp\left\{-\frac{w(H-\delta)^2}{6}\right\}.
\]
\end{lemma}

\begin{proof} 
Set $X_k:=Z_k-b_k$. Conditionally on $\cF_k$, the variable $X_k$ has mean zero and is supported on an interval of length at most $2\sqrt3$. Conditional Hoeffding's inequality therefore gives
\[
 \E[e^{\lambda X_k}\mid\cF_k]
 \wle e^{3\lambda^2/2},
 \qquad \lambda\in\R.
\]
For every deterministic starting index $s$, successive conditioning
on $\cF_{s+w-1},\ldots,\cF_s$ yields
\begin{equation}
\label{eq:bin_conditional_mgf}
    \E\left[
        \exp\left\{\lambda\sum_{r=0}^{w-1}X_{s+r}\right\}
        \,\middle|\,\cF_s
    \right]
    \le e^{3w\lambda^2/2}.
\end{equation}

To handle the random starting index $S$, observe that
$\{S=s\}\in\cF_s$ because $S$ is a stopping time. Thus,
using~\eqref{eq:bin_conditional_mgf},
\begin{align*}
    \E\left[
        \exp\left\{\lambda\sum_{r=0}^{w-1}X_{S+r}\right\}
    \right]
    &=
    \sum_{s=1}^{\kappa-w+1}
    \E\left[
        \mathbf1_{\{S=s\}}
        \exp\left\{\lambda\sum_{r=0}^{w-1}X_{s+r}\right\}
    \right]
    \\
    &=
    \sum_{s=1}^{\kappa-w+1}
    \E\left[
        \mathbf1_{\{S=s\}}
        \E\left[
            \exp\left\{\lambda\sum_{r=0}^{w-1}X_{s+r}\right\}
            \,\middle|\,\cF_s
        \right]
    \right]
    \\
    &\le
    e^{3w\lambda^2/2}
    \sum_{s=1}^{\kappa-w+1}\pr(S=s)
    =
    e^{3w\lambda^2/2}.
\end{align*}
Chernoff's bound, with $\lambda=t/3$ for the upper tail and $\lambda=-t/3$ for the lower tail, gives, for every $t>0$,
\begin{align*}
 \pr\left(\frac1w\sum_{r=0}^{w-1}X_{S+r}\ge t\right)
 \wle e^{-wt^2/6}
  \quad \text{ and } \quad
 \pr\left(\frac1w\sum_{r=0}^{w-1}X_{S+r}\le-t\right)
 \wle e^{-wt^2/6}.
\end{align*}

Under the assumption of (i), on $\mathcal E$,
\[
    \overline Z_S
    \le \delta+\frac1w\sum_{r=0}^{w-1}X_{S+r}.
\]
Consequently, $\mathcal E\cap\{\overline Z_S>H\}$ implies
$\frac1w\sum_{r=0}^{w-1}X_{S+r}>H-\delta$, proving (i).
Similarly, under the assumption of (ii), on $\mathcal E$,
\[
    \overline Z_S
    \ge z+\frac1w\sum_{r=0}^{w-1}X_{S+r},
\]
and the lower-tail bound proves (ii).
Finally, the assertion for deterministic bins follows from (i)
with $\mathcal E=\Omega$ and a union bound.
\end{proof}

\begin{lemma}[False-negative bound]
\label{lemma:FN_bound}
Let \((\mathcal F_k)_{k\ge1}\) be a filtration, and let \(Z_1,\ldots,Z_w\) be random variables such that \(Z_k\) is \(\cF_{k+1}\)-measurable. 
Assume that $Z_k\in[-\sqrt 3,\sqrt 3]$ almost surely and that \( \min_{1 \le k \le w} \E[Z_k\mid\mathcal F_{k}]\ge z .\) 
Then, for every \(H<z\), 
\[
 \pr\left( \frac1w\sum_{k=1}^w Z_k \le H \right)
 \wle \exp\left\{ -\frac{w(z-H)^2}{6} \right\}. 
\]
\end{lemma}

\begin{proof}
Set \( X_k:=Z_k-\E[Z_k\mid\cF_{k}].\)
As in the proof of Lemma~\ref{lemma:FP_bound}, the variables \(X_k\) form a bounded martingale-difference sequence and satisfy the conditional Hoeffding bound
\[
 \pr\left( \sum_{k=1}^w X_k<-wt \right)
 \wle \exp\left(-\frac{w t^2}{6}\right).
\]
Because \(\E[Z_k\mid\mathcal F_{k}]\ge z\), we have
$\sum_{k=1}^w Z_k \wge \sum_{k=1}^w X_k+wz .$ 
Therefore,
\[
 \pr\left( \frac1w\sum_{k=1}^w Z_k \le H \right)
 \wle \pr\left( \sum_{k=1}^w X_k\le -w(z-H) \right)
 \wle \exp\left\{ -\frac{w(z-H)^2}{6} \right\}.
\]
\end{proof}

\subsection{Proof of Lemma~\ref{lm:jump_size}} \label{sec:prf_jump_size}

\begin{proof}[Proof of Lemma~\ref{lm:jump_size}]
Set $d_n:=np_n$ and $m_k:=p_nk(n-k)$.
Whenever $C_k>0$, the conditional law
$a_k\mid\cF_k\sim\Exp(C_k)$ gives
\begin{equation}
\label{eq:EZ_exact}
    \E[Z_k\mid\cF_k]
    =\sqrt3\,\frac{C_k-m_k}{C_k+m_k}.
\end{equation}
By Lemma~\ref{lem:first-entry}, the event $H:=\{K_1\le n^{1-\alpha}\log n\}$ has probability $1-o(1)$.

\smallskip
\noindent
\textit{(i) Pre-change.}
Couple the planted graph with
$\widetilde G\sim G(n,p_n)$ so that they agree on every pair
outside $\binom{V_1}{2}$, and use the same source and transmission
weights on shared edges.
Before the first visit to $V_1$, no internal edge of $V_1$ can
affect the exploration. Thus the two explorations and their cut
sizes coincide for every $k<K_1$.

On $H$, we have $K_1<n/2$ for all sufficiently large $n$.
Proposition~\ref{prop:homogeneous-cut-concentration}, applied to
$\widetilde G$, therefore gives, with high probability,
\[
    \sup_{1\le k<K_1}
    \left|\frac{C_k}{m_k}-1\right|
    \le\delta_n
    \quad \text{ where } \quad 
    \delta_n:=\sqrt{\frac{24\log(ed_n)}{d_n}}=o(1).
\]
Using \eqref{eq:EZ_exact}, we obtain on this event
$\sup\limits_{1\le k<K_1} \bigl|\E[Z_k\mid\cF_k]\bigr|
 \le\frac{\sqrt3\,\delta_n}{2-\delta_n}
 = o(1),
$
which proves (i).

\smallskip
\noindent
\textit{(ii) Post-change.}
Fix $D>0$ and consider $K_1\le k\le K_1+D\log n$.
Because $d_n\log n=o(n_1)$, we may apply Proposition~\ref{prop:I1_in_V1_bd}(i) with $m_n=\lceil D\log n\rceil+1=o(n_1/d_n)$.
Together with $\pr(H)=1-o(1)$, this gives, with high probability, $I_1(k)=j$ (where $j:=k-K_1+1$) uniformly over the stated interval. 
Thus,
\begin{equation}
\label{eq:Dk_early_dense}
    D_k:=I_1(k)(n_1-I_1(k))
    =j(n_1-j)\ge\frac12jn_1
\end{equation}
for all sufficiently large $n$.

Proposition~\ref{prop:G_npbx_cut_bound} gives
$C_k\ge\xi D_k-\Delta_k^{\mathrm{in}}$ with high probability.
Uniformly over the interval,
\[
 \frac{\Delta_k^{\mathrm{in}}}{D_k}
 \weq \cO\left(\sqrt{\frac{\log n}{n_1}}\right) \weq o(1),
\]
with $\Delta_k^{\mathrm{in}}=0$ when $\xi=1$.
As $\xi>0$ is fixed, \eqref{eq:Dk_early_dense} therefore yields
\begin{equation*}
%\label{eq:Ck_early_dense_lower}
    C_k\ge cjn_1
\end{equation*}
for some constant $c>0$, uniformly with high probability.

Using $k=K_1+j-1$ and $j\ge1$, we obtain on $H$
\begin{align*}
 \frac{m_k}{C_k}
 \wle C\frac{d_n(K_1+j)}{jn_1} 
 \wle C\left( \frac{d_n n^{1-\alpha}\log n}{n_1} + \frac{d_n}{n_1} \right)
 \weq o(1),
\end{align*}
where the last equality follows from
$d_n\log n=o(n^{2\alpha-1})$ and $n_1\asymp n^\alpha$.
All these bounds are uniform over the stated interval.
Finally, \eqref{eq:EZ_exact} gives
\[
 \left|\E[Z_k\mid\cF_k]-\sqrt3\right|
 \weq \frac{2\sqrt3\,m_k}{C_k+m_k}
 \wle 2\sqrt3\,\frac{m_k}{C_k}.
\]
Hence $\max\limits_{K_1\le k\le K_1+D\log n} \left|\E[Z_k\mid\cF_k]-\sqrt3\right|
 \weq o_{\pr}(1)$, which proves (ii).
\end{proof}

\subsection{Proof of Proposition~\ref{prop:K1_detection}} \label{sec:prf_K1_detection}

\begin{proof}[Proof of Proposition~\ref{prop:K1_detection}] 
We first work with the oracle variables $Z_k$. 
By Lemma~\ref{lem:PIT_plugin_stability} and Corollary~\ref{cor:p_hat_rate}, $\sup_{k\le n/2}|\widehat Z_k-Z_k|=o_{\pr}(1)$. Hence it is enough to establish the result for the oracle bin averages, and to transfer the conclusion to the feasible averages, we will introduce a fixed margin $\eta>0$ around the threshold $H_1=z/2$.

If necessary, set $Z_k=0$ once the cut vanishes; with probability $1-o(1)$, this convention is never used in the observation window.
Moreover, we write $b_k:=\E[Z_k\mid\cF_k]$. 
Because $c>12/z^2$ and $H_1=z/2$, we may choose fixed $\eta,\epsilon>0$ sufficiently small that
\begin{equation}
\label{eq:K1_eta_choice}
    H_1+\eta<z,
    \qquad
    H_1-\eta-\epsilon>0,
    \qquad
    \frac{c(H_1-\eta-\epsilon)^2}{6}>\frac12.
\end{equation}
By Lemma~\ref{lem:first-entry}, because $\alpha>1/2$ and $\kappa_1=\Theta(\sqrt n)$, we have  $\pr(K_1\le\kappa_1-2w)=1-o(1).$
Together with Lemma~\ref{lm:jump_size},
Lemma~\ref{lem:PIT_plugin_stability}, and
Corollary~\ref{cor:p_hat_rate}, this gives an event
$\mathcal E_n$, with $\pr(\mathcal E_n)=1-o(1)$, on which
\begin{align*}
 K_1\le\kappa_1-2w,
 \quad
 \sup_{k<K_1}b_k\le\epsilon,
 \quad 
 \inf_{K_1\le k\le K_1+2w}b_k\ge z, 
 \quad
 \sup_{k\le n/2}|\widehat Z_k-Z_k|\le\eta.
\end{align*}
For the post-entry drift bound, we apply
Lemma~\ref{lm:jump_size}(ii) with any fixed $D>2c$.

Let $B_j:=\{(j-1)w+1,\ldots,jw\}$, $L_1:=\lfloor\kappa_1/w\rfloor$, and $j_1:=\lceil K_1/w\rceil.$
On $\mathcal E_n$, the bin $B_{j_1+1}$ is entirely contained in the search window, and the oracle and feasible bin averages $Z_j$ and $\hZ_j$ differ by at most $\eta$.

\smallskip
\noindent
\textit{No crossing before $K_1$.}
For each deterministic bin $B_j$, define $\mathcal E_{n,j}:=\mathcal E_n\cap\{jw<K_1\}.$ 
On this event, $b_k\le\epsilon$ throughout $B_j$.
Lemma~\ref{lemma:FP_bound}(i), applied with starting index $(j-1)w+1$ and threshold $H_1-\eta$, therefore gives
\[
 \pr\bigl( \cE_{n,j}\cap \{\overline Z_j>H_1-\eta\} \bigr)
 \wle \exp\left\{-\frac{w(H_1-\eta-\epsilon)^2}{6}\right\}.
\]
Because $j<j_1$ is equivalent to $jw<K_1$, a union bound yields
\begin{align*}
 \pr\left( \cE_n\cap \{\exists j<j_1:\overline Z_j>H_1-\eta\} \right) 
 \wle L_1\exp\left\{ - \frac{w(H_1-\eta-\epsilon)^2}{6} \right\}
 \weq o(1),
\end{align*}
by~\eqref{eq:K1_eta_choice} and $L_1=O(\sqrt n/\log n)$.

\smallskip
\noindent
\textit{Crossing immediately after $K_1$.}
Round $K_1$ up to the next bin start and cap the result at
$\kappa_1$, defining the bounded stopping time
\[
    S:=\min\{w\lceil K_1/w\rceil+1,\kappa_1\}.
\]
For sufficiently large $n$, we have
$S+w-1\le\kappa_1+w-1\le n/2$.
On $\mathcal E_n$, the cap is inactive and
$\{S,\ldots,S+w-1\}=B_{j_1+1}$.
This bin lies entirely between $K_1$ and $K_1+2w$, so
$b_k\ge z$ throughout it.
Lemma~\ref{lemma:FP_bound}(ii), with event $\mathcal E_n$
and threshold $H_1+\eta$, gives
\[
 \pr\left( \cE_n\cap \{\overline Z_{j_1+1}\le H_1+\eta\} \right)
 \wle \exp\left\{-\frac{w(z-H_1-\eta)^2}{6}\right\}
 \weq o(1).
\]

Combining these bounds with $\pr(\cE_n^c)=o(1)$, we find that, with probability $1-o(1)$, no feasible bin strictly preceding $B_{j_1}$ crosses $H_1$, whereas $B_{j_1+1}$ does.
The first crossing therefore occurs in $B_{j_1}$ or $B_{j_1+1}$.
Because the algorithm returns $(j-1)w$ for the first crossing bin, and $(j_1-1)w<K_1\le j_1w$, we obtain 
\begin{align*}
 \pr\bigl(|\widehat K_1-K_1|\le w\bigr) \weq 1-o(1).
\end{align*}
\end{proof}

\subsection{Proof of Lemma~\ref{lm:K_2_Z_exp}} \label{sec:K_2_Z_exp}

 \begin{proof}[Proof of Lemma~\ref{lm:K_2_Z_exp}]
Recall from the proof of Lemma~\ref{lm:jump_size} that 
\begin{equation}
\label{eq:EZ_C_over_m}
    \mathbb E[Z_k\mid \mathcal F_k]
    =
    \sqrt 3\,
    \frac{C_k-m_k}{C_k+m_k}, 
\end{equation}
where $m_k:=p_nk(n-k)$.
Thus, for part~(i) it is enough to prove that $m_k/C_k=o(1)$
uniformly over the stated range, while for part~(ii) it is enough to
prove that $C_k/m_k=1+o(1)$ uniformly.

\smallskip
\noindent
\textit{Proof of (i).}
Write $j:=k-K_1+1$ and $R_k:=n_1-I_1(k)$. 
By Lemma~\ref{lem:first-entry}, with high probability $K_1\le n^{1-\alpha}\log n=o(s_n).$ 
Moreover, Proposition~\ref{prop:I1_in_V1_bd} gives, uniformly over the
range considered,
\[
 j-Cnp_n\log n \wle I_1(k)\wle j.
\]
Because $j\ge s_n-b\log n$ and $np_n\log n=o(s_n)$, it follows that
\begin{equation}
\label{eq:I1_comparable_j}
    I_1(k)=(1-o(1))j
    \quad \text{ and } \quad 
    k=(1+o(1))j,
\end{equation}
uniformly for $K_1+s_n-b\log n \le k \le K_2^{\star} - r_n$.

At the other endpoint, $k\le K_2^\star-r_n=K_1+n_1-r_n$, and hence, using the deterministic bound $I_1(k)\le k-K_1+1$,
\[
 R_k = n_1-I_1(k)
 \wge n_1-(k-K_1+1)
 \wge r_n-1
 \weq (1-o(1))r_n.
\]

Set $D_k:=I_1(k)R_k$ and $M_k:=\min\{I_1(k),R_k\}$. 
By Proposition~\ref{prop:G_npbx_cut_bound}, with high probability, uniformly over $k$, we have 
\[
    C_k\wge \xi D_k-\Delta_k^{\mathrm{in}}.
\]
If $\xi<1$, then
\[
    \frac{\Delta_k^{\mathrm{in}}}{D_k}
    \le
    C
    \sqrt{
        \frac{M_k\log(en_1)}{D_k}
    }
    =
    C
    \sqrt{
        \frac{\log(en_1)}
             {\max\{I_1(k),R_k\}}
    }
    \le
    C
    \sqrt{\frac{2\log(en_1)}{n_1}}
    =
    o(1),
\]
while for $\xi=1$ the fluctuation term is zero. Hence, uniformly over
the range considered,
\[
    C_k\wge c\,I_1(k)R_k
\]
for some constant $c>0$ and all sufficiently large $n$.

Using \eqref{eq:I1_comparable_j}, we therefore obtain
\[
 \frac{m_k}{C_k}
 \wle C\frac{p_nkn}{I_1(k)R_k}
 \wle C'\frac{np_n}{R_k}
 \wle C''\frac{np_n}{r_n}
 \weq o(1),
\]
uniformly for \(K_1+s_n-b\log n\le k\le K_2^\star-r_n\). 
Equation~\eqref{eq:EZ_C_over_m} then gives
\[
    \E[Z_k\mid\mathcal F_k]
    =
    \sqrt3(1-o(1)).
\]

\smallskip
\noindent
\textit{Proof of (ii).}
Fix $B>0$. By Lemma~\ref{lem:post_dense_cut}~(i),
\[
    \varepsilon_n
    :=
    \sup_{
        K_2\le k\le
        K_2+\lceil Bnp_n\log n\rceil
    }
    \left|
        \frac{C_k}{p_nk(n-k)}-1
    \right|
    =
    o_{\pr}(1).
\]
Hence, with probability tending to one, $\varepsilon_n=o(1)$ and,
uniformly over the same interval,
\[
    \left|\frac{C_k}{m_k}-1\right|
    \le \varepsilon_n.
\]
On the event $\{\varepsilon_n<1\}$, Equation~\eqref{eq:EZ_C_over_m}
therefore yields
\[
    \left|
        \E[Z_k\mid\mathcal F_k]
    \right|
    =
    \sqrt3
    \left|
        \frac{C_k/m_k-1}{C_k/m_k+1}
    \right|
    \le
    \sqrt3\,
    \frac{\varepsilon_n}{2-\varepsilon_n}
    =
    o(1)
\]
uniformly for $K_2 \le k\le K_2+\lceil Bnp_n\log n\rceil.$ 
This proves the second claim.
\end{proof}

\subsection{Proof of Proposition~\ref{prop:K2_detection}}

 \begin{proof}[Proof of Proposition~\ref{prop:K2_detection}] 
Write $d_n:=np_n$, $b_k:=\E[Z_k\mid\cF_k]$, and set $r_n:=d_n\log n$ and $s_n^\circ:=\frac12\bigl(\sqrt n+d_n\log^2 n\bigr).$ 
As in the proof of Proposition~\ref{prop:K1_detection},
extend $Z_k$ by zero after extinction if necessary.
Our assumptions imply
\[
    n^{1-\alpha}\log n+d_n\log n
    \ll s_n^\circ\ll n_1,
    \qquad
    d_n\ll r_n\ll n_1,
    \qquad
    w=o(r_n).
\]
Thus Lemma~\ref{lm:K_2_Z_exp} applies with
$s_n=s_n^\circ$ and $r_n=d_n\log n$.

By Corollary~\ref{cor:p_hat_rate},
\[
    \widehat s_n
    =(1+o_{\pr}(1))\bigl(\sqrt n+d_n\log^2 n\bigr).
\]
Hence, $\widehat s_n-s_n^\circ\gg w$ with high probability, and $\widehat s_n=o_{\pr}(n_1)$.
The first bin inspected by the endpoint procedure has first index
\[
    q_n
    :=\widehat K_1
      +w\left\lceil\frac{\widehat s_n}{w}\right\rceil+1
    =K_1+\widehat s_n+\cO_{\pr}(w),
\]
where we used Proposition~\ref{prop:K1_detection}.
It follows that, with high probability,
\[
    K_1+s_n^\circ \wle q_n \wle K_2,
\]
since $K_2\ge K_1+n_1$ whenever $V_1$ is fully explored.
Moreover, Lemma~\ref{lem:first-entry} and
Proposition~\ref{prop:I1_in_V1_bd} give
\begin{equation}
\label{eq:K2_completion_error}
    0\le K_2-K_2^\star=\cO_{\pr}(r_n),
    \qquad
    K_2+r_n=o_{\pr}(n).
\end{equation}

Because
$\frac{c(z-H_2)^2}{6} = \frac{3cz^2}{32}>1,$ 
we may choose fixed $\eta,\epsilon>0$ sufficiently small such that 
$H_2+\eta<z$, $\epsilon+2\eta<H_2$, and $\frac{c(z-H_2-\eta)^2}{6}>1.$ 
Let $\kappa:=\lfloor n/2\rfloor$ and
\[
    \cK_n:=[K_1+s_n^\circ,\;K_2^\star-r_n].
\]
The preceding estimates, Lemma~\ref{lm:K_2_Z_exp}, and
Lemma~\ref{lem:PIT_plugin_stability} yield an event
$\mathcal E_n$, with $\pr(\mathcal E_n)=1-o(1)$, on which
\begin{align*}
    K_1+s_n^\circ\le q_n\le K_2,
    \quad
    K_2+r_n\le\kappa,
    \quad 
    \inf_{k\in\cK_n}b_k\ge z,
    \quad
    \sup_{K_2\le k\le K_2+r_n}b_k\le\epsilon,
    \quad
    \sup_{k\le\kappa}|\widehat Z_k-Z_k|\le\eta.
\end{align*}

\smallskip
\noindent
\textit{(i) No stopping substantially before $K_2^\star$.}
For every deterministic bin
$B_j:=\{(j-1)w+1,\ldots,jw\}$, define
$\mathcal E_{n,j} :=\mathcal E_n\cap\{B_j\subseteq\cK_n\}.$ 
On $\mathcal E_{n,j}$, the drift satisfies $b_k\ge z$ throughout $B_j$.
Lemma~\ref{lemma:FP_bound}(ii), with threshold $H_2+\eta$, therefore gives
\[
    \pr\bigl(
        \mathcal E_{n,j}\cap
        \{\overline Z_j\le H_2+\eta\}
    \bigr)
    \le
    \exp\left\{-\frac{w(z-H_2-\eta)^2}{6}\right\}.
\]
Taking a union bound over all
$L:=\lfloor\kappa/w\rfloor$ deterministic bins yields
\begin{align*}
 \pr\left( \mathcal E_n\cap \{\exists j\le L: B_j\subseteq\cK_n,\; \overline Z_j\le H_2+\eta\} \right)
 \wle L\exp\left\{-\frac{w(z-H_2-\eta)^2}{6}\right\}
 \weq o(1).
\end{align*}
Here the last equality follows from $L=O(n/\log n)$
and the choice of $\eta$.

Outside this exceptional event, on $\mathcal E_n$,
every feasible bin average corresponding to a bin contained
in $\cK_n$ exceeds $H_2$.
Since the first inspected bin starts at or after
$K_1+s_n^\circ$, any downward crossing must occur in a bin
whose last index exceeds $K_2^\star-r_n$.
The reported left endpoint therefore satisfies
\begin{equation}
\label{eq:K2_lower_localization}
    \widehat K_2\ge K_2^\star-r_n-w,
\end{equation}
whenever the procedure finds a crossing.

\smallskip
\noindent
\textit{(ii) Stopping shortly after completion of $V_1$.}
Consider the bounded stopping time
\[
    S:=\min\left\{
        w\left\lceil\frac{K_2}{w}\right\rceil+1,\;
        \kappa-w+1
    \right\}.
\]
On $\mathcal E_n$, for sufficiently large $n$, the cap is
inactive, since $2w\le r_n$ and $K_2+r_n\le\kappa$.
Thus $S$ is the first index of the bin
$B_\ell$, where $\ell:=\lceil K_2/w\rceil+1$.
This bin lies entirely in $[K_2,K_2+r_n]$, so its drift
is bounded above by $\epsilon$ on $\mathcal E_n$.

Lemma~\ref{lemma:FP_bound}(i), with threshold $H_2-2\eta$,
gives
\[
    \pr\left(
        \mathcal E_n\cap
        \{\overline Z_\ell>H_2-2\eta\}
    \right)
    \le
    \exp\left\{
        -\frac{w(H_2-2\eta-\epsilon)^2}{6}
    \right\}
    =o(1).
\]
On the complementary event within $\mathcal E_n$,
the feasible average in this bin is at most
$H_2-\eta<H_2$.
Furthermore, $q_n\le K_2$, so this bin is eligible for
the endpoint search.
The procedure therefore finds a crossing no later than
$B_\ell$, giving
\begin{equation}
\label{eq:K2_upper_actual}
    \widehat K_2
    \le(\ell-1)w
    \le K_2+w.
\end{equation}

Combining the two exceptional-probability bounds with
$\pr(\mathcal E_n^c)=o(1)$, both
\eqref{eq:K2_lower_localization} and
\eqref{eq:K2_upper_actual} hold with high probability.
Together with~\eqref{eq:K2_completion_error} and
$w=o(r_n)$, they imply
\[
    |\widehat K_2-K_2^\star|
    =\cO_{\pr}(r_n)
    =\cO_{\pr}(np_n\log n).
\]
\end{proof}

\subsection{Proof of consistency of \texorpdfstring{$\hat\alpha$}{alpha}} \label{sec:pf_alpha_consistency}
\begin{proof}[Proof of Corollary~\ref{cor:alpha_consistency}]
We derive in this section the consistency rate for \(\widehat\alpha\) from the localization bounds for the dense-phase endpoints. Define $\widehat n_1:=\widehat K_2-\widehat K_1$ and $\widehat\alpha:=\frac{\log \widehat n_1}{\log n}$.
By the triangle inequality, 
\begin{align*}
|\widehat n_1-n_1|
\weq 
\left|
(\widehat K_2-\widehat K_1)
-
(K_2^\star-K_1)
\right| 
\wle
|\widehat K_2-K_2^\star|
+
|\widehat K_1-K_1|.
\end{align*}
The bounds $|\widehat K_1-K_1|=\cO_{\mathbb P}(\log n)$ and $|\widehat K_2-K_2^\star| =\cO_{\mathbb P}\!\left(pn\log n\right)$ together with $np \gg 1$ give  
\[
|\widehat n_1-n_1|
\weq
\cO_{\mathbb P}\!\left(pn\log n\right).
\]
Since \(n_1=\lfloor n^\alpha\rfloor\), this error is \(o_{\mathbb P}(n_1)\)
under the standing assumptions. Hence $\frac{\widehat n_1}{n_1} \wto 1$ in probability.

We now pass from \(\widehat n_1\) to \(\widehat\alpha\). 
Since $n_1=n^\alpha( 1 + \cO(n^{-\alpha}) )$, we have $\log n_1=\alpha\log n+\cO(n^{-\alpha})$. Therefore
\begin{align}
|\widehat\alpha-\alpha| 
\weq \left| \frac{\log \widehat n_1}{\log n} - \alpha \right| 
& \wle \frac{1}{\log n} \left| \log\widehat n_1-\log n_1 \right| + \frac{1}{\log n} \left| \log n_1 - \alpha \log n \right| \nonumber \\ 
%& \wle \frac{1}{\log n} \left| \log\widehat n_1-\log n_1 \right| + O\!\left(\frac1{n^{\alpha} \log n}\right) \nonumber \\
& \weq \frac{1}{\log n} \left| \log\left(\frac{\widehat n_1}{n_1}\right) \right| + \cO\!\left(\frac1{n^{\alpha} \log n}\right). \label{eq:in_proof_alpha}
\end{align}
Observe that the map \(x \in [1/2,3/2] \mapsto\log x\) is Lipschitz, since \(|(\log x)'|=1/x\le2\) on \([1/2,3/2]\). 
Since \(\widehat n_1/n_1\to1\) in probability, the event $\{ \frac12\le \frac{\widehat n_1}{n_1}\le \frac32 \}$ has probability tending to one. On this event, we have 
\[
    \left|
        \log\left(\frac{\widehat n_1}{n_1}\right)
    \right|
    \wle
    2\left|
        \frac{\widehat n_1}{n_1}-1
    \right|
    \weq
    2\frac{|\widehat n_1-n_1|}{n_1}
\]
with probability tending to one. Hence
\[
    \log\left(\frac{\widehat n_1}{n_1}\right)
    =
    O_{\mathbb P}\!\left(
        \frac{|\widehat n_1-n_1|}{n_1}
    \right).
\]
Consequently, going back to Equation~\eqref{eq:in_proof_alpha} we obtain 
\[
\left| \widehat\alpha-\alpha \right| 
\weq O_{\mathbb P}\!\left(\frac{pn\log n}{n^\alpha\log n}\right) + O\!\left(\frac1{n^{\alpha} \log n}\right).
\]
\end{proof}
\section{Proofs for Section~\ref{sec:estimation}}

\subsection{Additional proofs for the estimation of \texorpdfstring{$p$}{p}}
\label{app:proofs_estimation_p}

\begin{proof}[Proof of Lemma~\ref{lem:p_hat_concentration}]
Set $Y_k:=p\,k(n-k)X_k$ and $S:=\frac1M\sum_{k\in\mathcal W}Y_k$, so that we have $ \frac{ \hp_{\MLE} }{ p } =S^{-1}$.
The lemma's assumptions on the rates give
$\frac1{1+\varepsilon_+}
 \le \E[Y_k]
 \le \frac1{1-\varepsilon_-},$ 
and hence $|\E[S]-1|\le\varepsilon$.

Because $\mu_kX_k\sim\Exp(1)$, there exists an absolute constant $C_0>0$ such that
\[
 \|Y_k-\E[Y_k]\|_{\psi_1} \wle \frac{C_0}{1-\varepsilon_-}.
\]
By Bernstein's inequality for independent sub-exponential variables
\citep[Section~2.8]{vershynin2018high}, there exists an absolute
constant $C>0$ such that, for every $x\ge0$, with probability
at least $1-2e^{-x}$,
\[
 |S-\E[S]|
 \wle \frac{C}{1-\varepsilon_-} \left(\sqrt{\frac{x}{M}}+\frac{x}{M}\right)
 \weq u(x).
\]
On this event, $|S-1|\le\varepsilon+u(x)$. If $\varepsilon+u(x)<1$, then $S\ge1-\varepsilon-u(x)>0$, and therefore
\[
 \frac{|\hp_{\MLE}-p|}{p}
 \weq \left|\frac1S-1\right|
 \weq \frac{|S-1|}{S}
 \wle \frac{\varepsilon+u(x)}{1-\varepsilon-u(x)}.
\]
\end{proof}

\subsection{Proof of Proposition~\ref{prop:xi_consistency}}
\label{sec:pf_xi_consistency} 

The main difficulty in the proof is that the estimation window
$\mathcal W_\xi$ is itself random and depends on the observed
inter-arrival times through $\widehat K_1$ and $\widehat K_2$.
Thus, although conditionally on the exploration the variables
$a_k$ have mean $1/C_k$, we cannot directly use this conditional
centering after summing over the data-dependent set $\mathcal W_\xi$.

To avoid this issue, we introduce an oracle window. Let $r_n^\circ:=\left\lceil\frac{n_1}{\log n}\right\rceil$   and define
\[
    \mathcal W_\xi^\circ
    :=
    \left\{
        k:
        K_1+r_n^\circ
        \le k
        \le
        K_2^\star-r_n^\circ
    \right\}.
\]
Unlike $\mathcal W_\xi$, the endpoints of $\mathcal W_\xi^\circ$ depend only on the underlying exploration and not on the inter-arrival times.  This allows us to analyze sums over $\mathcal W_\xi^\circ$ conditionally on the exploration, while retaining the independence and exponential distribution of the $a_k$.

The proof therefore proceeds by first working on the window $\cW_\xi^\circ$ and then showing that replacing it by $\mathcal W_\xi$ has a negligible effect.
The choice $r_n^\circ$ is large compared with the localization errors of $\widehat K_1$ and $\widehat K_2$, but is still $o(n_1)$. 
Hence the two windows differ only near their endpoints and both retain $(1-o_{\pr}(1))n_1$ indices.

For this comparison, it is convenient to introduce the buffer window
\[
 \mathcal W_\xi^+
 := \left\{ k: K_1+\frac{r_n^\circ}{2} \le k \le K_2^\star-\frac{r_n^\circ}{2}
    \right\}.
\]
The following lemma establishes that, with high probability,
\[
    \mathcal W_\xi\cup\mathcal W_\xi^\circ
    \subseteq
    \mathcal W_\xi^+,
\]
so that all uniform approximations needed on $\mathcal W_\xi$ or on $\mathcal W_\xi^\circ$ can be established on $\cW_\xi^+$. 

\begin{lemma}[Comparison of the feasible and oracle windows]
\label{lem:xi_window_approx}
Under the assumptions of Theorem~\ref{thm:estimation}, we have 
$\widehat r_n
    =
    r_n^\circ+o_{\pr}(r_n^\circ),$ 
    $|\mathcal W_\xi\triangle\mathcal W_\xi^\circ| = o_{\pr}(r_n^\circ),$
and
$|\mathcal W_\xi|
    =
    (1+o_{\pr}(1))n_1.$ 
Moreover,
$\pr\left(
        \mathcal W_\xi\cup\mathcal W_\xi^\circ
        \subseteq
        \mathcal W_\xi^+
    \right)
    \weq 1 - o(1).$ 
\end{lemma}

\begin{proof}
The localization bounds give $|\widehat K_1-K_1| = \cO_{\pr}(\log n)$ and $|\widehat K_2-K_2^\star| = \cO_{\pr}(np\log n)$. Hence, 
$ |\widehat n_1-n_1| = \cO_{\pr}(np\log n+\log n).$ 
Moreover, by using $np\ll\frac{n^{2\alpha-1}}{\log n},$ 
$n_1\asymp n^\alpha,$ and $r_n^\circ\asymp\frac{n_1}{\log n},$ 
we obtain $np\log n=o(r_n^\circ)$ and $\log n=o(r_n^\circ).$
Therefore 
\[
    \widehat n_1
    =
    n_1+o_{\pr}(r_n^\circ)
    \quad \text{ and } \quad 
    \widehat r_n
    =
    r_n^\circ+o_{\pr}(r_n^\circ).
\]

The left endpoints of the intervals $\mathcal W_\xi$ and
$\mathcal W_\xi^\circ$ differ by
\[
    (\widehat K_1-K_1)
    +
    (\widehat r_n-r_n^\circ)
    =
    o_{\pr}(r_n^\circ),
\]
and the same holds for the right endpoints. As the symmetric difference of two intervals is bounded by the sum of the difference of their endpoints, we obtain 
\[
    |\mathcal W_\xi\triangle\mathcal W_\xi^\circ|
    =
    o_{\pr}(r_n^\circ).
\]

Moreover,
$|\mathcal W_\xi^\circ|
    =
    n_1-2r_n^\circ+1
    =
    (1+o(1))n_1,$ 
because $r_n^\circ=o(n_1)$. Because 
\[
    \bigl|
        |\mathcal W_\xi|-|\mathcal W_\xi^\circ|
    \bigr|
    \wle
    |\mathcal W_\xi\triangle\mathcal W_\xi^\circ|
    \weq 
    o_{\pr}(r_n^\circ)
    \weq 
    o_{\pr}(n_1),
\]
it follows that
\[
    |\mathcal W_\xi|
    =
    (1+o_{\pr}(1))n_1.
\]

Finally, $\mathcal W_\xi^\circ$ lies inside
$\mathcal W_\xi^+$ with a margin $r_n^\circ/2$ at each endpoint.
Since the corresponding endpoints of $\mathcal W_\xi$ differ from
those of $\mathcal W_\xi^\circ$ by $o_{\pr}(r_n^\circ)$,
\[
    \pr\left(
        \mathcal W_\xi\cup\mathcal W_\xi^\circ
        \subseteq\mathcal W_\xi^+
    \right)
    \weq 1-o(1).
\]
\end{proof}

For the rest of the proof, we introduce 
\[
 B_k:=pk(n-k),
 \qquad D_k:=I_1(k)(n_1-I_1(k)),
 \qquad \mu_k(q):=B_k+qD_k.
\]

The next lemma collects the uniform approximations that will be used throughout the remainder of the proof.  In particular, on the buffer window $\mathcal W_\xi^+$, the true cut rate is asymptotically dominated by the dense contribution, and the feasible quantities approximate their oracle counterparts.

\begin{lemma}[Uniform approximation on the bulk]
\label{lem:xi_bulk_approx}
Under the assumptions of Theorem~\ref{thm:estimation},
uniformly for $k\in\mathcal W_\xi^+$,
\[
    \frac{\widehat B_k}{B_k}=1+o_{\pr}(1),
    \qquad
    \frac{\widehat D_k}{D_k}=1+o_{\pr}(1),
    \qquad 
    \frac{B_k}{q_0D_k}=o_{\pr}(1),
    \qquad
    \frac{C_k}{\mu_k(q_0)}=1+o_{\pr}(1).
\]
\end{lemma}

\begin{proof}
By Corollary~\ref{cor:p_hat_rate}, $\widehat p/p=1+o_{\pr}(1)$, which immediately gives
$\frac{\widehat B_k}{B_k}=1+o_{\pr}(1)$. 
Uniformly on $\mathcal W_\xi^+$, Proposition~\ref{prop:I1_in_V1_bd} gives
\[
    I_1(k)
    =
    k-K_1+1+\cO_{\pr}(np\log n).
\]
Because the buffer window remains at distance of order $r_n^\circ$ from both endpoints and $np\log n=o(r_n^\circ)$, we have $I_1(k)\ge c r_n^\circ$ and $n_1-I_1(k)\ge c r_n^\circ$ uniformly for some $c>0$. Together with Lemma~\ref{lem:xi_window_approx}, this implies
\[
 \widehat s_k = I_1(k)+o_{\pr}(r_n^\circ)
 \quad \text{ and } \quad 
 \widehat n_1-\widehat s_k = n_1-I_1(k)+o_{\pr}(r_n^\circ)
\]
uniformly on $\cW_\xi^+$, and hence
$\sup_{ k \in \cW_\xi^+ }\left| \frac{\widehat D_k}{D_k} - 1 \right| = o_{\pr}(1)$. 

Moreover,
$D_k = I_1(k)(n_1-I_1(k)) \ge c r_n^\circ n_1$
uniformly on $\mathcal W_\xi^+$. Because throughout the dense phase
$k=\cO_{\pr}(n^\alpha)$ and $n-k\asymp n$, we have 
$B_k=\cO_{\pr}(pn^{1+\alpha}),$ 
so
\[
    \sup_{k\in\mathcal W_\xi^+}
    \frac{B_k}{q_0D_k}
    \weq 
    \cO_{\pr}\left(\frac{np}{r_n^\circ}\right)
    \weq 
    o_{\pr}(1).
\]

Finally, Proposition~\ref{prop:G_npbx_cut_bound} gives, uniformly on
$\mathcal W_\xi^+$,
\begin{equation}
\label{eq:in_proof_upper_bound_Ckbulk}
    |C_k-\mu_k(q_0)|
    \le
    \Delta_k^{\rm in}
    + U_k
    + p_n\bigl(k(n-k)-D_k\bigr).
\end{equation}
Since $\mu_k(q_0)=B_k+q_0D_k\ge q_0D_k$, it is enough to divide the
right-hand side by $q_0D_k$. Using the bounds established above,
$D_k\ge c r_n^\circ n_1$, $M_k\le n_1$, $n-k\asymp n$, and the fact
that $q_0=\xi-p_n$ is bounded away from zero, we obtain uniformly on
$\mathcal W_\xi^+$,
\[
 \frac{\Delta_k^{\rm in}}{q_0D_k}=o(1),
 \quad
 \frac{U_k}{q_0D_k}=o_{\pr}(1),
 \quad \text{ and } \quad 
 \frac{p_n(k(n-k)-D_k)}{q_0D_k} \wle \frac{B_k}{q_0D_k} \weq o_{\pr}(1).
\]
We conclude that
$\sup_{k\in\mathcal W_\xi^+} \left| \frac{C_k}{\mu_k(q_0)}-1 \right| = o_{\pr}(1).$
\end{proof}

We now analyze the likelihood at the true value $q_0$. 
We first work on the oracle window $\mathcal W_\xi^\circ$, where conditioning on the exploration preserves the independent exponential law of the inter-arrival times.  
We will then transfer the resulting bound to the feasible likelihood based on $\mathcal W_\xi$.

Define $\mathcal J:=\sigma\bigl(G,V_1,(u_k)_{k\ge1}\bigr)$, that is, $\mathcal J$ contains the graph and the order in which vertices are explored, but not the inter-arrival times. 
The quantities $C_k$, $D_k$, and $\mathcal W_\xi^\circ$ are $\mathcal J$-measurable. Conditionally on $\mathcal J$, the inter-arrival times are independent and satisfy $a_k\sim\Exp(C_k)$. 

We first analyze the oracle score, that is, the derivative with respect to $q$ of the approximate log-likelihood restricted to $\mathcal W_\xi^\circ$, evaluated at $q_0$.

\begin{lemma}%[Oracle score]
\label{lem:xi_oracle_score}
Let $S_n^\circ(q_0) := \sum_{k\in\mathcal W_\xi^\circ} D_k \left( \frac{1}{\mu_k(q_0)}-a_k \right).$ 
Then $S_n^\circ(q_0)=o_{\pr}(n_1).$ 
\end{lemma}

\begin{proof}
Write $S_n^\circ(q_0)=A_n^\circ+M_n^\circ,$
where
$A_n^\circ := \sum_{k\in\mathcal W_\xi^\circ} D_k \left( \frac{1}{\mu_k(q_0)}-\frac{1}{C_k} \right)$ 
and
$M_n^\circ := \sum_{k\in\mathcal W_\xi^\circ} D_k \left( \frac{1}{C_k}-a_k \right).$

We first control the term $A_n^\circ$, caused by replacing the true rate $C_k$ by $\mu_k(q_0)$. 
By Lemma~\ref{lem:xi_bulk_approx}, uniformly on $\mathcal W_\xi^\circ$,
$\frac{C_k}{\mu_k(q_0)}=1+o_{\pr}(1).$ 
Moreover,
\[
    \frac{D_k}{\mu_k(q_0)}
    \wle \frac{1}{q_0}
    \weq \cO(1),
\]
because $q_0=\xi-p\to\xi>0$. 
Hence
\begin{align}
\label{eq:in_proof_term_An}
 |A_n^\circ|
 \wle |\mathcal W_\xi^\circ| \sup_{k\in\mathcal W_\xi^\circ} \frac{D_k}{\mu_k(q_0)} \sup_{k\in\mathcal W_\xi^\circ} \frac{|C_k-\mu_k(q_0)|}{C_k} 
 \weq o_{\pr}(n_1).
\end{align}

We next control the random fluctuations.  Conditionally on $\mathcal J$, $\E[M_n^\circ\mid\mathcal J]=0$ 
and, by conditional independence,
\[
    \Var(M_n^\circ\mid\mathcal J)
    =
    \sum_{k\in\mathcal W_\xi^\circ}
    \frac{D_k^2}{C_k^2}.
\]
By Lemma~\ref{lem:xi_bulk_approx}, $\mu_k(q_0) = q_0D_k(1+o_{\pr}(1)),$ uniformly on $\mathcal W_\xi^\circ$, 
because $B_k/(q_0D_k)=o_{\pr}(1)$. 
Together with $C_k/\mu_k(q_0)=1+o_{\pr}(1)$, this gives $C_k=q_0D_k(1+o_{\pr}(1))$ uniformly, and thus 
\[
 \Var(M_n^\circ\mid\mathcal J)
 \weq \frac{|\mathcal W_\xi^\circ|}{q_0^2} (1+o_{\pr}(1))
 \weq \cO_{\pr}(n_1).
\]
Conditional Chebyshev's inequality then yields
\begin{align}
 \label{eq:in_proof_term_Mn}
 M_n^\circ
 \weq \cO_{\pr}(\sqrt{n_1})
 \weq o_{\pr}(n_1).
\end{align}
We finish the proof by combining~\eqref{eq:in_proof_term_An} and~\eqref{eq:in_proof_term_Mn}.
\end{proof}

\begin{proof}[Proof of Proposition~\ref{prop:xi_consistency}]
Recall that the feasible approximate log-likelihood is
\[
    \widehat\ell_n(q)
    =
    \sum_{k\in\mathcal W_\xi}
    \left(
        \log\widehat\mu_k(q)
        -
        a_k\widehat\mu_k(q)
    \right),
\]
and write its derivative as
\[
    \widehat S_n(q)
    :=
    \widehat\ell_n'(q)
    =
    \sum_{k\in\mathcal W_\xi}
    \widehat D_k
    \left(
        \frac{1}{\widehat\mu_k(q)}-a_k
    \right).
\]

\paragraph*{Step 1: Transfer the oracle bound to the feasible likelihood}
We first show that
\begin{equation}
\label{eq:xi_score_small}
    \widehat S_n(q_0)=o_{\pr}(n_1).
\end{equation}
On the common part
$\mathcal W_\xi\cap\mathcal W_\xi^\circ$,
Lemma~\ref{lem:xi_bulk_approx} gives $\frac{\widehat D_k}{D_k}=1+o_{\pr}(1)$ and $\frac{\widehat\mu_k(q_0)}{\mu_k(q_0)} =1+o_{\pr}(1)$, uniformly in $k$. Thus, 
$\frac{\widehat D_k}{\widehat\mu_k(q_0)} - \frac{D_k}{\mu_k(q_0)}
 = o_{\pr}(1).
$
Therefore the difference between the feasible and oracle scores on the
common part is bounded by
\begin{align*}
&|\mathcal W_\xi|
\sup_{k\in\mathcal W_\xi^+}
\left|
    \frac{\widehat D_k}{\widehat\mu_k(q_0)}
    -
    \frac{D_k}{\mu_k(q_0)}
\right| 
+
\sup_{k\in\mathcal W_\xi^+} \left| \frac{\widehat D_k}{D_k}-1 \right|
\sum_{k\in\mathcal W_\xi^+}a_kD_k .
\end{align*}
The first term is $o_{\pr}(n_1)$.  For the second, conditionally on
$\mathcal J$,
\[
    \E\left[
        \sum_{k\in\mathcal W_\xi^+}a_kD_k
        \,\middle|\,
        \mathcal J
    \right]
    =
    \sum_{k\in\mathcal W_\xi^+}\frac{D_k}{C_k}
    =
    \cO_{\pr}(n_1),
\]
and its conditional variance is also $\cO_{\pr}(n_1)$.
Thus
\[
    \sum_{k\in\mathcal W_\xi^+}a_kD_k
    =
    \cO_{\pr}(n_1),
\]
and the contribution from the common part is $o_{\pr}(n_1)$.

It remains to control the indices that belong to only one of the two
windows.  By Lemma~\ref{lem:xi_window_approx},
$|\mathcal W_\xi\triangle\mathcal W_\xi^\circ| = o_{\pr}(r_n^\circ),$
and, with probability tending to one, all these indices belong to
$\mathcal W_\xi^+$.  Moreover, Lemma~\ref{lem:xi_bulk_approx} implies
uniformly on $\mathcal W_\xi^+$ that
\[
    \max\left\{ 
    \frac{D_k}{C_k},
    \frac{\widehat D_k}{C_k},
    \frac{D_k}{\mu_k(q_0)},
    \frac{\widehat D_k}{\widehat\mu_k(q_0)}
    \right\}
    \weq \cO_{\pr}(1).
\]
Conditionally on $\mathcal J$, the variables
$\{C_ka_k:k\in\mathcal W_\xi^+\}$ are independent $\Exp(1)$, so
$\max_{k\in\mathcal W_\xi^+}C_ka_k
    =
    \cO_{\pr}(\log n).
$ 
Hence each score contribution on the symmetric difference is $\cO_{\pr}(\log n)$ uniformly, and their total contribution is $o_{\pr}(r_n^\circ)\cO_{\pr}(\log n) = o_{\pr}(n_1)$,
because $r_n^\circ\log n=(1+o(1))n_1$.
Together with Lemma~\ref{lem:xi_oracle_score}, this proves
\eqref{eq:xi_score_small}.

\paragraph*{Step 2: Control the curvature of the likelihood}
We have
\[
    -\widehat\ell_n''(q)
    =
    \sum_{k\in\mathcal W_\xi}
    \frac{\widehat D_k^2}{\widehat\mu_k(q)^2}
    =
    \sum_{k\in\mathcal W_\xi}
    \frac{1}{
        \left(q+\widehat B_k/\widehat D_k\right)^2
    }.
\]
By Lemma~\ref{lem:xi_bulk_approx},
$\sup_{k\in\mathcal W_\xi}
    \frac{\widehat B_k}{\widehat D_k}
    =
    o_{\pr}(1).$
Because $q_0\to\xi>0$ and
$|\mathcal W_\xi|=(1+o_{\pr}(1))n_1$, for every sufficiently small fixed
$\varepsilon>0$ there exists $c_\varepsilon>0$ such that
\begin{equation}
\label{eq:xi_curvature}
    \pr\left(
        \inf_{|q-q_0|\le\varepsilon}
        \bigl\{-\widehat\ell_n''(q)\bigr\}
        \ge c_\varepsilon n_1
    \right)
    \longrightarrow1.
\end{equation}
Thus moving $q$ a fixed distance away from $q_0$ changes the derivative
of the likelihood by an amount of order $n_1$, whereas
\eqref{eq:xi_score_small} shows that the derivative at $q_0$ is
$o_{\pr}(n_1)$.

\paragraph*{Step 3: Localize the maximizer.}
Suppose first that $\xi<1$.  Fix $\varepsilon>0$ small enough that,
with probability tending to one,
$q_0-\varepsilon$ and $q_0+\varepsilon$ both belong to
$[0,1-\widehat p]$.  Since
$\widehat S_n'(q)=\widehat\ell_n''(q)$,
\eqref{eq:xi_score_small} and \eqref{eq:xi_curvature} imply
\begin{align}
    \widehat S_n(q_0-\varepsilon)>0
    \quad \text{ and } \quad 
    \widehat S_n(q_0+\varepsilon)<0,
 \label{eq:in_proof_score}
\end{align}
with probability tending to one. 
As $\widehat\ell_n$ is concave, its score $\widehat S_n$ is nonincreasing. Hence, the inequalities in ~\eqref{eq:in_proof_score} imply that every maximizer lies in $[q_0-\varepsilon,q_0+\varepsilon]$.

If $\xi=1$, the same argument at $q_0-\varepsilon$ gives $\widehat q\ge q_0-\varepsilon$ with probability tending to one. On the other hand, by the parameter constraint,
\[
 \widehat q
 \wle 1-\widehat p
 \weq q_0+(p-\widehat p)
 \weq q_0+o_{\pr}(1).
\]
Thus, in both cases, $\widehat q-q_0=o_{\pr}(1).$ 
Finally, because $\widehat p-p=o_{\pr}(1)$, we conclude that 
\[
 \widehat\xi-\xi
 \weq (\widehat p-p)+(\widehat q-q_0)
 \weq o_{\pr}(1). 
\]
\end{proof}

\section{Sampling the Cut Size Process}
\label{app:cut_size_sampler}

For the numerical experiments in Section~\ref{sec:numerical_experiments}, it is useful to sample the cut-size process without first generating the entire graph. The following algorithm has the same distribution as the cut-size sequence of the first-passage exploration on \(G(n,p)\).  

\begin{algorithm}[!ht]
\caption{\textsc{SampleCutSizes}}
\label{algo:cut_sizes}
\KwInput{number of vertices $n$, edge probability $p$}

Let $H(1,1) \sim \text{Bin}(n-1, p)$, $H(0,1)=n-1 - H(1,1)$,  $H(i,1)= 0 \ \ \forall i \neq 0,1$ \label{ln:scs_init}

Let $C \gets [H(1,1)]$

\For{$k \in \{2, \cdots, n\}$ }
{   
    \If {\(\sum_j jH(j,k-1)=0\)}
    {append zeros to \(C\) until length \(n\)
    
    \KwReturn{C}}
    
    $\forall i: H(i,k) \gets H(i,k-1)$ \label{ln:scs_copy_over}
    
    $\text{deg}_u \gets \text{random choice from } \{i: H(i,k-1)>0\} \text{ with probability } p_i = \frac{iH(i,k-1)}{\sum_j jH(j,k-1)}$ \label{ln:scs_sample_u}

    $H(\text{deg}_u, k) \gets H(\text{deg}_u, k-1)-1$ \label{ln:scs_rmv_u}

    \For{$j \in [j: H(j,k)>0, \text{ descending order}]$}
    {
        $r \sim \text{Bin}(H(j,k), p)$  \label{ln:scs_sample_promotion}
        
        $H(j,k) \gets H(j,k) - r$

        $H(j+1,k) \gets H(j+1,k) + r$
    }

    $C$ append new value $\sum_j jH(j,k)$ \label{ln:scs_calc_Ck}

}
\KwReturn{C}
\end{algorithm}

\begin{lemma}%[Algorithm \ref{algo:cut_sizes} Correctness] 
\label{lm:algo_correctness}
 The cut size sequence $\left(C_k\right)_{k \in [n]}$ produced by \textsc{SampleCutSizes}$(n,p)$ has the same distribution as the sequence of cut sizes in a $\gnp$.
\end{lemma}

\begin{proof}
We prove the claim by coupling the algorithm with the FPP exploration of
\(G(n,p)\), using the deferred-decisions principle.

Let the FPP process start from a fixed vertex \(s\). For an infected set
\(\mathcal I_k\), define the cut-degree of an inactive vertex \(v\in \Ikc\) by
$d_{\mathcal I_k}(v) = |\{u\in \mathcal I_k:(u,v)\in E\}|.$ 
The corresponding cut-degree histogram is $H_{\Ik}(i) = |\{v\in \Ikc:d_{\mathcal I_k}(v)=i\}|$ and the cut size is $C_k=\sum_i iH_{\Ik}(i).$

We construct the graph exploration lazily. Initially, only the edges from the
source \(s\) to the other \(n-1\) vertices are revealed. Each such edge is
present independently with probability~\(p\). Therefore
$H_{\{s\}}(1)\sim \operatorname{Bin}(n-1,p),$ $H_{\{s\}}(0)=n-1-H_{\{s\}}(1),$ 
and \(H_{\{s\}}(i)=0\) for \(i\notin\{0,1\}\). This agrees with Line
\ref{ln:scs_init}.

Now suppose the current infected set is \(\mathcal I_k\), and write
$H(i)=H_{\mathcal I_k}(i).$ 
Let $C=\sum_i iH(i)$  be the current cut size. If \(C=0\), the connected component of the source has
been exhausted, and the cut size remains zero thereafter. Thus assume \(C>0\).

Conditionally on the currently revealed graph, the exponential clocks on the
cut edges are independent and have the same rate. By the memorylessness of the
exponential distribution, the next infection edge is therefore uniformly
distributed among the \(C\) cut edges. Hence, the probability that the newly
infected vertex has current cut-degree \(i\) is
\[
\frac{iH(i)}{\sum_j jH(j)}.
\]
This is exactly the sampling rule in Line \ref{ln:scs_sample_u}. After this
vertex is infected, it is removed from the inactive side, which corresponds to
subtracting one from the appropriate histogram bin, as in Line
\ref{ln:scs_rmv_u}.

It remains to update the cut-degrees of the remaining inactive vertices. Let
\(u\) be the newly infected vertex. For any remaining inactive vertex \(w\), the
edge \((u,w)\) was previously an edge between two inactive vertices. Hence it was
not revealed by the previous exploration. Since the underlying graph is
\(G(n,p)\), these edges are still conditionally independent Bernoulli\((p)\)
random variables.

Therefore, among the vertices that currently have cut-degree \(j\), the number
that are adjacent to \(u\) is distributed as $R_j\sim \operatorname{Bin}(H(j),p).$ 
Those \(R_j\) vertices move from bin \(j\) to bin \(j+1\). This is exactly the
promotion step in Lines \ref{ln:scs_sample_promotion}--\ref{ln:scs_calc_Ck}.
The descending order ensures that a vertex promoted from \(j\) to \(j+1\) is
not considered again in the same update.

Thus, the graph exploration and the algorithm can be coupled so that, at
every step, the histogram stored by the algorithm is exactly the cut-degree
histogram of the FPP exploration. Consequently, the cut size $C$ returned by the algorithm
has the same distribution as the cut size in the FPP exploration of \(G(n,p)\).
Therefore, the sequence returned by \textsc{SampleCutSizes} has the same
distribution as the true cut-size sequence.
\end{proof}

%% file: biblio.bib
@article{prokhorenkova2022less,
  title={When less is more: Systematic analysis of cascade-based community detection},
  author={Prokhorenkova, Liudmila and Tikhonov, Alexey and Litvak, Nelly},
  journal={ACM Transactions on Knowledge Discovery from Data (TKDD)},
  volume={16},
  number={4},
  pages={1--22},
  year={2022},
  publisher={ACM New York, NY}
}

@inproceedings{mossel_high_deg_2024,
	title = {Finding {Super}-spreaders in {Network} {Cascades}},
	url = {https://proceedings.mlr.press/v247/mossel24a.html},
	language = {en},
	urldate = {2025-03-11},
	booktitle = {Proceedings of {Thirty} {Seventh} {Conference} on {Learning} {Theory}},
	publisher = {PMLR},
	author = {Mossel, Elchanan and Sridhar, Anirudh},
	month = {06},
	year = {2024},
	note = {ISSN: 2640-3498},
	pages = {3874--3914},
}

@article{brandenberger_detecting_change_points_2025,
  title={Detecting {Abrupt} {Changes} in {Point} {Processes}: {Fundamental} {Limits} and {Applications}},
  author={Brandenberger, Anna and Mossel, Elchanan and Sridhar, Anirudh},
  journal={arXiv preprint arXiv:2501.08392},
  year={2025},
  url = {http://arxiv.org/abs/2501.08392},
  doi = {10.48550/arXiv.2501.08392}
}

@article{janson_one_1999,
    title = {One, {Two} and {Three} {Times} $\log n/n$ for {Paths} in a {Complete} {Graph} with {Random} {Weights}},
    volume = {8},
    issn = {1469-2163, 0963-5483},
    language = {en},
    number = {4},
    urldate = {2023-10-26},
    journal = {Combinatorics, Probability and Computing},
    author = {Janson, Svante},
    month = {07},
    year = {1999},
    pages = {347--361},
    doi={10.1017/S0963548399003892}
}

@book{vershynin2018high,
  title={High-dimensional probability: An introduction with applications in data science},
  author={Vershynin, Roman},
  volume={47},
  year={2018},
  publisher={Cambridge University Press},
  doi={10.1017/9781108231596}
}

@article{netrapalli_learning_2012,
    title = {Learning the graph of epidemic cascades},
    volume = {40},
    issn = {0163-5999},
    url = {https://doi.org/10.1145/2318857.2254783},
    doi = {10.1145/2318857.2254783},
    number = {1},
    urldate = {2023-02-20},
    journal = {ACM SIGMETRICS Performance Evaluation Review},
    author = {Netrapalli, Praneeth and Sanghavi, Sujay},
    month = {06},
    year = {2012},
    pages = {211--222},
}

@article{gray_bayesian_2020,
    title = {Bayesian {Inference} of {Network} {Structure} {From} {Information} {Cascades}},
    volume = {6},
    issn = {2373-776X},
    doi = {10.1109/TSIPN.2020.2990276},
    journal = {IEEE Transactions on Signal and Information Processing over Networks},
    author = {Gray, Caitlin and Mitchell, Lewis and Roughan, Matthew},
    year = {2020},
    pages = {371--381},
}

@inproceedings{abrahao_trace_2013,
    address = {New York, NY, USA},
    series = {{KDD} '13},
    title = {Trace complexity of network inference},
    isbn = {978-1-4503-2174-7},
    url = {https://dl.acm.org/doi/10.1145/2487575.2487664},
    doi = {10.1145/2487575.2487664},
    urldate = {2023-07-04},
    booktitle = {Proceedings of the 19th {ACM} {SIGKDD} international conference on {Knowledge} discovery and data mining},
    publisher = {Association for Computing Machinery},
    author = {Abrahao, Bruno and Chierichetti, Flavio and Kleinberg, Robert and Panconesi, Alessandro},
    month = {08},
    year = {2013},
    pages = {491--499},
}

@inproceedings{wilinski_prediction-centric_2021,
    title = {Prediction-{Centric} {Learning} of {Independent} {Cascade} {Dynamics} from {Partial} {Observations}},
    issn = {2640-3498},
    url = {https://proceedings.mlr.press/v139/wilinski21a.html},
    language = {en},
    urldate = {2023-05-20},
    booktitle = {Proceedings of the 38th {International} {Conference} on {Machine} {Learning}},
    publisher = {PMLR},
    author = {Wilinski, Mateusz and Lokhov, Andrey},
    month = {07},
    year = {2021},
    pages = {11182--11192},
}

@article{stargel_case_2022,
    title = {Case {Investigation} and {Contact} {Tracing} {Efforts} {From} {Health} {Departments} in the {United} {States}, {November} 2020 to {December} 2021},
    volume = {75},
    issn = {1058-4838},
    url = {https://doi.org/10.1093/cid/ciac442},
    doi = {10.1093/cid/ciac442},
    number = {Supplement\_2},
    urldate = {2026-04-01},
    journal = {Clinical Infectious Diseases},
    author = {Stargel, Alison and Taylor, Melanie M and Zansky, Shelley and Spencer, Kimberly and Hogben, Matthew and Shultz, Alvin},
    month = {10},
    year = {2022},
    pages = {S326--S333},
}

@inproceedings{gomez_rodriguez_inferring_2010,
    address = {New York, NY, USA},
    series = {{KDD} '10},
    title = {Inferring networks of diffusion and influence},
    isbn = {978-1-4503-0055-1},
    url = {https://dl.acm.org/doi/10.1145/1835804.1835933},
    doi = {10.1145/1835804.1835933},
    urldate = {2026-04-15},
    booktitle = {Proceedings of the 16th {ACM} {SIGKDD} international conference on {Knowledge} discovery and data mining},
    publisher = {Association for Computing Machinery},
    author = {Gomez Rodriguez, Manuel and Leskovec, Jure and Krause, Andreas},
    month = {07},
    year = {2010},
    pages = {1019--1028},
}

@techreport{matula_largest_1976,
    address = {Dallas, Texas},
    title = {The {Largest} {Clique} {Size} in a {Random} {Graph}},
    institution = {Southern Methodist University},
    author = {Matula, David},
    month = {04},
    year = {1976},
}

@article{grimmett_colouring_1975,
    title = {On colouring random graphs},
    volume = {77},
    issn = {1469-8064, 0305-0041},
    url = {https://www.cambridge.org/core/journals/mathematical-proceedings-of-the-cambridge-philosophical-society/article/on-colouring-random-graphs/38719DFD65620CD94A379293A3E50B7F},
    doi = {10.1017/S0305004100051124},
    language = {en},
    number = {2},
    urldate = {2026-04-17},
    journal = {Mathematical Proceedings of the Cambridge Philosophical Society},
    author = {Grimmett, G. R. and McDiarmid, C. J. H.},
    month = {03},
    year = {1975},
    pages = {313--324},
}

@article{kucera_expected_1995,
    series = {Combinatorial optimization 1992},
    title = {Expected complexity of graph partitioning problems},
    volume = {57},
    issn = {0166-218X},
    url = {https://www.sciencedirect.com/science/article/pii/0166218X9400103K},
    doi = {10.1016/0166-218X(94)00103-K},
    number = {2},
    urldate = {2026-04-20},
    journal = {Discrete Applied Mathematics},
    author = {Kučera, Luděk},
    month = {02},
    year = {1995},
    pages = {193--212},
}

@inproceedings{alon_finding_1998,
    address = {USA},
    series = {{SODA} '98},
    title = {Finding a large hidden clique in a random graph},
    isbn = {978-0-89871-410-4},
    url = {https://dl.acm.org/doi/10.5555/314613.315014},
    urldate = {2026-04-17},
    booktitle = {Proceedings of the ninth annual {ACM}-{SIAM} symposium on {Discrete} algorithms},
    publisher = {Society for Industrial and Applied Mathematics},
    author = {Alon, Noga and Krivelevich, Michael and Sudakov, Benny},
    month = {01},
    year = {1998},
    pages = {594--598},
}

@article{feige_finding_2000,
    title = {Finding and certifying a large hidden clique in a semirandom graph},
    volume = {16},
    copyright = {Copyright © 2000 John Wiley \& Sons, Inc.},
    issn = {1098-2418},
    url = {https://onlinelibrary.wiley.com/doi/abs/10.1002/%28SICI%291098-2418%28200003%2916%3A2%3C195%3A%3AAID-RSA5%3E3.0.CO%3B2-A},
    doi = {10.1002/(SICI)1098-2418(200003)16:2<195::AID-RSA5>3.0.CO;2-A},
    language = {en},
    number = {2},
    urldate = {2026-04-20},
    journal = {Random Structures \& Algorithms},
    author = {Feige, Uriel and Krauthgamer, Robert},
    year = {2000},
    pages = {195--208},
}

@article{dekel_finding_2014,
    title = {Finding {Hidden} {Cliques} in {Linear} {Time} with {High} {Probability}},
    volume = {23},
    issn = {0963-5483, 1469-2163},
    url = {https://www.cambridge.org/core/journals/combinatorics-probability-and-computing/article/finding-hidden-cliques-in-linear-time-with-high-probability/88A05D6B07878D54E62433DA72190623},
    doi = {10.1017/S096354831300045X},
    language = {en},
    number = {1},
    urldate = {2026-04-20},
    journal = {Combinatorics, Probability and Computing},
    author = {Dekel, Yael and Gurel-Gurevich, Ori and Peres, Yuval},
    month = {01},
    year = {2014},
    pages = {29--49},
}

@inproceedings{feige_finding_2010,
    title = {Finding hidden cliques in linear time},
    booktitle = {DMTCS Proceedings vol. AM, 21st International Meeting on Probabilistic, Combinatorial, and Asymptotic Methods in the Analysis of Algorithms (AofA'10)},
    issn = {1365-8050},
    url = {https://dmtcs.episciences.org/2802},
    pages = {189-204},
    doi = {10.46298/dmtcs.2802},
    journal = {Discrete Mathematics \& Theoretical Computer Science},
    author = {Feige, Uriel and Ron, Dorit},
    year = {2010},
}

@article{deshpande_finding_2015,
    title = {Finding {Hidden} {Cliques} of {Size} $\sqrt{N/e}$ in {Nearly} {Linear} {Time}},
    volume = {15},
    issn = {1615-3375, 1615-3383},
    url = {http://link.springer.com/10.1007/s10208-014-9215-y},
    doi = {10.1007/s10208-014-9215-y},
    language = {en},
    number = {4},
    urldate = {2026-04-20},
    journal = {Foundations of Computational Mathematics},
    author = {Deshpande, Yash and Montanari, Andrea},
    month = {08},
    year = {2015},
    pages = {1069--1128},
}

@article{gamarnik_landscape_2024,
    title = {The landscape of the planted clique problem: {Dense} subgraphs and the overlap gap property},
    volume = {34},
    issn = {1050-5164, 2168-8737},
    shorttitle = {The landscape of the planted clique problem},
    url = {https://projecteuclid.org/journals/annals-of-applied-probability/volume-34/issue-4/The-landscape-of-the-planted-clique-problem--Dense-subgraphs/10.1214/23-AAP2003.full},
    doi = {10.1214/23-AAP2003},
    language = {en},
    number = {4},
    urldate = {2026-04-17},
    journal = {The Annals of Applied Probability},
    publisher = {Institute of Mathematical Statistics},
    author = {Gamarnik, David and Zadik, Ilias},
    month = {08},
    year = {2024},
    pages = {3375--3434},
}

@article{arias-castro_community_2014,
    author = {Ery Arias-Castro and Nicolas Verzelen},
    title = {{Community detection in dense random networks}},
    volume = {42},
    journal = {The Annals of Statistics},
    number = {3},
    publisher = {Institute of Mathematical Statistics},
    pages = {940 -- 969},
    year = {2014},
    doi = {10.1214/14-AOS1208},
    URL = {https://doi.org/10.1214/14-AOS1208}
}

@article{verzelen_community_2015,
    title = {Community detection in sparse random networks},
    volume = {25},
    issn = {1050-5164, 2168-8737},
    url = {https://projecteuclid.org/journals/annals-of-applied-probability/volume-25/issue-6/Community-detection-in-sparse-random-networks/10.1214/14-AAP1080.full},
    doi = {10.1214/14-AAP1080},
    language = {en},
    number = {6},
    urldate = {2026-04-20},
    journal = {The Annals of Applied Probability},
    publisher = {Institute of Mathematical Statistics},
    author = {Verzelen, Nicolas and Arias-Castro, Ery},
    month = {12},
    year = {2015},
    pages = {3465--3510},
}

@inproceedings{hajek_computational_2015,
    title = {Computational {Lower} {Bounds} for {Community} {Detection} on {Random} {Graphs}},
    issn = {1938-7228},
    url = {https://proceedings.mlr.press/v40/Hajek15.html},
    language = {en},
    urldate = {2026-04-20},
    booktitle = {Proceedings of {The} 28th {Conference} on {Learning} {Theory}},
    publisher = {PMLR},
    author = {Hajek, Bruce and Wu, Yihong and Xu, Jiaming},
    month = {06},
    year = {2015},
    pages = {899--928},
}

@inproceedings{hajek_information_2016,
    title = {Information limits for recovering a hidden community},
    issn = {2157-8117},
    url = {https://ieeexplore.ieee.org/document/7541628},
    doi = {10.1109/ISIT.2016.7541628},
    urldate = {2026-04-20},
    booktitle = {2016 {IEEE} {International} {Symposium} on {Information} {Theory} ({ISIT})},
    author = {Hajek, Bruce and Wu, Yihong and Xu, Jiaming},
    month = {07},
    year = {2016},
    note = {ISSN: 2157-8117},
    pages = {1894--1898},
}

@article{montanari_finding_2015,
    title = {Finding {One} {Community} in a {Sparse} {Graph}},
    volume = {161},
    issn = {1572-9613},
    url = {https://doi.org/10.1007/s10955-015-1338-2},
    doi = {10.1007/s10955-015-1338-2},
    language = {en},
    number = {2},
    urldate = {2026-04-20},
    journal = {Journal of Statistical Physics},
    author = {Montanari, Andrea},
    month = {10},
    year = {2015},
    pages = {273--299},
}

@inproceedings{brennan_reducibility_2018,
    title = {Reducibility and {Computational} {Lower} {Bounds} for {Problems} with {Planted} {Sparse} {Structure}},
    issn = {2640-3498},
    url = {https://proceedings.mlr.press/v75/brennan18a.html},
    language = {en},
    urldate = {2026-04-21},
    booktitle = {Proceedings of the 31st  {Conference} {On} {Learning} {Theory}},
    publisher = {PMLR},
    author = {Brennan, Matthew and Bresler, Guy and Huleihel, Wasim},
    month = {06},
    year = {2018},
    pages = {48--166},
}

@inproceedings{brennan_universality_2019,
    title = {Universality of {Computational} {Lower} {Bounds} for {Submatrix} {Detection}},
    issn = {2640-3498},
    url = {https://proceedings.mlr.press/v99/brennan19a.html},
    language = {en},
    urldate = {2026-04-21},
    booktitle = {Proceedings of the {Thirty}-{Second} {Conference} on {Learning} {Theory}},
    publisher = {PMLR},
    author = {Brennan, Matthew and Bresler, Guy and Huleihel, Wasim},
    month = {06},
    year = {2019},
    pages = {417--468},
}

@article{bhamidi_first_2010,
    title = {First passage percolation on random graphs with finite mean degrees},
    volume = {20},
    issn = {1050-5164},
    pages = {1907 -- 1965},
    url = {https://projecteuclid.org/journals/annals-of-applied-probability/volume-20/issue-5/First-passage-percolation-on-random-graphs-with-finite-mean-degrees/10.1214/09-AAP666.full},
    doi = {10.1214/09-AAP666},
    number = {5},
    urldate = {2025-11-03},
    journal = {The Annals of Applied Probability},
    author = {Bhamidi, Shankar and Van Der Hofstad, Remco and Hooghiemstra, Gerard},
    month = {10},
    year = {2010},
}

@article{barbour_approximating_2013,
author = {Andrew Barbour and Gesine Reinert},
title = {{Approximating the epidemic curve}},
volume = {18},
journal = {Electronic Journal of Probability},
publisher = {Institute of Mathematical Statistics and Bernoulli Society},
pages = {1 -- 30},
year = {2013},
doi = {10.1214/EJP.v18-2557},
URL = {https://doi.org/10.1214/EJP.v18-2557}
}

@article{bhamidi_front_2014,
    title = {The front of the epidemic spread and first passage percolation},
    volume = {51},
    issn = {0021-9002, 1475-6072},
    url = {http://arxiv.org/abs/1401.0914},
    doi = {10.1239/jap/1417528470},
    language = {en},
    number = {A},
    urldate = {2026-04-17},
    journal = {Journal of Applied Probability},
    author = {Bhamidi, Shankar and Hofstad, Remco van der and Komjathy, Julia},
    month = {12},
    year = {2014},
    pages = {101--121},
}
